\documentclass[final,1p]{elsarticle}
\usepackage{amssymb,latexsym,amsmath}
\usepackage{amsthm}
\usepackage{algorithm} 
\usepackage{color,caption,cases}
\usepackage{hyperref}
\usepackage{bm}
\usepackage{graphicx,graphics}
\usepackage{subfigure}
\usepackage{algpseudocode}
\usepackage{booktabs}
\usepackage{tabularx}
\usepackage{stmaryrd}
\usepackage[T1]{fontenc}
\usepackage{makecell}
\usepackage{multirow}
\biboptions{numbers,sort&compress}
\usepackage{chemfig}
\usepackage{float}
\allowdisplaybreaks

\newtheorem{theorem}{Theorem}[section]
\newtheorem{lemma}[theorem]{Lemma}
\newtheorem{corollary}[theorem]{Corollary}

\theoremstyle{definition}

\newtheorem{example}[theorem]{Example}

\theoremstyle{remark}
\newtheorem{remark}{Remark}
\numberwithin{equation}{section}

\newcommand{\f}{\frac}
\newcommand{\df}{\dfrac}

\newcommand{\p}{\partial}

\newcommand{\mal}{\mathcal}
\newcommand{\mbb}{\mathbb}

\newcommand{\de}{\delta}

\newcommand{\Tuv}{\mal{T}_{\mal{U} \rightarrow \mal{V}}}
\newcommand{\Tvu}{\mal{T}_{\mal{V} \rightarrow \mal{U}}}
\journal{}
\begin{document}
\begin{frontmatter}
    \title{Bound-preserving and mass-conserving projection BCFD methods for incompressible miscible displacement in porous media}		
	\author[ouc]{Xiaoying Wang} \ead{wxy7121@stu.ouc.edu.cn}
	\author[ouc,lab]{Shusen Xie} \ead{shusenxie@ouc.edu.cn}
	\author[ouc,lab]{Hongfei Fu\corref{Fu}}\ead{fhf@ouc.edu.cn}
		
	\address[ouc]{School of Mathematical Sciences, Ocean University of China, Qingdao, Shandong 266100, China}
	\address[lab]{Laboratory of Marine Mathematics, Ocean University of China, Qingdao, Shandong 266100, China}
	\cortext[Fu]{Corresponding author.}		
        
\begin{abstract}
In this paper, we propose and analyze two novel structure-preserving projection schemes for the incompressible miscible displacement problem in porous media. These schemes are specifically designed to preserve two essential physical properties--concentration bounds and the mass balance identity, and are accompanied by a rigorous error analysis for the projection approach. Both schemes are second-order accurate in space but differ in temporal accuracy: one is first-order and the other is second-order. Methodologically, both schemes comprise a prediction step for the concentration variable, based on a second-order modified upwind discretization together with an implicit-explicit block-centered finite difference (BCFD) approximation, followed by a projection (or correction) step that enforces the desired physical constraints, namely bounds and mass conservation. For the first-order temporal scheme, the velocity and pressure are  solved at the first step on staggered spatial grids with explicit treatment of concentration. In contrast, for the second-order scheme, to ensure well-posedness and to suppress non-physical oscillations in the simulation, the velocity and pressure are solved at the final step with the concentration treated implicitly. A rigorous numerical theoretical framework is established for the first-order scheme, encompassing optimal-order error estimates, stability, and proofs of the existence and uniqueness of the numerical solutions. The numerical analysis framework is further extended to the more involved second-order scheme, and the corresponding theoretical results are briefly outlined. Finally, numerical examples, including some ’real’ simulations, are presented to verify the theoretical findings and demonstrate the strong performance of the proposed schemes. In particular, the occurrence of viscous fingering is successfully captured and confirmed by the numerical results.
\end{abstract}


\end{frontmatter}
	
\section{Introduction}\label{sec:intro}
Miscible displacement, the process of one fluid displacing another in a porous medium, is of great significance in oil recovery and contaminant transport \cite{Ewing83,Bear72}. The governing system comprises an elliptic system for pressure $p$ and velocity $\bm{u}=(u^{x},u^{y})^{\top}$, derived from mass conservation of the fluid mixture and Darcy's law, together with a transport equation that enforces the conservation of the injected fluid concentration $c$ \cite{Peaceman77,Douglas83cmd}:
\begin{equation}\label{mod}
   \begin{aligned}
        & \nabla \cdot \bm{u}  = q:= q_{I} + q_{P}, &\quad \Omega \times (0,T],\\
   	  & \bm{u} = - \f{k(\bm{x})}{\mu(c)} \nabla p, &\quad  \Omega \times (0,T],\\
   	  &\phi(\bm{x})\p_{t}c + \nabla \cdot ( \bm{u}c - D \nabla c ) - q_{P}c = q_{I}c_{I}, &\quad \Omega \times (0,T],
   \end{aligned}
\end{equation}
subject to the following no-flow/homogeneous Neumann boundary conditions and initial condition:
\begin{equation}\label{mod:in}
   \begin{aligned}
	    \bm{u} \cdot \bm{n} & = 0, \quad ( \bm{u}c - D \nabla c ) \cdot \bm{n} = 0, &\quad \p \Omega \times (0,T],\\
	  c(\bm{x},0) & = c^{o}(\bm{x}), &\quad \bar{\Omega}\times \{0\},
   \end{aligned}
\end{equation}
where $\bar{\Omega}:=[x_{L},x_{R}] \times [y_{L},y_{R}]$ with boundary denoted by $\p \Omega$, $\Omega = \bar{\Omega} \setminus \p \Omega$, and $\bm{n}$ is the unit outward normal to $\p \Omega$. To ensure uniqueness of the pressure solution, the compatibility condition $\int_{\Omega} q {\rm d} \bm{x} = 0$ is enforced. Here, $\phi(\bm{x})$ and $k(\bm{x})$ are the porosity and permeability of the porous medium, respectively; $q_{I}~(>0)$ and $q_{P}~(<0)$ denote the injection and production rates; $D=D(\bm{x},t)$ is the diffusion coefficient; and $\mu(c)$ is the viscosity determined by some mixture law \cite{viscosity63}, often given by
\begin{equation}\label{mod:mu}
        \mu(c) = \mu_r \left[\mal M^{1/4} c + (1 - c) \right]^{-4},
\end{equation}
where $\mal M := \mu_r / \mu_s$ is the mobility ratio, and $\mu_r$ and $\mu_s$ are the viscosities of resident and solvent, respectively. Equation \eqref{mod:mu} indicates that the viscosity $\mu(c)$ is largely determined by the mobility ratio $\mal{M}$. When $\mal{M} > 1$, the interface between the resident petroleum and the invading water becomes unstable. The larger the mobility ratio, the more readily the mobile water tends to develop elongated fingers that grow toward the production wells, thereby bypassing substantial amounts of hydrocarbons. This phenomenon, known as viscous fingering \cite{Ewing83}, can lead to early water breakthrough once a continuous water channel connects the injection and production wells. At that point, the production wells primarily yield water, and oil recovery is significantly reduced. Consequently, it is highly desirable to design numerical algorithms capable of simulating fingering and assisting in the formulation of strategies to mitigate early water breakthrough and enhance oil recovery.

Physically, the incompressible miscible displacement problem \eqref{mod}--\eqref{mod:in} possesses two essential properties. One is mass conservation, which follows from the fundamental conservation principle of continuum mechanics and can be derived by integrating the concentration transport equation:
\begin{equation}\label{mod:mass_conservation}
  \frac{d}{dt} \int_{\Omega} \phi(\bm{x}) c(\bm{x},t) {\rm d}\bm{x}   =  \int_{\Omega} q_I(\bm{x},t) c_I(\bm{x},t) {\rm d}\bm{x} +\int_{\Omega} q_P(\bm{x},t) c(\bm{x},t) {\rm d}\bm{x}.
\end{equation}
Numerical methods that violate this physical law would lead to non-physical mass accumulation or loss, resulting in inaccurate and unstable long-term simulations. The other property is the boundedness of the concentration, i.e., $0 \le c(\bm{x},t)\le 1$. This constraint is physically mandated, as $c(\bm{x},t)$ represents the volumetric fraction of the injected fluid and consequently must remain within the physical interval $[0,1]$ for all time.  If its approximation falls outside this interval, the numerical quantity $\mal M^{1/4} c + (1 - c)$ in \eqref{mod:mu} may vanish or become sufficiently small, leading to infinite or excessively large values in the coefficient matrix of the velocity-pressure system. Thus, this arithmetic singularity may cause the linear solver to fail and the program to terminate prematurely, as will be demonstrated in the numerical experiments of Section \ref{sec:num}. Consequently, it is essential to devise numerical schemes that simultaneously preserve both mass conservation and bound preservation, thereby ensuring physically admissible solutions and enabling long-term stable simulations. 

Although mass conservation is readily achieved by many existing schemes, preserving positivity or bounds remains a nontrivial task and has received extensive attention. A variety of approaches have been developed to address this challenge, including logarithmic-term-based nonlinear formulations \cite{CWB19,SX}, cut-off techniques \cite{LBY20,YJ22}, M-matrix methods \cite{LJG18,HJW20,HZ23}, 
post-processing or limiter strategies \cite{SQW10,ZXX18}, function transformation approaches \cite{SJ21,WLF24}, and the Lagrange multiplier method \cite{XY19,SJ22-CMAME,SJ22-SIAM}. Nevertheless, the application of the aforementioned structure-preserving algorithms to miscible displacement problems remains largely unexplored. For instance, in \cite{Russell85,Rui13}, the cut-off strategy $\min\{1,\max\{C_h,0\}\}$ is employed to enforce bound preservation by simply truncating any out-of-range concentration values back onto the physical interval $[0,1]$. Although widely used in practical simulations, this approach inevitably compromises the crucial mass conservation property that the underlying physical model intrinsically satisfies. The positivity-preserving limiter technique has also been successfully applied to compressible displacement in porous media \cite{GH17,GH22,GH23}; however, as an explicit treatment of the concentration is adopted for time marching, the associated bound-preserving analysis and simulation typically entail certain time-step restrictions. In \cite{SSY24}, an energy stable, positivity-preserving and mass-conserving numerical method is developed for compressible immiscible two-phase flow with rock compressibility. The core idea is to reformulate the original pressure-based mathematical model into an alternative framework that obeys an energy dissipation law. 
The positivity-preserving property is then ensured by treating the nonlinear convex terms in the logarithmic free energy implicitly, which, however, yields a nonlinear discrete system that necessitates efficient iterative solvers. Several high-order bound-preserving and mass-conserving schemes based on the Lagrange multiplier method are developed in \cite{LXL25} for the thermodynamically consistent incompressible immiscible two-phase flow model in porous media. This approach circumvents the need for complex nonlinear constrained optimization solvers, thereby substantially improving computational efficiency. Unfortunately, the error analysis of such numerical schemes has received little attention to date and remains a challenging task; to the best of our knowledge, the only available  result in this direction is that in \cite{SJ22-SIAM} for a class of semilinear and quasi-linear parabolic equations.  More recently, a cell-centered Galerkin method is developed in \cite{CCGM26} for the incompressible miscible displacement in heterogeneous porous media, it is shown that the coefficient matrix of the method for the Poisson problem is an M-matrix or a monotone matrix; however, no analogous  proof is provided for the miscible displacement problem itself. 
In conclusion, although great achievements have been made in developing positivity- and bound-preserving schemes for various displacement models in porous media, it remains highly necessary to develop linear numerical schemes that are unconditionally bound-preserving and mass-conserving, and to establish rigorous error estimates for such schemes.

On the other hand, a more natural approach to numerically enforce physical constraints is to project (or correct) the numerical solution onto the constrained manifold \cite{SJ92,CYY18}. Recently, a class of positivity- and mass-preserving projection finite difference methods is developed in \cite{CYY24} for the Poisson-Nernst-Planck equation. The underlying idea is to first employ a standard discretization to produce a predicted concentration of  charged particles, which typically fails to preserve positivity and/or mass conservation, and then project it onto the constraint manifold characterized by these two properties, thereby obtaining a corrected concentration of charged particles that rigorously satisfies the physical constraints. In practical simulation, this projection step is characterized by the well-known Karush--Kuhn--Tucker (KKT) conditions, involving an explicit update and a single-variable nonlinear solve for the positivity-preserving constraint; moreover, the existence and uniqueness of the corrected solution are naturally guaranteed by standard convex analysis. More importantly, by introducing an auxiliary exact solution that lies in the constraint manifold and differs from the exact solution only by a known numerical integration error, and leveraging the well-known non-expansive nature of the projection operator, rigorous $L^2$-norm error estimates are established in this work, which are both second-order accurate in space and time. Subsequently, this algorithmic framework has been successfully extended to the Patlak-Keller-Segel equation \cite{CYY26} and the multi-species Keller-Segel chemotaxis system \cite{ZZF26}. However, to the best of our knowledge, the projection idea has not yet been applied to the miscible displacement problem in porous media. Motivated by the aforementioned works, our objective is to design some unconditionally bound-preserving and mass-conserving BCFD projection algorithms for model \eqref{mod}--\eqref{mod:in}, and to establish corresponding rigorous error estimates for the approximations of velocity $\bm{u}$, pressure $p$, and concentration $c$. However, it should be noted that the physical mass in the models considered in \cite{CYY24,CYY26,ZZF26} remains constant over time, while the incompressible miscible displacement model exhibits a mass balance identity in the sense of an injection-production balance, characterized by a time-dependent total mass (see \eqref{mod:mass_conservation}). Consequently, neither the projection algorithm nor the theoretical analysis framework developed in \cite{CYY24,CYY26,ZZF26} can be directly applied to the present model problem. 

To address these challenges, we construct a specially designed weighted $L^{2}$-projection operator for the first-order temporal scheme and rigorously establish its non-expansive property in the corresponding weighted norm. For the purpose of analysis, we further introduce a novel "biased" auxiliary concentration solution, which differs from the exact solution by a provably small perturbation of the same order as the numerical error of the scheme. These two ingredients are essential for establishing a valid error analysis framework tailored to the present model, yet they alone are insufficient to  complete the error estimates. To that end, we establish the relationship between the "biased" errors and the original errors in Lemma \ref{lem:be:ec}. The ideas outlined above are subsequently extended to the second-order scheme. In summary, we propose two novel structure-preserving prediction-then-projection schemes for the incompressible miscible displacement problem. Both schemes are second-order accurate in space by using the well-known block-centered finite difference (BCFD) method \cite{WRF26,XJ22,XJ25,ZZF26,Weiser88,Arbogast97} and the second-order modified upwind technique \cite{Lazarov96-SIAMNA,LD17-JCAM,LD20-JCP}; temporally, one is first-order, whereas the other is second-order. They differ not only in temporal discretizations but also in the ordering of computational steps. In the first-order scheme, we first project the predicted concentration onto the manifold that enforces both bound preservation and mass conservation, yielding a numerical solution that satisfies the desired physical constraints. For the second-order scheme, we reverse the update order of the velocity-pressure system and the concentration to ensure well-posedness and numerical stability. This design choice is necessitated by the fact that the velocity-pressure system depends on an extrapolated concentration, which may still violate the physical bounds even after projection, see Example \ref{exm:e3} in Section \ref{sec:num}. The main contributions of this paper are as follows:
\begin{itemize}
    \item A class of prediction-then-projection BCFD schemes is developed for the incompressible miscible displacement problem. The schemes are almost linear (except for the efficient $L^{2}$- or $H^{1}$-projection step) and are fully decoupled through a judicious ordering of the computational sequences, thereby significantly enhancing computational efficiency.

    \item The schemes are unconditionally bound-preserving and mass-conserving; which ensures the well-posedness of the discretized system and thereby enhances the physical realism of the simulations.

    \item In the error analysis, we construct a specially designed "biased" auxiliary concentration solution. By establishing its perturbation error and the relationship between the "biased" error and the original error (see Lemmas \ref{lem:alpha:be} and \ref{lem:be:ec} for the first-order scheme, and Lemmas \ref{lem:alpha:cn} and \ref{lem:cn:ec} for the second-order scheme), we rigorously derive optimal-order error estimates for the projection algorithms. To the best of our knowledge, this is the first such error analysis for the projection method applied to incompressible miscible displacement problems with time-varying total mass. 

    \item Numerical experiments confirm the theoretical findings. Furthermore, the proposed schemes accurately simulate the standard five-spot well pattern for a oil-water system in porous media. Notably, the viscous fingering phenomenon is well captured and verified.
\end{itemize}

The outline of the paper is presented as follows. In Section \ref{sec:scm}, we first introduce some notations and preliminary lemmas, and then propose the two bound‑preserving and mass‑conserving BCFD projection schemes. In Section \ref{sec:error_analysis}, Optimal-order error estimates, stability results, together with the unique solvability of solutions to the proposed schemes are derived. In Section \ref{sec:num}, we present several numerical experiments to validate the accuracy, physical property-preserving, and reliability of the proposed scheme in some 'real' simulations. Concluding remarks are given in Section \ref{sec:conclusion}.  Throughout this paper, we denote by $K$ with or without subscripts a generic positive constant that is independent of the grid parameters, but may have different values in different occurrences.

\section{Bound-preserving and mass-conserving projection schemes}\label{sec:scm}
In this section, we shall develop two novel linear, bound-preserving and mass-conserving BCFD schemes for model \eqref{mod}--\eqref{mod:in}. Throughout the paper, we suppose that the following assumptions hold:
\begin{itemize}
	\item[(A1)] The exact solutions of model \eqref{mod}--\eqref{mod:in} satisfy
	$$
	\begin{aligned}
		& c \in W^{1,\infty}([0,T] \times\Omega) \cap L^{\infty}(0,T;H^{4}(\Omega)) \cap H^{3}(0,T;L^{\infty}(\Omega)) ,\\
		& \bm{u} \in  L^{\infty}([0,T] \times\Omega) \cap L^{\infty}(0,T;H^{3}(\Omega)),\quad  
		p \in  L^{\infty}(0,T;H^{3}(\Omega)).
	\end{aligned}
	$$
   In addition, there exists a positive constant $K_{e}$ such that
    $$
	\|c\|_{W^{1,\infty}([0,T] \times\Omega)}  + \|\bm{u}\|_{L^{\infty}([0,T]\times \Omega)} \leq K_{e}.
	$$

	\item [(A2)] There exist positive constants $\phi_{*}$, $\phi^{*}$, $k_{*}$, $k^{*}$, $D_{*}$, $D^{*}$ and $K_{p}$, such that
	$$
	\phi_{*} \leq \phi(\bm{x}) \leq \phi^{*}, ~
	k_{*} \leq k(\bm{x}) \leq k^{*}, ~ D_{*} \le D(\bm{x},t) \leq D^{*}, ~
    \|q \|_{L^{\infty}([0,T]\times \Omega)}  \leq K_{p}.
	$$
    Furthermore, we assume that there exist positive constants $\mu_{*}$, $\mu^{*}$ and $L_{\mu}$ such that
	$$
	\mu_{*} \leq \mu(w) \leq \mu^{*},
    \quad
    |\mu(w)-\mu(v)| \leq L_{\mu}|w-v|,
    \qquad \text{for}~ w,v\in [0,1].
	$$
\end{itemize}
For simplicity, below we denote $a(w) := \mu(w)/k(\bm{x})$. Then assumption (A2) implies that there exist positive constants $a_{*}:= { \mu_{*} }/{ k^{*} }$, $a^{*}:= { \mu^{*} }/{ k_{*} }$ and $L_{a}:= L_{\mu}/k_{*}$ such that
\begin{equation}\label{assump:a}
    a_{*} \le a(w) \le a^{*},
    \quad
    |a(w) - a(v)| \leq L_{a}|w-v|,
    \qquad \text{for}~ w,v\in [0,1].
\end{equation}
Specifically, for the viscosity $\mu(c)$ defined in \eqref{mod:mu}, we have
$$
\mu_{*}:= \min\{ \mu_{r}, \mu_{s} \},~
\mu^{*}:= \max\{ \mu_{r}, \mu_{s} \},~
L_{\mu}:= \max \big\{ 
    4\mu_{r}\big| \mal{M}^{1/4} - 1 \big|, 
    4\mu_{r} \big| \mal{M}^{1/4} - 1 \big|\mal{M}^{-5/4}
    \big\}.
$$

\subsection{Notations and preliminary results}
We first present some basic notations. Let $N_{t}$ be a positive integer, and let the time interval $[0,T]$ be uniformly partitioned by $t^{n} = n \tau$, with the time step given by $\tau=T/N_{t}$. For temporal grid functions $\{\theta^{n}\}_{n \geq 0}$ and $\{\widehat{\theta}^{n}\}_{n \geq 0}$, define
$$
  d_{\tau}\theta^{n}= \f{ \theta^{n} - \theta^{n-1} }{ \tau }, \quad
  d_{\tau}\widehat{\theta}^{n}:= \f{ \widehat{\theta}^{n} - \widehat{\theta}^{n-1} }{ \tau }, \quad
  D_{\tau}\widehat{\theta}^{n}= \f{ \widehat{\theta}^{n} - \theta^{n-1} }{ \tau }.
$$  

Let $ N_x $ and $N_y$ be the number of grids along the $x$- and $y$-coordinates, respectively. We introduce staggered spatial grids along each spatial direction as follows:
$$
  \begin{aligned}
    \Pi_{x} :~ x_{i+1/2} = x_{L}+ih^{x}, \ i = 0,\ldots, N_{x}, \quad 
        \Pi_{x}^{*}:~ x_{i} = \f{ x_{i+1/2} +x_{i-1/2} }{2}, \ i = 1,\ldots,N_{x},\\
    \Pi_{y} :~ y_{j+1/2} = y_{L}+jh^{y}, \ j = 0,\ldots, N_{y}, \quad 
         \Pi_{y}^{*}:~ y_{j} = \f{ y_{j+1/2} + y_{j-1/2} }{2}, \ j = 1,\ldots,N_{y},
   \end{aligned}
$$
with spatial grid sizes $h^{x} = (x_{R}-x_{L})/N_{x}$ and $h^{y} = (y_{R}-y_{L})/N_{y}$. Let $h:= \max\{ h^{x}, h^{y}\}$. On the staggered spatial grids, we introduce the discrete function spaces for the approximations of $c$, $p$, and the two components of $\bm{u}$, respectively, by
$$
\begin{aligned}
    &\mal{C} := \left\lbrace g=\{g_{i,j}\} \mid (x_{i},y_{j})\in \Pi_{x}^{*}\times \Pi_{y}^{*} \right\rbrace ,\\
	&\mal{U} := \left\lbrace g = \{g_{i+1/2,j}\}  \mid  (x_{i+1/2},y_{j})\in \Pi_{x}\times \Pi_{y}^{*},
	~ g_{1/2,j}=g_{N_{x}+1/2,j}=0\right\rbrace,\\
	&\mal{V} := \left\lbrace g= \{g_{i,j+1/2}\}  \mid  (x_{i},y_{j+1/2})\in \Pi_{x}^{*}\times \Pi_{y},
	~ g_{i,1/2}=g_{i,N_{y}+1/2}=0  \right\rbrace.
\end{aligned}
$$
Moreover, for $w \in \mal{C}$ and $ \bm{v} = (v^{x},v^{y})^{\top} \in \mal{U} \times \mal{V}$, we define the discrete gradient and divergence operators as 
$$
  \nabla_{h}w=(d_{x} w,d_{y} w)^{\top}, \qquad \nabla_{h} \cdot \bm{v}= \de_{x} v^{x} + \de_{y} v^{y},
$$  
where the difference operators $d_{x} w \in \mal{U}$, $d_{y} w \in \mal{V}$ and $\de_{x}v^{x}, \de_{y}v^{y} \in \mal{C}$ are given by
$$
   [d_x w]_{i+1/2,j}=\df{w_{i+1,j}-w_{i,j}}{h^{x}},\quad
   [d_y w]_{i,j+1/2}=\df{w_{i,j+1}-w_{i,j}}{h^{y}},
$$
$$
   [\de_x v^{x}]_{i,j}=\df{v^{x}_{i+1/2,j}-v^{x}_{i-1/2,j}}{h^{x}},\quad
   [\de_y v^{y}]_{i,j}=\df{v^{y}_{i,j+1/2}-v^{y}_{i,j-1/2}}{h^{y}}.
$$ 

In addition, we also introduce some local weighted average operators and the second-order modified upwind discretization. First, the local directional average operators $\mal{T}_{x}:~\mal{C} \rightarrow \mal{U}$ and $\mal{T}_{y}:~\mal{C} \rightarrow \mal{V}$ are defined as
$$
  [ \mal{T}_{x} w ]_{i+1/2,j} = \f{w_{i,j} + w_{i+1,j}  }{2},\quad 
  [ \mal{T}_{y} w ]_{i,j+1/2} = \f{w_{i,j} + w_{i,j+1}  }{2}.
$$	
Second, for $w \in \mal{C}$ satisfying the boundary conditions $w_{i,j}=w_{i+1,j}$ for $i \in\{0,N_{x}\}$, $j=1,\ldots,N_{y}$, and $w_{i, j}=w_{i, j+1}$ for $i=1,\ldots,N_{x}$, $j \in \{0,N_{y}\}$, and a given velocity field $\bm{v} = (v^{x},v^{y})^{\top} \in \mal{U}\times \mal{V}$, we introduce the second-order modified upwind FD approximation of $\bm{v} w - D(\bm{x},t) \nabla w$ by \cite{Lazarov96-SIAMNA,LD17-JCAM,LD20-JCP}:
\begin{equation}\label{mupwind}
     \bm{\mal{B}}(w)\bm{v} - \bm{\mal{D}}(\bm{v}) \nabla_{h} w = \left(
  \begin{aligned}
      &v^{x}\,\mal{H}_{x} w\\
      &v^{y}\,\mal{H}_{y} w
  \end{aligned} \right)
  - \left( 
  \begin{aligned}
     & \underline{D}_{1,1}      &  \\
     &       & \underline{D}_{2,2}  
  \end{aligned} \right) \nabla_{h} w
  \in 
  \left(
  \begin{aligned}
      &\mal{U} \\
      &\mal{V}
  \end{aligned} \right),
\end{equation}
where $\mal{H}_{x} w$ and $\mal{H}_{y} w$ are taken as the upwind values on $\Pi_{x} \times \Pi_{y}^{*}$ and $\Pi_{x}^{*} \times \Pi_{y}$, respectively, i.e., determined by the signs of the velocity components:
\begin{equation*}
  \begin{aligned}
      & [\mal{H}_{x} w]_{i+1/2,j} = \psi\big(v_{i+1/2,j}^{x}\big) w_{i,j} + \big( 1 - \psi\big(v_{i+1/2,j}^{x} \big) \big) w_{i+1,j}, \\
      & [\mal{H}_{y} w]_{i,j+1/2} = \psi\big(v_{i,j+1/2}^{y}\big) w_{i,j} + \big( 1 - \psi\big(v_{i,j+1/2}^{y}\big) \big) w_{i,j+1},
  \end{aligned}
\end{equation*}
and 
\begin{equation*}
    \psi(s)=\left\{
    \begin{aligned}
        & 1, & s \geq 0,\\
        & 0, & s <0,
    \end{aligned}
    \right. \quad 
    \underline{D}_{1,1} = \f{D}{1+\f{|v^{x}|}{2D}h^{x}},\quad
    \underline{D}_{2,2} = \f{D}{1+\f{|v^{y}|}{2D}h^{y}}.
\end{equation*}
Furthermore,  we define the compact notation:
$$
\bm{\mal{A}}(w)= \left( 
   \begin{aligned}
	&a(\mal{T}_{x}w) &\\
	& &a(\mal{T}_{y}w)
  \end{aligned}\right).
$$

 Finally, for any vectors $\bm{w}=(w^x,w^y)^{\top}, ~\bm{v}=(v^x, v^y)^{\top}\in \mal{U} \times \mal{V}$ and scalars $w, v \in \mal{C}$, we introduce the following discrete $L^2$ inner products
$$
   (\bm{w},\bm{v})_{\rm T}=(w^x,v^x)_{x}+(w^y,v^y)_{y},   \quad 
   (w,v)_{\rm M}=\sum_{i=1}^{N_{x}}\sum_{j=1}^{N_{y}}h^{x} h^{y}\, w_{i,j}\, v_{i,j},
$$
with
$$
	\begin{aligned}
		 (w^x,v^x)_{x}= \sum_{i=1}^{N_{x}-1}\sum_{j=1}^{N_{y}} h^{x} h^{y}\, w^x_{i+1/2,j}\, v^x_{i+1/2,j},\quad        (w^y,v^y)_{y}=\sum_{i=1}^{N_{x}}\sum_{j=1}^{N_{y}-1} h^{x} h^{y}\, w^y_{i,j+1/2}\, v^y_{i,j+1/2}.
	\end{aligned}
$$
The corresponding discrete norms are defined as $\| \cdot \|_{\ell} = \sqrt{(\cdot,\cdot)_{\ell}}$ for $\ell = x, y, \rm{T}, \rm{\rm{M}}$. Besides, we also introduce the following weighted $L^2$-norm and weighted $H^{1}$-seminorm: 
$$
  \|w\|_{{\rm M},\sigma} := \sqrt{(\sigma w,w)_{\rm M}}, \qquad \|\bm{w}\|_{{\rm T},\bm{E}}:=\sqrt{( \bm{E} \bm{w},\bm{w})_{\rm T}},
$$ 
where $\sigma \ge 0$ is a prescribed weight function and $\bm{E}$ is a given positive definite weight matrix.

Next, we present several useful and well-known lemmas that will be used in the theoretical analysis.
\begin{lemma}[\cite{SYL21,XJ25}]\label{lem:dis}
Let $w \in \mal{C}$ and $\bm{v} \in \mal{U} \times \mal{V}$. Then, we have
$$
\big( \nabla_{h} \cdot \bm{v}, w \big)_{\rm M} = - \big( \bm{v}, \nabla_{h} w \big)_{\rm T}.
$$
\end{lemma}

\begin{lemma}[\cite{SYL21}]\label{lem:seminorm}
	Assume that $w \in \mal{C}$ with $w_{1,1}=0$. Then, we have 
	\begin{equation*}
		\| w \|_{\rm M}^{2} \leq 2(x_{R}-x_{L})^{2} \|d_{x}w \|_{x}^{2} + 2(y_{R}-y_{L})^{2} \|d_{y}w \|_{y}^{2} \leq K_{o} \| \nabla_{h}w \|_{\rm T}^{2},
	\end{equation*}
	where $K_{o}:= 2(x_{R}-x_{L})^{2} + 2(y_{R}-y_{L})^{2}$. 
\end{lemma}

\begin{lemma}[\cite{TNA09}]\label{lem:t:err}
	Assume that $g \in H^{2}(\Omega)$. Then, there exists a positive constant $K$, independent of $h$, such that
	$$
	\|g - \mal{T}_{x}g\|_{x} + \|g - \mal{T}_{y}g\|_{y} \leq Kh^{2}.
	$$
\end{lemma}


The following result is a direct consequence of the Bramble-Hilbert lemma (cf.~\cite{BH}).
\begin{lemma}\label{lem:d:err}
	If $\bm{f} \in H^{3}(\Omega)$ and $g \in H^{3}(\Omega)$. Then, we have
	\begin{equation*}
		\| \nabla_{h} \cdot \bm{f} - \nabla \cdot \bm{f} \|_{\rm M}  + \| \nabla_{h}g - \nabla g \|_{\rm T}  \leq Kh^{2}.
	\end{equation*}
\end{lemma}

Following the results in \cite{Lazarov96-SIAMNA,LD17-JCAM}, the modified second‑order upwind FD approximation \eqref{mupwind} is second-order accurate, that is
\begin{lemma}\label{lem:upwind:err}
    If $\bm{v} \in L^{\infty}(\Omega)$ and $w \in H^{3}(\Omega)$. Then, we have
    \begin{equation*}
		  \| \bigl(\bm{v} w - D(\bm{x},t) \nabla w\bigr) - \bigl(\bm{\mal{B}}(w)\bm{v} - \bm{\mal{D}}(\bm{v}) \nabla_{h} w \bigr)\|_{\rm T}  \leq Kh^{2}.
	\end{equation*}
\end{lemma}

\subsection{First-order temporal scheme}
In this subsection, we present a linear first-order projection scheme for approximating the velocity $\bm{u}$, pressure $p$, and concentration $c$, which is both bound-preserving and mass-conserving for the concentration. The scheme combines the backward Euler (BE) method with an implicit-explicit (IMEX) treatment for time discretization, and employs the BCFD method together with the upwind technique for spatial discretization on staggered grids.

Let $P_{h}^{n}=\{P_{i,j}^{n}\}\in \mal{C}$, $ \bm{U}_{h}^{n}=\{U_{i+1/2,j}^{x,n}, U_{i,j+1/2}^{y,n}\}\in  \mal{U} \times \mal{V}$, and $C_{h}^{n}=\{C_{i,j}^{n}\}\in \mal{C}$ denote the finite difference approximations of the exact solutions $p(t^{n})$, $\bm{u}(t^{n})$, and $c(t^{n})$, respectively. Then, given the initial concentration $C_{h}^{0} =\widehat{C}_{h}^{0} =c^o\in \mal{C}$, the first-order prediction-then-projection BCFD scheme (abbreviated as BE-BCFD-P) for $n=1,\ldots,N_{t}$ is proposed as follows:

\medskip
\textbf{Step 1:} (Velocity \& pressure update). Given $C_{h}^{n-1} \in \mal{C}$, solve the velocity $\bm{U}_{h}^{n} \in \mal{U} \times \mal{V}$ and pressure $P_{h}^{n} \in \mal{C}$ through the linear algebraic system:
\begin{equation}\label{scm:be:u}
	\nabla_{h} \cdot \bm{U}_{h}^{n} = q^n, 
\quad 
	\bm{\mal{A}}(C_{h}^{n-1}) \bm{U}_{h}^{n} +  \nabla_{h} P_{h}^{n} = 0.
\end{equation}
 In addition, to ensure the uniqueness of the solution $P_{h}^{n}$, we enforce $P_{1,1}^{n}=0$.

\textbf{Step 2:} (Prediction of concentration). Solve the intermediate concentration $\widehat{C}_{h}^{n}$ via the BE-BCFD method
\begin{equation}\label{scm:be:c}
	\phi D_{\tau}\widehat{C}_{h}^{n} + \nabla_{h} \cdot \bigl( \bm{\mal{B}}(\widehat{C}_{h}^{n})\bm{U}_{h}^{n}- \bm{\mal{D}}(\bm{U}_{h}^{n}) \nabla_{h}\widehat{C}_{h}^{n} \bigr) - q_{P}\widehat{C}_{h}^{n} = \big[q_{I}c_{I}\big]^{n},
\end{equation}
 enclosed with homogeneous boundary conditions 
 $$ 
  (\bm{\mal{B}}(\widehat{C}_{h}^{n})\bm{U}_{h}^{n}- \bm{\mal{D}}(\bm{U}_{h}^{n}) \nabla_{h}\widehat{C}_{h}^{n}) \cdot \bm{n}=0.
 $$  

\textbf{Step 3:} (Correction of concentration). Project the nodal solution $\widehat{C}_{h}^{n}$ to the constrained manifold with mass conservation and bound preservation via the weighted $L^{2}$-projection such that
\begin{equation}\label{scm:be:min}
	\begin{aligned}
		&\quad \min_{C_{h}^{n} \in \mal{C}}~ \f{1}{2} \|C_{h}^{n} - \widehat{C}_{h}^{n}\|_{{\rm M},\beta^{n}}^{2}\\
		& \text{s.t.} \quad 0 \leq C_{h}^{n} \leq 1, \quad G(C_{h}^{n}):= \big( \beta^{n}C_{h}^{n},1 \big)_{\rm M} - \big( \beta^{n}\widehat{C}_{h}^{n},1 \big)_{\rm M}=0,
	\end{aligned}
\end{equation}
where $\beta^{n}:=\phi - \tau q_{P}^{n}$, and $G(C_{h}^{n})=0$ is used to enforce mass conservation.

Using the well-known Karush-Kuhn-Tucker (KKT) condition \cite{KKT08} for the constrained convex minimization problem, the solution to \eqref{scm:be:min} is equivalently characterized by
\begin{equation}\label{scm:be:k1}
	\beta^{n} C_{h}^{n} = \beta^{n} \widehat{C}_{h}^{n} + \lambda_h^{n} - \xi_h^{n} - \gamma^{n} \beta^{n},
\end{equation}
\begin{equation}\label{scm:be:k2}
	 G(C_{h}^{n})=0,
\end{equation}
\begin{equation}\label{scm:be:k3}
	 0 \leq C_{h}^{n} \leq 1, \quad \lambda_h^{n}C_{h}^{n} =0,\quad \xi_h^{n}\,(1-C_{h}^{n})=0, \quad \lambda_h^{n} \geq 0, \quad \xi_h^{n} \geq 0,
\end{equation}
where $\lambda_h^{n}$, $\xi_h^{n} \in \mal{C}$ are the Lagrange multipliers associated with the bound-preserving constraints, and $ \gamma^{n} \in \mathbb{R}$ is the Lagrange multiplier corresponding to the mass conservation constraint.

\begin{remark}\label{rem:mass:be} Taking the discrete inner product of \eqref{scm:be:c} with $1$, and applying Lemma \ref{lem:dis} together with the homogeneous boundary conditions, yields
\begin{equation*}\label{mass:be:1}
    \begin{aligned}
       \big( \phi \widehat{C}_{h}^{n} - \phi C_{h}^{n-1} ,1 \big)_{\rm M}  
       - \tau \big( q_{P}^{n}\widehat{C}_{h}^{n} , 1 \big)_{\rm M}
       = \tau \big( [q_{I}c_{I}]^{n}, 1 \big)_{\rm M}.
    \end{aligned}
\end{equation*}
Furthermore, combining this with \eqref{scm:be:k2} gives 
\begin{equation}\label{mass:be:2}
    \big( \phi  C_{h}^{n},  1 \big)_{\rm M}   - \big( \phi  C_{h}^{n-1},  1 \big)_{\rm M} 
    =
     \tau \big( [q_{I}c_{I}]^{n}, 1 \big)_{\rm M}+ \tau \big( [q_{P}C_{h}]^{n}, 1 \big)_{\rm M},
\end{equation}
which is precisely a discrete analogue of the mass balance identity \eqref{mod:mass_conservation}. This observation motivates the choice of the mass constraint condition $G(C_{h}^{n})=0$ in the correction step \eqref{scm:be:min}.
\end{remark}
\begin{remark}\label{rem:lm:be} In view of the KKT conditions \eqref{scm:be:k1}--\eqref{scm:be:k3}, the numerical concentration $C_{h}^{n}$ and the Lagrange multipliers $\lambda_h^{n}$ and $\xi_h^{n}$ admit the following explicit closed-form expressions:
\begin{equation}\label{be:lm}
	\big( C_{h}^{n}, \lambda_h^{n}, \xi_h^{n} \big) = \left\lbrace
	\begin{aligned}
		&\big( \widehat{C}_{h}^{n} - \gamma^{n}, 0, 0   \big),                     &  0 \leq\widehat{C}_{h}^{n} - \gamma^{n} \leq 1,\\
		&\big( 0, -\beta^{n} \big(\widehat{C}_{h}^{n} - \gamma^{n}\big), 0  \big), &  \widehat{C}_{h}^{n} - \gamma^{n} < 0,\\
		&\big( 1, 0, \beta^{n} \big(\widehat{C}_{h}^{n} - \gamma^{n}-1\big) \big), &  \widehat{C}_{h}^{n} - \gamma^{n} > 1.
	\end{aligned}\right. 
\end{equation}
Meanwhile, the Lagrange multiplier $\gamma^{n} \in \mathbb{R}$ is determined by the mass conservation constraint \eqref{scm:be:k2}; specifically, it is given as the solution to the nonlinear algebraic equation $F(\gamma)=0$, where
\begin{equation}\label{F:be}
    F(\gamma):=\sum_{0 \leq \widehat{C}_{i,j}^{n} - \gamma^{n} \leq 1} h^{x}h^{y} \beta_{i,j}^{n}\big(\widehat{C}_{i,j}^{n}-\gamma \big)
       + \sum_{\widehat{C}_{i,j}^{n} - \gamma^{n} > 1} h^{x}h^{y} \beta_{i,j}^{n} 
       - \big( \beta^{n}\widehat{C}_{h}^{n},1 \big)_{\rm M}.
\end{equation}
We note that the cases $C_{h} \equiv 1$ and $C_{h} \equiv 0$ correspond to two extreme states of the injected fluid: the former indicates that the injected fluid completely displaces the resident fluid, whereas the latter means that no injection has occurred. Excluding these two extreme cases, the mass conservation constraint \eqref{scm:be:k2} implies
$$
  0 \le \big( \beta^{n}\widehat{C}_{h}^{n},1 \big)_{\rm M} =  \big( \beta^{n}C_{h}^{n},1 \big)_{\rm M} 
  \le \big( \beta^{n},1 \big)_{\rm M}. 
$$
If $\widehat{C}_{h}^{n} - \gamma^{n} \leq 0$ everywhere, i.e., $\gamma^n \geq \gamma_{\text{max}}^{n}= \max_{i,j}\{\widehat{C}_{i,j}^{n}\}$, then from \eqref{F:be} we obtain
$$
F(\gamma^{n}) = F(\gamma_{\text{max}}^{n}) = - \big( \beta^{n}\widehat{C}_{h}^{n},1 \big)_{\rm M} \le 0.
$$ 
If $\widehat{C}_{h}^{n} - \gamma^{n} \geq 1$ everywhere, i.e., $\gamma^{n} \leq \gamma_{\text{min}}^{n}:= \min_{i,j}\{\widehat{C}_{i,j}^{n}-1\}$, then 
$$
F(\gamma^{n}) = F(\gamma_{\text{min}}^{n}) =  \big( \beta^{n},1 \big)_{\rm M} -  \big( \beta^{n}\widehat{C}_{h}^{n},1 \big)_{\rm M} \ge 0.
$$
Otherwise, for $\gamma^n \in [\gamma_{\text{min}}^{n},\gamma_{\text{max}}^{n}]$, it follows from the above discussion that $F(\gamma)$ is strictly monotonically decreasing and satisfies
$$
  F(\gamma_{\text{max}}^{n}) \le F(\gamma^{n}) \le F(\gamma_{\text{min}}^{n}).
$$
Therefore, the nonlinear equation \eqref{F:be} admits a unique solution $\gamma_{*}^{n}\in [\gamma_{\text{min}}^{n},\gamma_{\text{max}}^{n}]$.
\end{remark}
\begin{remark} In practice, the objective function in the projection step \eqref{scm:be:min} can be simply taken as the standard $\f{1}{2}\|C_{h}^{n}-\widehat{C}_{h}^{n}\|_{\rm M}^{2}$ instead of $\f{1}{2} \|C_{h}^{n} - \widehat{C}_{h}^{n}\|_{{\rm M},\beta^{n}}^{2}$, owing to the equivalence of these two norms. Nevertheless, the latter form is preferred for the convenience of numerical analysis. 
\end{remark}
\subsection{Second-order temporal scheme}
In the previous subsection, we presented a first‑order algorithm that preserves both the solution bounds and the mass balance identity for the concentration. To achieve higher‑order temporal accuracy for practical simulations, we now propose an alternative second‑order projection algorithm based on the Crank-Nicolson (CN) time discretization and the IMEX approach, while still preserving the desired physical properties.

For $n=1,\ldots,N_{t}$, the second-order prediction-then-projection BCFD scheme (abbreviated as CN-BCFD-P) is implemented as follows: \\

\textbf{Step 1:} (Prediction of concentration). The intermediate concentration $\widehat{C}_{h}^{n}$ is first computed via the second-order CN-BCFD method
\begin{equation}\label{scm:cn:c}
	\begin{aligned}
		\phi D_{\tau}\widehat{C}_{h}^{n} & + \f{1}{2}\nabla_{h} \cdot \big( \big[ \bm{\mal{B}}(\widehat{C}_{h}^{n}) + \bm{\mal{B}}(C_{h}^{n-1}) \big]\widetilde{\bm{U}}_{h}^{n-1/2}  - \bm{\mal{D}}(\widetilde{\bm{U}}_{h}^{n-1/2}) \nabla_{h}\big[ \widehat{C}_{h}^{n} + C_{h}^{n-1}\big]  \big)  \\
		& - \f{1}{2} \big( \big[ q_{P} \widehat{C}_{h} \big]^{n} + \big[ q_{P}C_{h} \big]^{n-1} \big) = \big[q_{I}c_{I}\big]^{n-1/2},
	\end{aligned}
\end{equation}
enclosed with homogeneous boundary conditions 
$$ 
\big( \big[ \bm{\mal{B}}(\widehat{C}_{h}^{n}) + \bm{\mal{B}}(C_{h}^{n-1}) \big]\widetilde{\bm{U}}_{h}^{n-1/2}  - \bm{\mal{D}}(\widetilde{\bm{U}}_{h}^{n-1/2}) \nabla_{h}\big[ \widehat{C}_{h}^{n} + C_{h}^{n-1}\big]  \big) \cdot \bm{n}=0.
$$
Here
$$
\widetilde{\bm{U}}_{h}^{n-1/2}=
\left\{
\begin{aligned}
    & \bm{U}_{h}^{0},\quad &n=1,\\
    &  \f{3\bm{U}_{h}^{n-1} - \bm{U}_{h}^{n-2}}{2}, \quad & n \geq 2,
\end{aligned}
\right.
$$
and the initial velocity $\bm{U}_{h}^{0}$ is obtained by solving the linear system in \eqref{scm:cn:u} with $n=0$.

\textbf{Step 2:} (Correction of concentration). Project the nodal solution $\widehat{C}_{h}^{n}$ to the constrained manifold with mass conservation and bound preservation via the weighted $L^{2}$-projection such that 
\begin{equation}\label{scm:cn:min}
	\begin{aligned}
		&\quad \min_{C_{h}^{n} \in \mal{C}}~ \f{1}{2} \|C_{h}^{n} - \widehat{C}_{h}^{n}\|_{{\rm M},\bar{\beta}^{n}}^{2} \\
		& \text{s.t.} \quad 0 \leq C_{h}^{n} \leq 1, \quad \bar{G}(C_{h}^{n}):= \big( \bar{\beta}^{n}C_{h}^{n},1 \big)_{\rm M} - \big( \bar{\beta}^{n}\widehat{C}_{h}^{n},1 \big)_{\rm M}=0,
	\end{aligned}
\end{equation}
where $\bar{\beta}^{n}:= \phi - \tau q_{P}^{n}/2$, and $\bar{G}(C_{h}^{n})=0$ is used to enforce the mass conservation. 

\textbf{Step 3:} (Velocity \& pressure update). Given $C_{h}^{n}\in \mal{C}$, solve the velocity $\bm{U}_{h}^{n} \in \mal{U} \times \mal{V}$ and pressure $P_{h}^{n} \in \mal{C}$ via the linear algebraic system 
\begin{equation}\label{scm:cn:u}
	\nabla_{h} \cdot \bm{U}_{h}^{n} = q^n, 
\quad 
	\bm{\mal{A}}(C_{h}^{n}) \bm{U}_{h}^{n} +  \nabla_{h} P_{h}^{n} = 0.
\end{equation}
 In addition, to ensure the uniqueness of the solution $P_{h}^{n}$, we enforce $P_{1,1}^{n}=0$.
\begin{remark}\label{rem:extrap} It is worth noting that the implementation sequence of the proposed second-order projection scheme differs from that of the first-order one, in which different treatments of the concentration variable are adopted for the velocity and pressure updates; see \eqref{scm:be:u} and \eqref{scm:cn:u} for comparison. If we were to solve the velocity-pressure system first, as in the first-order approach, a second-order extrapolation $\tilde{C}_{h}^{n}=2C_{h}^{n-1} - C_{h}^{n-2}$ would be required to approximate the concentration at time $t=t^{n}$. Although the previous time-level concentrations satisfy the bounds, i.e., $C_{h}^{n-1}, C_{h}^{n-2} \in [0,1]$, the extrapolated value $\tilde{C}_{h}^{n}$ may fall outside this interval. Consequently, the quantity $\mal{M}^{1/4}\tilde{C}_{h}^{n} + 1- \tilde{C}_{h}^{n}$ could become zero or sufficiently small, which renders \eqref{mod:mu} ill-posed and may lead to serious numerical blow-up. This issue will be illustrated through numerical experiments  in Section \ref{sec:num}. To address this issue, the proposed second-order scheme reverses the usual implementation sequence by computing the concentration prior to the updates of velocity and pressure.
\end{remark}
\begin{remark}\label{rem:mass:cn} Analogous to Remark \ref{rem:mass:be}, taking the discrete inner product of \eqref{scm:cn:c} with $1$, and applying Lemma \ref{lem:dis} together with the homogeneous boundary conditions, yields
\begin{equation*}\label{mass:cn:1}
    \begin{aligned}
       \big( \phi \widehat{C}_{h}^{n} - \phi C_{h}^{n-1} ,1 \big)_{\rm M} 
       - \f{\tau}{2} \big( [q_{P}\widehat{C}_{h}]^{n} + [q_{P}C_{h}]^{n-1} , 1 \big)_{\rm M}
       = \tau \big( [q_{I}c_{I}]^{n-1/2}, 1 \big)_{\rm M}.
    \end{aligned}
\end{equation*}
Combining this with \eqref{scm:cn:k2} implies that 
\begin{equation*}\label{mass:cn:2}
    \big( \phi  C_{h}^{n},  1 \big)_{\rm M}   - \big( \phi  C_{h}^{n-1},  1 \big)_{\rm M} 
    =
     \tau \big( [q_{I}c_{I}]^{n-1/2}, 1 \big)_{\rm M}+ \f{\tau}{2} \big( [q_{P}C_{h}]^{n}+[q_{P}C_{h}]^{n-1}, 1 \big)_{\rm M}.
\end{equation*}
This identity provides a second-order discrete analogue of the mass balance identity \eqref{mod:mass_conservation}, which motivates us to impose the mass constraint $\bar{G}(C_{h}^{n})=0$ in the correction step \eqref{scm:cn:min}.
\end{remark}
\begin{remark} The convex optimization problem \eqref{scm:cn:min} is equivalent to the following KKT conditions: 
\begin{equation}\label{scm:cn:k1}
	 \bar{\beta}^{n}C_{h}^{n} =\bar{\beta}^{n} \widehat{C}_{h}^{n} + \lambda_h^{n} - \xi_h^{n} - \gamma^{n} \bar{\beta}^{n},
\end{equation}
\begin{equation}\label{scm:cn:k2}
	\bar{G}(C_{h}^{n})=0,
\end{equation}
\begin{equation}\label{scm:cn:k3}
	 0 \leq C_{h}^{n} \leq 1, \quad \lambda_h^{n}C_{h}^{n} =0,\quad \xi_h^{n}\,(1-C_{h}^{n})=0, \quad \lambda_h^{n} \geq 0, \quad \xi_h^{n} \geq 0,
\end{equation}
where $\lambda_h^{n}$, $\xi_h^{n} \in \mal{C}$ and $ \gamma^{n} \in \mathbb{R}$ are the corresponding Lagrange multipliers as the first-order scheme. The detailed implementation of the KKT conditions \eqref{scm:cn:k1}--\eqref{scm:cn:k3}, is similar to that of the first‑order scheme; the specific details are provided in Remark \ref{rem:lm:be}. In addition to the weighted $L^{2}$-projection used in the correction step of the first- and second-order schemes (see \eqref{scm:be:min} and  \eqref{scm:cn:min}),  other projection types can also be employed \cite{CYY24}, for example, the weighted $H^{1}$-projection (see \ref{App:A} for details). The corresponding first- and second-order projection schemes are referred to as BE‑BCFD‑P2 and CN‑BCFD‑P2, respectively.
\end{remark}

\section{Numerical analysis} \label{sec:error_analysis}
In this section, we first establish the numerical analysis for the BE-BCFD-P scheme \eqref{scm:be:u}--\eqref{scm:be:min}, including its unique solvability, optimal-order error estimate, and stability with respect to the initial concentration and source term. The corresponding analysis for the CN-BCFD-P scheme \eqref{scm:cn:c}--\eqref{scm:cn:u} is then briefly outlined.

\subsection{Analysis for the first-order scheme}\label{subsec:1stAnal}
Regarding the discrete weighted $L^2$-projection, we have the following two conclusions.
\begin{lemma}[Orthogonality of the weighted $L^2$-projection]\label{lem:lm:be}
The three Lagrange multipliers $\lambda_h^{n}$, $\xi_h^{n} \in \mal{C}$ and $\gamma^{n} \in \mathbb{R}$ in \eqref{scm:be:k1}--\eqref{scm:be:k3} satisfy
$$
\big( \lambda_h^{n} - \xi_h^{n} - \gamma^{n} \beta^{n}, 1 \big)_{\rm M} = 0,\quad n = 1,\ldots,N_{t}.
$$
\end{lemma}
\begin{proof} It follows from condition \eqref{scm:be:k1} that
\begin{equation*}
    \big( \beta^{n} C_{h}^{n},1  \big)_{\rm M} =\big( \beta^{n} \widehat{C}_{h}^{n},1  \big)_{\rm M} 
     + \big( \lambda_{h}^{n} - \xi_{h}^{n} - \gamma^{n} \beta^{n} , 1 \big)_{\rm M}.
\end{equation*}
Together with the mass conservation constraint \eqref{scm:be:k2}, the conclusion is obtained.
\end{proof}

\begin{lemma}[Stability of the weighted $L^2$-projection]\label{lem:hat:be}
For the intermediate concentration $\widehat{C}_{h}^{n}$ and the projected concentration $C_{h}^{n}$ in the BE-BCFD-P scheme \eqref{scm:be:u}--\eqref{scm:be:min}, it holds
$$
  \| C_{h}^{n} \|_{{\rm M},\beta^{n}}^{2} + \| C_{h}^{n}- \widehat{C}_{h}^{n} \|_{{\rm M},\beta^{n}}^{2} \leq \| \widehat{C}_{h}^{n} \|_{{\rm M},\beta^{n}}^{2},\quad n = 1,\ldots,N_{t}.
$$
\end{lemma}
\begin{proof}
Taking the discrete inner product of \eqref{scm:be:k1} with $ C_{h}^{n}$, we obtain
\begin{equation}\label{cd:hat:1}
	\f{\| C_{h}^{n} \|_{{\rm M},\beta^{n}}^{2} + \| C_{h}^{n}- \widehat{C}_{h}^{n} \|_{{\rm M},\beta^{n}}^{2}      - \| \widehat{C}_{h}^{n} \|_{{\rm M},\beta^{n}}^{2} }{2}
      = \big( \lambda_h^{n} - \xi_h^{n} - \gamma^{n} \beta^{n}, C_{h}^{n} \big)_{\rm M}. 
\end{equation} 
Define $\varpi^{n} := \big( \beta^{n}, C_{h}^{n} \big)_{\rm M}/\big( \beta^{n}, 1 \big)_{\rm M}$. Then, by Lemma \ref{lem:lm:be}, together with $\beta^{n} >0$ and $C_{h}^{n} \in [0,1]$, we have
\begin{equation}\label{cd:hat:2}
 \gamma^{n} \big( \beta^{n}, C_{h}^{n} \big)_{\rm M} = \varpi^{n} \big( \lambda_{h}^{n} - \xi_{h}^{n},  1 \big)_{\rm M} \quad \text{and}\quad \varpi^{n} \in [0,1]. 
\end{equation}
Furthermore, combining the KKT condition \eqref{scm:be:k3} with \eqref{cd:hat:2} yields
\begin{equation}\label{cd:hat:3}
	\begin{aligned}
		\big( \lambda_{h}^{n} - \xi_{h}^{n} - \gamma^{n} \beta^{n},  C_{h}^{n} \big)_{\rm M}
		&= \big( \lambda_h^{n}, C_{h}^{n} \big)_{\rm M} + \big( \xi_h^{n}, 1-C_{h}^{n}  \big)_{\rm M} 
         -\varpi^{n} \big( \lambda_h^{n},  1 \big)_{\rm M} + (\varpi^{n} - 1 )\big( \xi_h^{n}, 1 \big)_{\rm M}\\
		& = -\varpi^{n} \big( \lambda_h^{n},  1  \big)_{\rm M} + (\varpi^{n} - 1 )\big( \xi_h^{n}, 1 \big)_{\rm M} \le 0.
	\end{aligned}
\end{equation}
The non-positivity in the last inequality follows from the facts that $\varpi^{n} \in [0,1]$, $\lambda_h^{n} \geq 0$, and $\xi_h^{n} \geq 0$.

Finally, inserting \eqref{cd:hat:3} into \eqref{cd:hat:1} gives the desired conclusion.
\end{proof}

The following lemma plays a key role in the error analysis of the projection scheme.
\begin{lemma}\label{lem:hatec:be} 
   Let $\widehat{C}_{h}^{n}$ and $C_{h}^{n}$ be the intermediate and projected concentration solutions of the BE-BCFD-P scheme \eqref{scm:be:u}--\eqref{scm:be:min}. Suppose $w_{h}^{n} \in \mal{C}$ satisfies
  \begin{equation}\label{bound_mass:e1}
    0 \leq w_{h}^{n} \leq 1, \quad \big(\beta^{n} w_{h}^{n},1 \big)_{\rm M}
    =\big(\beta^{n} \widehat{C}_{h}^{n},1 \big)_{\rm M}.
  \end{equation}
   Then, for $n=1,\ldots,N_{t}$, the following estimates hold:
	$$
	\|w_{h}^{n} - C_{h}^{n}\|_{{\rm M},\beta^{n}}^{2} 
      + \|(w_{h}^{n} - C_{h}^{n})-(w_{h}^{n} - \widehat{C}_{h}^{n})\|_{{\rm M},\beta^{n}}^{2} 
      \leq \|w_{h}^{n} - \widehat{C}_{h}^{n}\|_{{\rm M},\beta^{n}}^{2},
	$$	
    $$
    \big( \beta^{n}\big(w_{h}^{n} - C_{h}^{n}\big)-\beta^{n}\big(w_{h}^{n} - \widehat{C}_{h}^{n}\big), w_{h}^{n} - C_{h}^{n} \big)_{\rm M} \leq 0.
    $$
\end{lemma}
\begin{proof} It follows from the KKT condition \eqref{scm:be:k1} that
	\begin{equation*}\label{hatec:1}
		\beta^{n}\big(w_{h}^{n} - C_{h}^{n}\big)-\beta^{n}\big(w_{h}^{n} - \widehat{C}_{h}^{n}\big) = - \lambda_h^{n} + \xi_h^{n} + \gamma^{n} \beta^{n}.
	\end{equation*}
	Then, taking the discrete inner product with $w_{h}^{n} - C_{h}^{n}$ yields
	\begin{equation*}
		\begin{aligned}
		  & \f{1}{2} \Big( \|w_{h}^{n} - C_{h}^{n}\|_{{\rm M},\beta^{n}}^{2} + \| (w_{h}^{n} - C_{h}^{n})-(w_{h}^{n} - \widehat{C}_{h}^{n}) \|_{{\rm M},\beta^{n}}^{2} - \|w_{h}^{n} - \widehat{C}_{h}^{n}\|_{{\rm M},\beta^{n}}^{2} \Big) \\
        & \quad = \big( \beta^{n}\big(w_{h}^{n} - C_{h}^{n}\big)-\beta^{n}\big(w_{h}^{n} - \widehat{C}_{h}^{n}\big), w_{h}^{n} - C_{h}^{n} \big)_{\rm M} \\
		& \quad = \big( - \lambda_h^{n} + \xi_h^{n} + \gamma^{n} \beta^{n}, w_{h}^{n} - C_{h}^{n}  \big)_{\rm M}\\
		& \quad = - \big( \lambda_h^{n}, w_{h}^{n} \big)_{\rm M} + \big( \lambda_h^{n}, C_{h}^{n} \big)_{\rm M} + \big( \xi_h^{n}, w_{h}^{n} - 1 \big)_{\rm M} + \big( \xi_h^{n}, 1-C_{h}^{n} \big)_{\rm M} \\
		& \quad = - \big( \lambda_h^{n}, w_{h}^{n} \big)_{\rm M} + \big( \xi_h^{n}, w_{h}^{n} - 1 \big)_{\rm M} \leq 0,
		\end{aligned}
	\end{equation*}
	where, the KKT conditions \eqref{scm:be:k2}--\eqref{scm:be:k3} and assumption \eqref{bound_mass:e1} have been invoked.
\end{proof}
Define the errors $e_{c}^{n}:=c^{n}-C_{h}^{n}$, $\widehat{e}_{c}^{n}:=c^{n}-\widehat{C}_{h}^{n}$, $e_{p}^{n}:=p^{n}-P_{h}^{n}$ and $\bm{e}_{\bm{u}}^{n}:=(e_{u^{x}}^{n},e_{u^{y}}^{n})^{\top} =\bm{u}^{n}-\bm{U}_{h}^{n}$. It is clear that $e_{c}^{0}=\widehat{e}_{c}^{0}=0$. For numerical analysis, we introduce two "biased" error functions as follows:
$$
e_{\eta}^{n}:= c_{\eta}^{n} - C_{h}^{n}, \quad \widehat{e}_{\eta}^{n}:= c_{\eta}^{n} - \widehat{C}_{h}^{n}, 
$$
where the "biased" auxiliary concentration solution is defined as
\begin{equation}\label{c:eta:be}
    c_{\eta}^{n}: = \big( 1 + \alpha^{n} \big) c^{n}, 
    \quad 
    \alpha^{n}:= \varepsilon^{n} \chi(c^{n}),
    \quad
    \varepsilon^{n}:=\f{ \big(\beta^{n} C_{h}^{n},1 \big)_{\rm M} - \big(\beta^{n} c^{n},1 \big)_{\rm M} }{\big(\beta^{n} c^{n},\chi(c^{n}) \big)_{\rm M}}.
\end{equation}
Here, $\chi(w)$ is chosen as a smooth cut-off function (cf. \cite{XJ22,ZBY23,ZBY24}). For instance, we can take
\begin{equation*}
     \chi(w):= \left\lbrace
     \begin{aligned}
         & 1, & w \in [0,1/2-\eta), \\
         & 2(1-\eta-w), & w \in [1/2-\eta,1-\eta], \\
         & 0, & w \in ( 1-\eta,1],
     \end{aligned} \right.
 \end{equation*}
where $\eta<1/2$ is a given relatively small positive constant.  It is straightforward to verify that $\chi(w)\in [0,1]$ for all $w\in [0,1]$, and that $\chi$ is locally Lipschitz continuous with
\begin{equation}\label{chi:L}
    |\chi(w_{1}) - \chi(w_{2})| \le 2 |w_{1} -w_{2}|, \quad \forall w_{1}, w_{2} \in [0,1].
\end{equation}
Moreover, apart from the the trivial cases $c \equiv 1$ and $c \equiv 0$ (i.e., the fluid in the porous medium is either pure water or pure oil), we have $(\beta^{n} c^{n},\chi(c^{n}) )_{\rm M}> 0$. Therefore, we assume that there exists a positive constant $K_{*}$ such that for sufficiently small grid size $h$, $(\beta^{n} c^{n},\chi(c^{n}) )_{\rm M} \geq K_{*} > 0$. 
\begin{remark}\label{rem:hatec:be} By definition \eqref{c:eta:be} and the discrete mass conservation constraint \eqref{scm:be:k2}, it is easy to check that
$$
  \big( \beta^{n}c_{\eta}^{n}, 1 \big)_{\rm M} 
= \big( \beta^{n}c^{n}, 1 \big)_{\rm M} 
   + \f{ \big(\beta^{n} C_{h}^{n},1 \big)_{\rm M} - \big(\beta^{n} c^{n},1 \big)_{\rm M} }{ \big(\beta^{n} c^{n},\chi(c^{n}) \big)_{\rm M} } \big(\beta^{n} c^{n},\chi(c^{n}) \big)_{\rm M} 
= \big(\beta^{n} C_{h}^{n},1 \big)_{\rm M}
=\big(\beta^{n} \widehat{C}_{h}^{n},1 \big)_{\rm M}.
$$
Furthermore, from the definition of the smooth cut-off function $\chi(w)$, if $|\varepsilon^{n}| \leq \eta < 1/2$ (to be verified later), we have
$$
 0 \leq \frac{c^{n}}{2}
  \leq \big( 1 - \eta \chi(c^{n}) \big)c^{n}
  \leq c_{\eta}^{n} 
   \leq \big( 1 + \eta \chi(c^{n}) \big)c^{n} 
= 1 - \eta^{2} < 1, \quad \text{for}~c^{n} \in [0,1-\eta],
$$
and
$ c_{\eta}^{n} = c^{n} \in (1-\eta,1]$  for $c^{n} \in (1-\eta,1]$.
Hence, the auxiliary solution $c_{\eta}^{n}$ fulfills the two constraint conditions required in \eqref{bound_mass:e1}, i.e., 
\begin{equation*}
       0 \le c_{\eta}^{n} \le 1,\quad \big( \beta^{n}c_{\eta}^{n}, 1 \big)_{\rm M} =\big(\beta^{n} \widehat{C}_{h}^{n},1 \big)_{\rm M},
\end{equation*}
and therefore, Lemma \ref{lem:hatec:be} implies the following
projection error estimates:
\begin{equation}\label{aux:proj_err:be}
       \|e_{\eta}^{n}\|_{{\rm M},\beta^{n}}^{2} 
       + \|e_{\eta}^{n}-\widehat{e}_{\eta}^{n}\|_{{\rm M},\beta^{n}}^{2} 
       \leq \|\widehat{e}_{\eta}^{n}\|_{{\rm M},\beta^{n}}^{2},
       \quad 
     \big( \beta^{n} \big(e_{\eta}^{n}- \widehat{e}_{\eta}^{n}\big), e_{\eta}^{n} \big)_{\rm M} \leq 0,
\end{equation}
which are similar to the stability estimate in Lemma~\ref{lem:hat:be}.
\end{remark}

Next, we establish the following main lemma for the perturbation factor $\varepsilon^{n}$, which will serve as another key ingredient in the subsequent error analysis. 
\begin{lemma}\label{lem:alpha:be}
Under assumptions (A1)--(A2), there exists a positive constant $K$, independent of $h$ and $\tau$, such that for $n=1,\ldots,N_{t}$, the following estimates hold: 
$$
\begin{aligned}
    &|\varepsilon^{n}|_  \leq K\|e_{c}^{n-1}\|_{{\rm M},\phi} + K\big( \tau^{2} + h^{2} \big),\\
    &|\varepsilon^{n}|^{2} \leq K \tau \sum_{\ell=1}^{n-1} \|e_{c}^{\ell}\|_{{\rm M},\phi}^{2} + K\big( \tau^{2} + h^{4} \big),\\
    &|d_{\tau} \varepsilon^{n}| \leq K\|e_{c}^{n-1}\|_{{\rm M},\phi} + K\big( \tau + h^{2} \big).
\end{aligned}
$$
\end{lemma}
\begin{proof} In order to present estimates for $\varepsilon^{n}$, we observe from \eqref{mod:mass_conservation} that the exact solution $c^n$ satisfies
\begin{equation}\label{alpha:be:1}
    \begin{aligned}
        &\big( \phi c^{n},1 \big)_{\rm M} - \big( \phi c^{n-1},1 \big)_{\rm M} - \tau \big( [q_{P}c]^{n},1 \big)_{\rm M} = \tau \big( [q_{I} c_{I}]^{n},1 \big)_{\rm M} +  \tau \Theta^{n},  
    \end{aligned}
\end{equation}
where
$$
\Theta^{n}:= \big( \phi d_{\tau}c^{n} - \phi \p_{t}c^{n} ,1 \big)_{\rm M} + \big( \phi \p_{t}c^{n} - [q_{P}c + q_{I} c_{I}]^{n},1 \big)_{\rm M} - \int_{\Omega} \phi \p_{t} c^{n} -  [q_P c + q_I c_I]^{n}~{\rm d}\bm{x}.
$$
Then, subtracting the discrete mass conservation law \eqref{mass:be:2} from  \eqref{alpha:be:1}, we obtain an error equation in the discrete inner product form:
\begin{equation}\label{alpha:be:2}
   \big( \beta^{n}e_{c}^{n},1 \big)_{\rm M} = \big( \beta^{n}c^{n},1 \big)_{\rm M} - \big( \beta^{n}C_{h}^{n},1 \big)_{\rm M} = \big( \phi e_{c}^{n-1},1 \big)_{\rm M} + \tau \Theta^{n}.
\end{equation}
Furthermore, combining \eqref{alpha:be:2} with the definition of $\varepsilon^{n}$, it then follows from the numerical integration error of middle-point formula that 
\begin{equation}\label{alpha:be:3}
    |\varepsilon^{n}| 
    \leq  K \big| \big( \beta^{n} e_{c}^{n}, 1 \big)_{\rm M} \big| 
    \leq K \|e_{c}^{n-1}\|_{{\rm M},\phi} + K \big( \tau^{2} + h^{2} \big).
\end{equation}

Next, replacing $n$ by $\ell$ in both \eqref{alpha:be:1} and \eqref{mass:be:2} and summing from $\ell=1$ to $n~(n\ge 1)$  gives
\begin{align}
   \big( \beta^{n}c^{n},1 \big)_{\rm M} 
    & = \big( \phi c^{0}, 1 \big)_{\rm M} + \tau \sum_{\ell=1}^{n-1} \big( [q_{P} c]^{\ell}, 1 \big)_{\rm M} 
    + \tau \sum_{\ell=1}^{n}\left[  \big( [q_{I}c_{I}]^{\ell}, 1 \big)_{\rm M} +  \Theta^{\ell}\right],\label{alpha:be:4}\\
   \big( \beta^{n}C_{h}^{n}, 1 \big)_{\rm M} 
   & = \big( \phi C_{h}^{0}, 1 \big)_{\rm M} + \tau \sum_{\ell=1}^{n-1} \big( [q_{P}C_{h}]^{\ell}, 1\big)_{\rm M} 
   + \tau \sum_{\ell=1}^{n}\big( [q_{I} c_{I}]^{\ell}, 1 \big)_{\rm M}.\label{alpha:be:5}
 \end{align}
Then, subtracting \eqref{alpha:be:5} from \eqref{alpha:be:4} yields
\begin{equation*}
   \big( \beta^{n}e_{c}^{n},1 \big)_{\rm M} = \tau \sum_{\ell=1}^{n-1} \big( [q_{P} e_{c}]^{\ell}, 1 \big)_{\rm M} + \tau \sum_{\ell=1}^{n} \Theta^{\ell},
\end{equation*}
from which, similar to the derivation of \eqref{alpha:be:3}, we get another estimate for $\varepsilon^{n}$:
\begin{equation}\label{alpha:be:6}
   |\varepsilon^{n}| 
   \leq K \big| \big( \beta^{n} e_{c}^{n}, 1 \big)_{\rm M} \big| \leq K \tau \sum_{\ell=1}^{n-1} \|e_{c}^{\ell}\|_{{\rm M},\phi} + K\big( \tau + h^{2} \big).
\end{equation}
Thus, the inequality \eqref{alpha:be:6} directly indicates the second conclusion of this lemma. 

Finally, for $n \geq 2$, applying the operator $d_{\tau}$ to \eqref{alpha:be:4} and \eqref{alpha:be:5}, respectively, and then subtracting the resulting two equations, we obtain
\begin{equation*}
    d_{\tau} \big( \beta^{n} e_{c}^{n}, 1 \big)_{\rm M} = d_{\tau} \big( \beta^{n} c^{n}, 1 \big)_{\rm M} - d_{\tau} \big( \beta^{n} C_{h}^{n}, 1 \big)_{\rm M} = \big( [ q_{P} e_{c}]^{n-1}, 1 \big)_{\rm M} +\Theta^{n}. 
\end{equation*}
Again, analogous to the derivation of \eqref{alpha:be:3} and utilizing \eqref{alpha:be:2} and \eqref{chi:L}, we have
\begin{equation}\label{alpha:be:7}
   \begin{aligned}
     |d_{\tau}\varepsilon^{n}| 
       \leq K \left[ \big| d_{\tau}\big( \beta^{n} e_{c}^{n}, 1 \big)_{\rm M} \big| 
        +  \big|\big( \beta^{n-1} e_{c}^{n-1}, 1 \big)_{\rm M} \big| \right]
       \leq K \|e_{c}^{n-1}\|_{{\rm M},\phi} + K\big( \tau + h^{2} \big), \quad n \geq 2.
   \end{aligned}
\end{equation}
In particular, for $n=1$, note that $\varepsilon^{0}=e_{c}^{0}=0$, it then follows from \eqref{c:eta:be} and  \eqref{alpha:be:2} that
\begin{equation}\label{alpha:be:8}
   |d_{\tau}\varepsilon^{1}|
   = \frac{1}{\tau} | \varepsilon^{1}| 
   \leq \f{1}{K_{*}\tau} \big|\big( \beta^{1} e_{c}^{1}, 1 \big)_{\rm M} 
   \le K |\Theta^{1}|  \leq K\big( \tau + h^{2} \big).
\end{equation}
Therefore, \eqref{alpha:be:7}--\eqref{alpha:be:8} establish the third desired conclusion.
\end{proof}
\begin{remark}\label{rem:alpha} 
Notably, 
the relation between the prediction concentration error $\widehat{e}_{c}^{n}$ and the projection concentration error $e_{c}^{n}$ is not clear. To address this gap, we introduce a novel "biased" auxiliary solution $c_{\eta}^{n}$ in \eqref{c:eta:be}, which satisfies both bound preservation and discrete mass conservation, see Remark \ref{rem:hatec:be}. Most importantly, we establish a novel non-expansive result for the weighted $L^{2}$-projection \eqref{scm:be:min}, as stated in Lemma \ref{lem:hatec:be}. This result enables us to relate the errors before and after projection, provided that the auxiliary solution $w_h^n$ satisfies the two constraint conditions in \eqref{bound_mass:e1}. It should be noted, however, that the exact solution $c^{n}$ itself only guarantees the bound-preserving property, while the discrete mass conservation constraint $G(c^{n}) = 0$ does not hold. 

Moreover, in contrast to the model problems in \cite{CYY24,CYY26,ZZF26}, where the mass is constant in time, the mass of the present model is maintained by an injection-production balance and hence varies with time. Consequently, the "biased" auxiliary solution introduced herein involves the unknown error $e_{c}^{n}$ (see \eqref{c:eta:be}) and 
$$
\big| c_{\eta}^{n} - c^{n} \big| \leq \f{ \|\beta^{n}\|_{\rm M} }{K_{*}} \|e_{c}^{n}\|_{\rm M} \le K\|e_{c}^{n}\|_{\rm M},
$$ 
whereas $c_{\eta}^{n}$ in \cite{CYY24,CYY26,ZZF26} contains only known numerical integration errors of the exact solution (see (3.4) of \cite{CYY24}) and satisfy 
$$
\big| c_{\eta}^{n} - c^{n} \big| 
=  \f{|c^{n}|}{\big(c^{n},1\big)_{\rm M}} \Big| \int_{\Omega} c^{n}~{\rm d} \bm{x} - \big(c^{n},1\big)_{\rm M} \Big| 
\leq K h^{2}. 
$$
Hence, their framework allows direct estimation of $e_{\eta}^{n}$, $\widehat{e}_{\eta}^{n}$ and the recovery of $e_{c}^{n}$ via 
$$
\|e_{c}^{n}\|_{\rm M} 
\leq \|e_{\eta}^{n}\|_{\rm M} + \|c_{\eta}^{n} - c^{n}\|_{\rm M} 
\leq \|e_{\eta}^{n}\|_{\rm M} + Kh^2.
$$
However, the present work with error equation \eqref{err:be:3} couples $e_{\eta}^{n}$, $\widehat{e}_{\eta}^{n}$ with $e_{c}^{n}$, which prevents us from directly following their approach to derive the corresponding error estimates. Furthermore, we show that $\varepsilon^{n}$, defined in \eqref{c:eta:be}, is merely a perturbation factor and satisfies the error estimates established in Lemma \ref{lem:alpha:be}; these estimates will serve as a key ingredient in the main error analysis.  
\end{remark}

From \eqref{mod} and \eqref{scm:be:u}--\eqref{scm:be:c}, we obtain the following error equations
\begin{equation}\label{err:be:1}
  \nabla_{h} \cdot \bm{e}_{\bm{u}}^{n} = \epsilon_{1}^{n},	
\end{equation} 
\begin{equation}\label{err:be:2}
  \big[ \bm{\mal{A}}(c^{n-1}) - \bm{\mal{A}}(C_{h}^{n-1}) \big]\bm{u}^{n} + \bm{\mal{A}}(C_{h}^{n-1}) \bm{e}_{\bm{u}}^{n} + \nabla_{h} e_{p}^{n} = \bm{\epsilon}_{\bm{2}}^{n},
\end{equation}
\begin{equation}\label{err:be:3}
   \big[ \phi D_{\tau}\widehat{e}_{\eta} + \nabla_{h} \cdot ( \bm{\mal{B}}(\widehat{e}_{\eta})\bm{U}_{h} - \bm{\mal{D}}(\bm{U}_{h}) \nabla_{h}\widehat{e}_{\eta}  + \bm{e}_{\bm{u}} c + \bm{\epsilon}_{\bm{3}} - \bm{\zeta}_{\bm{1}}  ) - q_{P}\widehat{e}_{\eta} \big]^{n} = \epsilon_{4}^{n} + \zeta_{2}^{n},
\end{equation}
where the local truncation errors
$$
\begin{aligned}
    \epsilon_{1}^{n}&:= \nabla_{h} \cdot \bm{u}^{n} - \nabla \cdot \bm{u}^{n},\\ 
    \bm{\epsilon}_{\bm{2}}^{n}&:=\bm{\mal{A}}(c^{n-1})\bm{u}^{n}  - a(c^{n})\bm{u}^{n} + \nabla_{h}p^{n} - \nabla p^{n},\\
	\bm{\epsilon}_{\bm{3}}^{n}&:= \big[ \bm{U}_{h}c - D\nabla c  - \bm{\mal{B}}(c)\bm{U}_{h} + \bm{\mal{D}}(\bm{U}_{h}) \nabla_{h}c\big]^{n},\\
	\epsilon_{4}^{n}&:=\big[ \phi (d_{\tau}c - \p_{t} c) + \nabla_{h} \cdot (\bm{u}c - D \nabla c) - \nabla \cdot (\bm{u}c - D \nabla c)\big]^{n}, \\
	\bm{\zeta}_{\bm{1}}^{n}&:= \big[ \bm{\mal{B}}(\alpha c)\bm{U}_{h} -  \bm{\mal{D}}(\bm{U}_{h}) \nabla_{h} (\alpha c ) \big]^{n},\quad \zeta_{2}^{n}:= \big[ \phi  d_{\tau} (\alpha c) - q_{P} (\alpha c) \big]^{n}.
\end{aligned}
$$
Under the assumptions (A1)--(A2), and using \eqref{chi:L} and Lemmas \ref{lem:t:err}--\ref{lem:d:err} and \ref{lem:alpha:be}, it is easy to see
\begin{equation}\label{err:truc:be1}
   \begin{aligned}
       \|\epsilon_{1}^{n}\|_{\rm M}  \leq Kh^{2}, \quad
       \|\bm{\epsilon}_{\bm{2}}^{n}\|_{\rm T} + \|\epsilon_{4}^{n}\|_{\rm M} \leq K\big(\tau + h^{2}\big),  \quad 
       \|\zeta_{2}^{n}\|_{\rm M} \leq  K\|e_{c}^{n-1}\|_{{\rm M},\phi} + K\big(\tau + h^{2}\big).
   \end{aligned}
\end{equation}
Furthermore, if the numerical velocity $\bm{U}_{h}^{n}$ is uniformly bounded (see \eqref{assum:be} below), together with Lemmas \ref{lem:upwind:err} and \ref{lem:alpha:be}, we obtain
\begin{equation}\label{err:truc:be2}
   \begin{aligned}
       \|\bm{\epsilon}_{\bm{3}}^{n}\|_{\rm T}  \leq Kh^{2}, \quad
       \|\bm{\zeta}_{\bm{1}}^{n}\|_{\rm T} \leq K\|e_{c}^{n-1}\|_{{\rm M},\phi} + K\big(\tau^{2} + h^{2}\big).
   \end{aligned}
\end{equation}
Therefore, in the following error analysis, we shall first establish the boundedness of the numerical velocity $\bm{U}_{h}^{n}$ before invoking the estimates in \eqref{err:truc:be2}.

\begin{lemma}\label{lem:err:be:pu} 
Under assumptions (A1)--(A2), the system \eqref{scm:be:u} at time $t=t^{n}$ admits a unique solution pair $\{P_{h}^{n},\bm{U}_{h}^{n}\}$ for given $C_{h}^{n-1}$. Moreover, there exists a positive constant $K$, independent of $h$ and $\tau$, such that for all $n=1,\ldots,N_t$, the following estimate holds:
\begin{equation*}
	\|e_{p}^{n}\|_{\rm M}^{2} + \|\nabla_{h}e_{p}^{n}\|_{\rm T}^{2} 
    + \| \bm{e}_{\bm{u}}^{n} \|_{\rm T}^{2}
	\leq  K\| e_{c}^{n-1} \|_{{\rm M},\phi}^{2} + K\big( \tau^{2} + h^{4}\big).
\end{equation*}
\end{lemma}
\begin{proof} 
The proof is divided into two parts. In the first part, we establish the existence and uniqueness of the solution, and in the second, we derive the desired error estimate.

\textbf{Part 1: Existence and uniqueness.}   
It suffices to show that the homogeneous system corresponding to \eqref{scm:be:u} admits only trivial solutions. To this end, let $\sigma_{h} \in \mal{C}$ and $ \bm{v}_{h} \in \mal{U}\times\mal{V}$ satisfy
\begin{equation}\label{exit:be:pu1}
  \nabla_{h} \cdot \bm{v}_{h} = 0, \quad
	\bm{\mal{A}}(C_{h}^{n-1}) \bm{v}_{h} + \nabla_{h} \sigma_{h} = 0.
\end{equation}
Eliminating $\bm{v}_h$ from \eqref{exit:be:pu1} (i.e., solving the second equation for $\bm{v}_h$ and substituting into the first), then taking the inner product of the resulting equation with $\sigma_{h}$ and applying Lemma \ref{lem:dis}, we obtain
\begin{equation*}
	- \left( \nabla_{h} \cdot \bigl(\bm{\mal{A}}^{-1}(C_{h}^{n-1}) \nabla_{h} \sigma_{h}\bigr), \sigma_{h} \right)_{\rm M} 
    = \left(  \bm{\mal{A}}^{-1}(C_{h}^{n-1}) \nabla_{h} \sigma_{h}, \nabla_{h} \sigma_{h} \right)_{\rm T} = 0,
\end{equation*}
which implies $\nabla_h \sigma_h \equiv 0$. Combined with the condition $\sigma_{1,1}=0$ gives $\sigma_h \equiv 0$. Substituting this back into \eqref{exit:be:pu1} also yields $\bm{v}_h \equiv \bm{0}$. Hence, the uniquely existence of solutions $P_{h}^{n} \in \mal{C}$ and $\bm{U}_{h}^{n} \in \mal{U}\times \mal{V}$ to \eqref{scm:be:u} is proved.

\textbf{Part 2: Error estimate.}
Taking the discrete inner products of the error equations \eqref{err:be:1} and \eqref{err:be:2} with $e_{p}^{n}$ and $\bm{e}_{\bm{u}}^{n}$ in the sense of $(\cdot,\cdot)_{\rm M}$ and $(\cdot,\cdot)_{\rm T}$, respectively, and adding the resulting equations yields
\begin{equation}\label{erru:be:1}
	\left(  \bm{\mal{A}}(C_{h}^{n-1}) \bm{e}_{\bm{u}}^{n}, \bm{e}_{\bm{u}}^{n} \right)_{\rm T}  = \left( \epsilon_{1}^{n}, e_{p}^{n} \right)_{\rm M} -  \left( \left(  \bm{\mal{A}}(c^{n-1}) - \bm{\mal{A}}(C_{h}^{n-1}) \right)  \bm{u}^{n} - \bm{\epsilon}_{\bm{2}}^{n}, \bm{e}_{\bm{u}}^{n} \right)_{\rm T} =:I_{1},
\end{equation}
where, Lemma \ref{lem:dis} is applied to the left-hand side of \eqref{erru:be:1}, i.e.,
\begin{equation*}
	\left( \nabla_{h} \cdot \bm{e}_{\bm{u}}^{n}, e_{p}^{n} \right)_{\rm M} + \left(  \nabla_{h} e_{p}^{n}, \bm{e}_{\bm{u}}^{n} \right)_{\rm T} = - \left( \bm{e}_{\bm{u}}^{n}, \nabla_{h} e_{p}^{n} \right)_{\rm T} + \left(  \nabla_{h} e_{p}^{n}, \bm{e}_{\bm{u}}^{n} \right)_{\rm T} = 0.
\end{equation*}

Assumptions (A1)--(A2), along with \eqref{assump:a}, implies that
\begin{equation}\label{erru:be:2}
	\| \left(  \bm{\mal{A}}(c^{n-1}) - \bm{\mal{A}}(C_{h}^{n-1}) \right)  \bm{u}^{n} \|_{\rm T} \leq  K\| e_{c}^{n-1} \|_{{\rm M},\phi},\quad
    \|\bm{\mal{A}}(C_{h}^{n-1}) \bm{e}_{\bm{u}}^{n}\|_{\rm T} \leq a^{*}\| \bm{e}_{\bm{u}}^{n} \|_{\rm T} . 
\end{equation}
Furthermore, utilizing the error equation \eqref{err:be:2} and the results \eqref{erru:be:2}, we get
\begin{equation}\label{erru:be:3}
	\begin{aligned}
		\| \nabla_{h} e_{p}^{n} \|_{\rm T}
		& = \| \bm{\mal{A}}(c^{n-1})\bm{u}^{n} - \bm{\mal{A}}(C_{h}^{n-1})\bm{u}^{n} + \bm{\mal{A}}(C_{h}^{n-1}) \bm{e}_{\bm{u}}^{n} - \bm{\epsilon}_{\bm{2}}^{n}\|_{\rm T}\\ 
		& \leq a^{*}\| \bm{e}_{\bm{u}}^{n} \|_{\rm T}  + K\| e_{c}^{n-1} \|_{{\rm M},\phi} + \| \bm{\epsilon}_{\bm{2}}^{n}\|_{\rm T}.
	\end{aligned}
\end{equation}
Then, invoking Lemma \ref{lem:seminorm}, the estimates \eqref{err:truc:be1}, \eqref{erru:be:2} and \eqref{erru:be:3},  the term $I_{1}$ can be bounded by
\begin{equation*}
	\begin{aligned}
		I_{1}
        &\leq \|e_{p}^{n} \|_{\rm M} \|\epsilon_{1}^{n}\|_{\rm M} + \| \bm{\mal{A}}(c^{n-1}) \bm{u}^{n} - \bm{\mal{A}}(C_{h}^{n-1})  \bm{u}^{n} - \bm{\epsilon}_{\bm{2}}^{n}\|_{\rm T} \|\bm{e}_{\bm{u}}^{n} \|_{\rm T} \\
		&\leq K \|\nabla_{h} e_{p}^{n} \|_{\rm T} \|\epsilon_{1}^{n}\|_{\rm M} + \| \bm{\mal{A}}(c^{n-1}) \bm{u}^{n} - \bm{\mal{A}}(C_{h}^{n-1})  \bm{u}^{n} - \bm{\epsilon}_{\bm{2}}^{n}\|_{\rm T} \|\bm{e}_{\bm{u}}^{n} \|_{\rm T} \\
		& \leq \f{a_{*}}{2} \| \bm{e}_{\bm{u}}^{n}\|_{\rm T}^{2} + K\| e_{c}^{n-1} \|_{{\rm M},\phi}^{2} + K \big( \tau^{2} + h^{4}\big).
	\end{aligned}
\end{equation*}
Finally, substituting the above estimate $I_{1}$ into \eqref{erru:be:1} and using \eqref{assump:a} yields
\begin{equation}\label{err:be:u}
    \| \bm{e}_{\bm{u}}^{n} \|_{\rm T}^{2} 
    \leq K\| e_{c}^{n-1} \|_{{\rm M},\phi}^{2} + K \big( \tau^{2} + h^{4}\big).
\end{equation}

Moreover, it follows from Lemma \ref{lem:seminorm} and the results \eqref{erru:be:3} and \eqref{err:be:u} that
\begin{equation}\label{err:be:p}
	\|e_{p}^{n}\|_{\rm M}^{2} + \|\nabla_{h}e_{p}^{n}\|_{\rm T}^{2} 
    \leq K \|\nabla_{h}e_{p}^{n}\|_{\rm T}^{2}
	\leq  K\| e_{c}^{n-1} \|_{{\rm M},\phi}^{2} + K \big( \tau^{2} + h^{4}\big).
\end{equation}
The proof is thereby completed by the estimates derived in \eqref{err:be:u}--\eqref{err:be:p}. 
\end{proof}

Based on the fact that $e_{c}^{0}=0$, along with Lemma \ref{lem:err:be:pu}, the following conclusion is established.
\begin{corollary}\label{col:err:be:u1} Under the assumptions of Lemma \ref{lem:err:be:pu}, the system \eqref{scm:be:u} at time $t=t^{1}$ admits a unique solution pair $\{P^{1},\bm{U}^{1}\}$, and there exists a constant $K>0$, independent of $h$ and $\tau$, such that
	\begin{equation*}
		\|e_{p}^{1}\|_{\rm M}^{2} + \|\nabla_{h}e_{p}^{1}\|_{\rm T}^{2} + \| \bm{e}_{\bm{u}}^{1} \|_{\rm T}^{2} \leq K\big( \tau^{2} + h^{4}\big).
	\end{equation*}
\end{corollary}

The following lemma provides an approach to address the difficulty mentioned in Remark \ref{rem:alpha}, and enables us to estimate both the "biased" error functions $e_{\eta}^{n}$, $\widehat{e}_{\eta}^{n}$ and the original error functions $e_{c}^{n}$, $\widehat{e}_{c}^{n}$, simultaneously.
\begin{lemma}\label{lem:be:ec}
Suppose that $|\varepsilon^{n}| \leq \eta < 1/2$. Then, there exists a positive constant $K$, independent of $h$ and $\tau$, such that for all $n=1,\ldots,N_t$, the following estimate holds:
  \begin{equation*}
     \begin{aligned}
         \big( \phi D_{\tau} \widehat{e}_{\eta}^{n},  \widehat{e}_{\eta}^{n} \big)_{\rm M} 
            - \left( [q_{P}\widehat{e}_{\eta}]^{n}, \widehat{e}_{\eta}^{n} \right)_{\rm M} 
            & \geq \f{1}{4}d_{\tau} 
     \Bigl[\|\widehat{e}_{\eta}^{n}\|_{{\rm M},\beta^{n}}^{2} 
         +\|e_{c}^{n}\|_{{\rm M},\phi}^{2} 
         + \|[\alpha c]^{n}\|_{{\rm M},\phi}^{2}  
         + 2\big( \phi e_{c}^{n}, [\alpha c]^{n} \big)_{\rm M}\Bigr]
               \\
            & \qquad - K \|\widehat{e}_{\eta}^{n}\|_{{\rm M},\beta^{n}}^{2}  
            - K \|e_{c}^{n-1}\|_{{\rm M},\phi}^{2} 
            - K \big( \tau^{2} + h^{4} \big).
        \end{aligned}
    \end{equation*}
\end{lemma}
\begin{proof} On one hand, note that $\widehat{e}_{\eta}^{n} = \widehat{e}_{c}^{n} + [\alpha c]^{n}$, $e_{\eta}^{n-1}=e_{c}^{n-1} + [\alpha c]^{n-1}$ and the identity $\widehat{e}_{c}^{n}-e_{c}^{n}= \widehat{e}_{\eta}^{n}-e_{\eta}^{n}$. Then, using the fact $q_{P}^{n} < 0$, we have 
\begin{equation}\label{be:ec:2}
  \begin{aligned}
     &  \big( \phi D_{\tau} \widehat{e}_{\eta}^{n},  \widehat{e}_{\eta}^{n} \big)_{\rm M} - \left( [q_{P}\widehat{e}_{\eta}]^{n}, \widehat{e}_{\eta}^{n} \right)_{\rm M} \ge 
     \big( \phi D_{\tau} \widehat{e}_{\eta}^{n},  \widehat{e}_{\eta}^{n} \big)_{\rm M} \\
     & \quad = \big( \phi d_{\tau} e_{c}^{n}  , e_{c}^{n} + \widehat{e}_{\eta}^{n} -e_{\eta}^{n} + [\alpha c]^{n}\big)_{\rm M} + \f{1}{\tau}\big( \phi ( \widehat{e}_{\eta}^{n} - e_{\eta}^{n}) , e_{\eta}^{n} + \widehat{e}_{\eta}^{n} - e_{\eta}^{n} \big)_{\rm M} + \big( \phi d_{\tau} [\alpha c]^{n},  \widehat{e}_{c}^{n} + [\alpha c]^{n} \big)_{\rm M}  \\
       & \quad = \big( \phi d_{\tau} e_{c}^{n}, e_{c}^{n} \big)_{\rm M} + \big( \phi d_{\tau} e_{c}^{n}, \widehat{e}_{\eta}^{n}-e_{\eta}^{n} \big)_{\rm M} + d_{\tau}\big( \phi e_{c}^{n}, [\alpha c]^{n} \big)_{\rm M} + \tau \big( \phi D_{\tau} \widehat{e}_{c}^{n}, d_{\tau}[\alpha c]^{n} \big)_{\rm M}  \\
       &\qquad  + \f{1}{\tau}  \big( \phi \big(\widehat{e}_{\eta}^{n}-e_{\eta}^{n}\big), e_{\eta}^{n} \big)_{\rm M}  + \f{1}{\tau} \|\widehat{e}_{\eta}^{n}-e_{\eta}^{n}\|_{{\rm M},\phi}^{2} + \big( \phi d_{\tau} [\alpha c]^{n}, [\alpha c]^{n} \big)_{\rm M} 
       \\
       &\quad =: I_{2} + d_{\tau}\big( \phi e_{c}^{n}, [\alpha c]^{n} \big)_{\rm M} + \tau \big( \phi D_{\tau} \widehat{e}_{c}^{n}, d_{\tau}[\alpha c]^{n} \big)_{\rm M} + \f{1}{\tau}  \big( \phi \big(\widehat{e}_{\eta}^{n}-e_{\eta}^{n}\big), e_{\eta}^{n} \big)_{\rm M},
    \end{aligned}
\end{equation}
where 
\begin{equation*}
   \begin{aligned}
       I_{2}& := \big( \phi d_{\tau} e_{c}^{n}, e_{c}^{n} \big)_{\rm M} + \big( \phi d_{\tau} e_{c}^{n}, \widehat{e}_{\eta}^{n}-e_{\eta}^{n} \big)_{\rm M}  + \f{1}{\tau} \|\widehat{e}_{\eta}^{n}-e_{\eta}^{n}\|_{{\rm M},\phi}^{2}  + \big( \phi d_{\tau} [\alpha c]^{n}, [\alpha c]^{n} \big)_{\rm M} \\
       & \geq \f{\|e_{c}^{n}\|_{{\rm M},\phi}^{2} + \|e_{c}^{n}-e_{c}^{n-1}\|_{{\rm M},\phi}^{2} - \|e_{c}^{n-1}\|_{{\rm M},\phi}^{2}}{2\tau} 
       - \f{ \|e_{c}^{n}-e_{c}^{n-1}\|_{{\rm M},\phi}^{2} 
            + \|\widehat{e}_{\eta}^{n}-e_{\eta}^{n}\|_{{\rm M},\phi}^{2} }{2\tau} \\
       & \qquad + \f{\|\widehat{e}_{\eta}^{n}-e_{\eta}^{n}\|_{{\rm M},\phi}^{2}}{\tau} 
       + \f{ \|[\alpha c]^{n}\|_{{\rm M},\phi}^{2} 
       - \|[\alpha c]^{n-1}\|_{{\rm M},\phi}^{2} }{2\tau} \\
       & \geq \f{\|e_{c}^{n}\|_{{\rm M},\phi}^{2} - \|e_{c}^{n-1}\|_{{\rm M},\phi}^{2}}{2\tau} 
         + \f{ \|[\alpha c]^{n}\|_{{\rm M},\phi}^{2} - \|[\alpha c]^{n-1}\|_{{\rm M},\phi}^{2} }{2\tau}.
   \end{aligned}
\end{equation*}
Moreover, it follows from Lemma \ref{lem:alpha:be} that
\begin{equation*}
   \begin{aligned}
       \tau \big( \phi D_{\tau} \widehat{e}_{c}^{n}, d_{\tau}[\alpha c]^{n} \big)_{\rm M} 
       &\ge - \big( \|\widehat{e}_{\eta}^{n} - [\alpha c]^{n}\|_{{\rm M},\phi} + \|e_{c}^{n-1}\|_{{\rm M},\phi} \big)  \|d_{\tau}[\alpha c]^{n}\|_{{\rm M},\phi} \\
       & \ge - K \|\widehat{e}_{\eta}^{n}\|_{{\rm M},\beta^{n}}^{2} 
       - K \|e_{c}^{n-1}\|_{{\rm M},\phi}^{2} 
       - K \big( \tau^{2} + h^{4} \big).
   \end{aligned}
\end{equation*}
Furthermore, by Remark \ref{rem:hatec:be}, if $|\varepsilon^{n}| \leq \eta < 1/2$, then the estimate \eqref{aux:proj_err:be} holds and we further obtain
\begin{equation*}
   \begin{aligned}
      \f{1}{\tau}  \big( \phi \big(\widehat{e}_{\eta}^{n}-e_{\eta}^{n}\big), e_{\eta}^{n} \big)_{\rm M} & = \f{1}{\tau}  \big( \beta^{n} \big(\widehat{e}_{\eta}^{n}-e_{\eta}^{n}\big), e_{\eta}^{n} \big)_{\rm M} +  \big( q_{P}^{n} \big(\widehat{e}_{\eta}^{n}-e_{\eta}^{n}\big), e_{\eta}^{n} \big)_{\rm M} \\
      & \geq \big( q_{P}^{n} \big(\widehat{e}_{\eta}^{n}-e_{\eta}^{n}\big), e_{\eta}^{n} \big)_{\rm M}
      \geq - K \|\widehat{e}_{\eta}^{n}\|_{{\rm M},\beta^{n}}^{2}.
   \end{aligned}
\end{equation*}
Therefore, substituting the above estimates into \eqref{be:ec:2} yields the lower bound
\begin{equation} \label{lowbound:e1}
     \begin{aligned}
         \big( \phi D_{\tau} \widehat{e}_{\eta}^{n},  \widehat{e}_{\eta}^{n} \big)_{\rm M} 
            - \left( [q_{P}\widehat{e}_{\eta}]^{n}, \widehat{e}_{\eta}^{n} \right)_{\rm M} 
            & \geq \f{1}{2} d_{\tau} \Bigl[\|e_{c}^{n}\|_{{\rm M},\phi}^{2} 
               +   \|[\alpha c]^{n}\|_{{\rm M},\phi}^{2} + 2\big( \phi e_{c}^{n}, [\alpha c]^{n} \big)_{\rm M} \Bigr] 
               \\
            & \qquad - K \|\widehat{e}_{\eta}^{n}\|_{{\rm M},\beta^{n}}^{2}  
            - K \|e_{c}^{n-1}\|_{{\rm M},\phi}^{2} 
            - K \big( \tau^{2} + h^{4} \big).
        \end{aligned}
    \end{equation}

On the other hand, using the projection error estimate in \eqref{aux:proj_err:be} and the fact that $q_{P}^{n}$, $q_{P}^{n-1}<0$, we derive another lower bound:
\begin{equation}\label{lowbound:e2}
\begin{aligned}
       & \left( \phi D_{\tau} \widehat{e}_{\eta}^{n}, \widehat{e}_{\eta}^{n} \right)_{\rm M} -  \left( [q_{P}\widehat{e}_{\eta}]^{n}, \widehat{e}_{\eta}^{n} \right)_{\rm M} \\ 
       & \quad = \f{1}{2\tau} \Big( \|\widehat{e}_{\eta}^{n}\|_{{\rm M},\phi}^{2} + \|\widehat{e}_{\eta}^{n}-e_{\eta}^{n-1}\|_{{\rm M},\phi}^{2} - \|e_{\eta}^{n-1}\|_{{\rm M},\phi}^{2} \Big) -  \left( [q_{P}\widehat{e}_{\eta}]^{n}, \widehat{e}_{\eta}^{n} \right)_{\rm M} \\
       & \quad  = \f{1}{2\tau} \Big( \|\widehat{e}_{\eta}^{n}\|_{{\rm M},\beta^{n}}^{2} 
            + \|\widehat{e}_{\eta}^{n}-e_{\eta}^{n-1}\|_{{\rm M},\phi}^{2} 
            - \|e_{\eta}^{n-1}\|_{{\rm M},\beta^{n-1}}^{2} \Big) 
       - \f{1}{2} \left( [q_{P}\widehat{e}_{\eta}]^{n}, \widehat{e}_{\eta}^{n} \right)_{\rm M} 
       - \f{1}{2} \left( [q_{P}e_{\eta}]^{n-1}, e_{\eta}^{n-1} \right)_{\rm M} \\
       &  \quad 
       \geq \f{1}{2\tau}\Big( \|\widehat{e}_{\eta}^{n}\|_{{\rm M},\beta^{n}}^{2}  - \|\widehat{e}_{\eta}^{n-1}\|_{{\rm M},\beta^{n-1}}^{2} \Big).
    \end{aligned}
\end{equation}
Thus, adding the above two estimates \eqref{lowbound:e1} and \eqref{lowbound:e2} together, we obtain the desired result.
\end{proof}

The main error estimate of this paper is stated as follows.
\begin{theorem}\label{thm:err:be} Under assumptions (A1)--(A2), the BE-BCFD-P scheme \eqref{scm:be:u}--\eqref{scm:be:min} admits a unique solution triple $\{P_{h}^{n},\bm{U}_{h}^{n},C_{h}^{n}\}$, provided that 
$\tau \leq \tau_{*}:= \f{\phi_{*}D_{*}}{8(K_{e}+1)^{2}}$, 
$h\leq h_{*}:=\f{2D_{*}}{K_{e}+1}$, and the time-step condition
\begin{equation}\label{cond:be}
	\tau = o(h)
\end{equation}
is satisfied. Furthermore, there exists a positive constant $K$, independent of $h$ and $\tau$, such that for all $n=1,\ldots, N_t$, the following estimate holds:
\begin{equation*}
\begin{aligned}
 \|e_{c}^{n}\|_{\rm M}^{2}   
     + \|e_{p}^{n}\|_{\rm M}^{2} + \|\nabla_{h} e_{p}^{n}\|_{\rm T}^{2} 
     + \| \bm{e}_{\bm{u}}^{n} \|_{\rm T}^{2} 
    \leq K\big( \tau^{2}  +  h^{4}\big).
\end{aligned}
\end{equation*}
\end{theorem}
\begin{proof} We will establish the error estimates by the mathematical induction 
\begin{equation}\label{assum:be}
	\| \bm{U}^{\ell} \|_{\infty} \leq K_{e} + 1, \quad 
	| \varepsilon^{\ell} | \leq \eta.
\end{equation}
To begin, for sufficiently small $\tau$ and $h$ satisfying the condition \eqref{cond:be}, the inverse inequality together with Corollary \ref{col:err:be:u1}, \eqref{alpha:be:3}, and the fact that $e_{c}^{0}=0$, we obtain 
\begin{equation*}\label{bound:u}
  \begin{aligned}
    &\| \bm{U}^{1}\|_{\infty} \leq \| \bm{u}^{1}\|_{\infty} + \| \bm{e}_{\bm{u}}^{1}\|_{\infty} \leq K_{e} + Kh^{-1}(\tau  + h^2)  \leq K_{e} + 1,\\
    &| \varepsilon^{1} | \leq K\|e_{c}^{0}\|_{{\rm M},\phi} + K  \big(\tau^{2} +  \tau h^{2} \big)   \leq K\big(\tau^{2} +  h^{2} \big)  \leq \eta.
  \end{aligned}
\end{equation*}

Now, assume that \eqref{assum:be} holds for all $\ell \leq n$. As noted in Remark \ref{rem:hatec:be}, this assumption implies that the estimates \eqref{err:truc:be2} and \eqref{aux:proj_err:be} are valid. Next, we shall first establish the desired estimate by following the proof of Lemma \ref{lem:err:be:pu}, which consists of two parts: existence and uniqueness, and the error estimate; Subsequently, we complete the induction for $\ell=n+1$ in Part 3.

\medskip
\textbf{Part 1: Existence and uniqueness.} On the one hand, Lemma \ref{lem:err:be:pu} establishes the unique solvability and error estimate for the system \eqref{scm:be:u}. On the other hand, since the correction step in Step 3 is a convex optimization problem, it follows from \eqref{be:lm} in Remark \ref{rem:lm:be} that, for a given $\widehat{C}_{h}^{n}$, the numerical concentration $C_{h}^{n}$ exists and is unique. Therefore, it remains only to prove the existence and uniqueness of the intermediate concentration $\widehat{C}_{h}^{n}$ that satisfies the linear system \eqref{scm:be:c}. 

Analogous to Lemma \ref{lem:err:be:pu}, it suffices to show that the corresponding homogeneous system of \eqref{scm:be:c} admits only the trivial solution $w_{h} \in \mal{C}$:
\begin{equation*}
	\phi w_{h} + \tau \nabla_{h} \cdot \left( \bm{\mal{B}}(w_{h}) \bm{U}_{h}^{n} - \bm{\mal{D}}(\bm{U}_{h}^{n}) \nabla_{h} w_{h}  \right) - \tau q_{P}^{n} w_{h} = 0.
\end{equation*}	
Taking the inner product of the above equation with $w_{h}$ yields
\begin{equation}\label{exit:be:c1}
	\left( \phi w_{h}, w_{h} \right)_{\rm M} + \tau \left( \nabla_{h} \cdot \left( \bm{\mal{B}}(w_{h}) \bm{U}_{h}^{n} - \bm{\mal{D}}(\bm{U}_{h}^{n}) \nabla_{h} w_{h} \right), w_{h} \right)_{\rm M} - \tau \left( q_{P}^{n} w_{h}, w_{h} \right)_{\rm M} 
    = 0.
\end{equation}
For the second term on the left-hand side of \eqref{exit:be:c1}, by applying Lemma \ref{lem:dis}, assumption (A2), and the boundedness of $\bm{U}_{h}^{n}$ from \eqref{assum:be}, we have
\begin{equation}\label{exit:be:c2}
	\begin{aligned}
		\left( \nabla_{h} \cdot \left( \bm{\mal{B}}(w_{h}) \bm{U}_{h}^{n} - \bm{\mal{D}}(\bm{U}_{h}^{n}) \nabla_{h} w_{h} \right), w_{h} \right)_{\rm M} 
		& = - \left( \bm{\mal{B}}(w_{h}) \bm{U}_{h}^{n} - \bm{\mal{D}}(\bm{U}_{h}^{n}) \nabla_{h} w_{h}, \nabla_{h} w_{h} \right)_{\rm T} \\
		& \geq \f{1}{2} \|\nabla_{h} w_{h}\|_{{\rm T},\bm{\mal{D}}(\bm{U}_{h}^{n})}^{2} -  \f{2(K_{e}+1)^{2}}{\phi_{*} D_{*} } \left( 1 + \f{K_{e}+1}{2D_{*}} h \right)   \|w_{h}\|_{{\rm M},\phi}^{2}.
	\end{aligned}
\end{equation}
Therefore, for $\tau \leq \tau_{*}$ and $h\leq h_{*}$, we obtain from \eqref{exit:be:c1}--\eqref{exit:be:c2} and the fact $q_P^{n}<0$ that 
\begin{equation*}
    \|w_{h}\|_{{\rm M},\phi}^{2} + \tau\|\nabla_{h} w_{h}\|_{{\rm T},\bm{\mal{D}}(\bm{U}_{h}^{n})}^{2} \leq 0,
\end{equation*}
which yields $w_{h} \equiv 0$. Hence, the solution triple of the BE-BCFD-P scheme \eqref{scm:be:u}--\eqref{scm:be:min} is guaranteed to be uniquely exist.	 

\medskip
\textbf{Part 2: Error estimate.} Taking the inner product of the error equation \eqref{err:be:3} with $\widehat{e}_{\eta}^{n}$ in the sense of $(\cdot,\cdot)_{\rm M}$ gives us
\begin{equation}\label{errc:be:1}
    \begin{aligned}
	& \left( \phi D_{\tau} \widehat{e}_{\eta}^{n}, \widehat{e}_{\eta}^{n} \right)_{\rm M} 
	+ \left( \nabla_{h} \cdot [ \bm{\mal{B}}(\widehat{e}_{\eta})\bm{U}_{h} - \bm{\mal{D}}(\bm{U}_{h}) \nabla_{h}\widehat{e}_{\eta} + \bm{e}_{\bm{u}}c 
        + \bm{\epsilon}_{\bm{3}} - \bm{\zeta}_{\bm{1}} ]^{n}, \widehat{e}_{\eta}^{n} \right)_{\rm M}  \\
       & \quad - \left( [q_{P}\widehat{e}_{\eta}]^{n}, \widehat{e}_{\eta}^{n} \right)_{\rm M} = \left(  \epsilon_{4}^{n} + \zeta_{2}^{n}, \widehat{e}_{\eta}^{n} \right)_{\rm M} 
		 \leq K \| \widehat{e}_{\eta}^{n} \|_{{\rm M},\beta^{n}}^{2} + K\| e_{c}^{n-1} \|_{{\rm M},\phi}^{2} + K\big( \tau^{2} + h^{4} \big),
	\end{aligned}
\end{equation}
where the truncation error estimates in \eqref{err:truc:be1} have been used.

Invoking Lemma \ref{lem:dis}, the second term on the left-hand side of \eqref{errc:be:1} can be written as
\begin{equation*}
  \begin{aligned}
   & \left( \nabla_{h} \cdot [ \bm{\mal{B}}(\widehat{e}_{\eta})\bm{U}_{h} - \bm{\mal{D}}(\bm{U}_{h}) \nabla_{h}\widehat{e}_{\eta} + \bm{e}_{\bm{u}}c + \bm{\epsilon}_{\bm{3}} - \bm{\zeta}_{\bm{1}} ]^{n}, \widehat{e}_{\eta}^{n} \right)_{\rm M} \\
   &\quad = - \left(  [ \bm{\mal{B}}(\widehat{e}_{\eta})\bm{U}_{h} - \bm{\mal{D}}(\bm{U}_{h}) \nabla_{h}\widehat{e}_{\eta} + \bm{e}_{\bm{u}}c + \bm{\epsilon}_{\bm{3}} - \bm{\zeta}_{\bm{1}} ]^{n}, \nabla_{h} \widehat{e}_{\eta}^{n} \right)_{\rm T}\\
   & \quad = \left(  \bm{\mal{D}}(\bm{U}_{h}^{n}) \nabla_{h} \widehat{e}_{\eta}^{n} ,  \nabla_{h} \widehat{e}_{\eta}^{n} \right)_{\rm T} - I_{3}
   =   \|\nabla_{h} \widehat{e}_{\eta}^{n}\|_{{\rm T},\bm{\mal{D}}(\bm{U}_{h}^{n})}^{2} - I_{3}.
   \end{aligned}
\end{equation*}
Here
$
I_{3}:= \left(  [  \bm{\mal{B}}(\widehat{e}_{\eta})\bm{U}_{h} + \bm{e}_{\bm{u}}c + \bm{\epsilon}_{\bm{3}} - \bm{\zeta}_{\bm{1}} ]^{n}, \nabla_{h} \widehat{e}_{\eta}^{n} \right)_{\rm T}
$
can be bounded as follows:
\begin{equation*}
	\begin{aligned}
		I_{3} 
		&\leq \f{1}{2} \|\nabla_{h} \widehat{e}_{\eta}^{n}\|_{{\rm T},\bm{\mal{D}}(\bm{U}_{h}^{n})}^{2}  
		+ K\|\bm{e}_{\bm{u}}^{n}\|_{\rm T}^{2} + K\| \widehat{e}_{\eta}^{n}\|_{{\rm M},\beta^{n}}^{2} +  K\| e_{c}^{n-1} \|_{{\rm M},\phi}^{2} + K \big( \tau^{4} + h^{4} \big),
	\end{aligned}
\end{equation*}
where we have used the boundedness result \eqref{assum:be} and the estimate \eqref{err:truc:be2}.  
Therefore, substituting the above estimates into \eqref{errc:be:1} and applying Lemma \ref{lem:err:be:pu} yields
\begin{equation}\label{errc:be:2}
   \begin{aligned}
   &  \left( \phi D_{\tau} \widehat{e}_{\eta}^{n}, \widehat{e}_{\eta}^{n} \right)_{\rm M}  
   - \left( [q_{P}\widehat{e}_{\eta}]^{n}, \widehat{e}_{\eta}^{n} \right)_{\rm M} 
   \leq  K\| \widehat{e}_{\eta}^{n}\|_{{\rm M},\beta^{n}}^{2} 
     +  K\| e_{c}^{n-1}\|_{{\rm M},\phi}^{2} 
     + K \bigl( \tau^{2} + h^{4} \bigr).
  \end{aligned}
\end{equation}
Thus, adding \eqref{errc:be:2} with the estimate in Lemma \ref{lem:be:ec}, 
we have
\begin{equation}\label{errc:be:4}
   \begin{aligned}
     &  \f{1}{4}d_{\tau} 
     \Bigl[\|\widehat{e}_{\eta}^{n}\|_{{\rm M},\beta^{n}}^{2} 
         +\|e_{c}^{n}\|_{{\rm M},\phi}^{2} 
         + \|[\alpha c]^{n}\|_{{\rm M},\phi}^{2}  
         + 2\big( \phi e_{c}^{n}, [\alpha c]^{n} \big)_{\rm M}\Bigr] 
         \\
     & \quad\leq K \|\widehat{e}_{\eta}^{n}\|_{{\rm M},\beta^{n}}^{2} + K \|e_{c}^{n-1}\|_{{\rm M},\phi}^{2} + K \big( \tau^{2} + h^{4} \big).
   \end{aligned}
\end{equation}

Now, multiplying both sides of \eqref{errc:be:4} by $4\tau$, replacing the index $n$ by $\ell$, and summing the resulting inequality over $\ell$ from $1$ to $n$. A straightforward application of Gronwall's inequality then yields an error estimate for $\|\widehat{e}_{\eta}^{n}\|_{{\rm M},\beta^{n}}$:
\begin{equation}\label{errc:be:5}
  \begin{aligned}
      \|\widehat{e}_{\eta}^{n}\|_{{\rm M},\beta^{n}}^{2}  
       \le \|\widehat{e}_{\eta}^{n}\|_{{\rm M},\beta^{n}}^{2}  + \| e_{c}^{n} \|_{{\rm M},\phi}^{2} + \| [\alpha c]^{n} \|_{{\rm M},\phi}^{2} + 2 \left( \phi e_{c}^{n}, [\alpha c]^{n} \right)_{\rm M}  
       \leq K \big( \tau^{2} + h^{4} \big),   
  \end{aligned}
\end{equation}
where the first inequality relies on the observation that 
$$
 \| e_{c}^{n} \|_{{\rm M},\phi}^{2} + \| [\alpha c]^{n} \|_{{\rm M},\phi}^{2} + 2 \left( \phi e_{c}^{n}, [\alpha c]^{n} \right)_{\rm M} \ge 0.
$$ 
Nevertheless, this estimate does not provide a corresponding bound for $\|e_{c}^{n}\|_{{\rm M},\phi}$,  due to the indefinite sign of the inner product term $\big( \phi e_{c}^{n}, [\alpha c]^{n} \big)_{\rm M}$.
However, note that
\begin{equation*}
    \begin{aligned}
        - 2\left( \phi e_{c}^{n}, [\alpha c]^{n} \right)_{\rm M}  \leq  \f{1}{2} \| e_{c}^{n} \|_{{\rm M},\phi}^{2} + 2 \| [\alpha c]^{n} \|_{{\rm M},\phi}^{2}.
    \end{aligned}
\end{equation*}
Adding this to \eqref{errc:be:5} and invoking Lemma \ref{lem:alpha:be}, we arrive at
\begin{equation}\label{errc:be:6}
    \begin{aligned}
         \| e_{c}^{n} \|_{{\rm M},\phi}^{2} \leq  2 \| [\alpha c]^{n} \|_{{\rm M},\phi}^{2}  + K \big( \tau^{2} + h^{4} \big) 
         \leq K \tau \sum_{\ell=1}^{n-1} \|e_{c}^{\ell}\|_{{\rm M},\phi}^{2} + K \big( \tau^{2} + h^{4} \big). 
    \end{aligned}
\end{equation}
Then, a second application of Gronwall's inequality to \eqref{errc:be:6} gives the desired estimate
\begin{equation}\label{errc:be:7}
   \| e_{c}^{n} \|_{{\rm M},\phi}^{2}
   \leq K \big( \tau^{2} + h^{4} \big). 
\end{equation}

Finally, combining the estimate \eqref{errc:be:7} with Lemma \ref{lem:err:be:pu} yields
\begin{equation}\label{errc:be:8}
	\|e_{p}^{n+1}\|_{\rm M}^{2} + \|\nabla_{h} e_{p}^{n+1}\|_{{\rm T}}^{2} + \| \bm{e}_{\bm{u}}^{n+1} \|_{\rm T}^{2} 
	\leq  K\| e_{c}^{n} \|_{{\rm M},\phi}^{2} + K\big( \tau^{2}  +  h^{4}\big)
	\leq K\big( \tau^{2}  +  h^{4}\big).
\end{equation}

\medskip
\textbf{Part 3: Completion of the mathematical induction.}
We end the proof of this theorem by completing the assumption \eqref{assum:be} for $\ell=n+1$. Following \eqref{errc:be:7}--\eqref{errc:be:8}, the inverse inequality, and Lemma \ref{lem:alpha:be}, we obtain
$$
\begin{aligned}
&\| \bm{U}^{n+1}\|_{\infty} \leq \| \bm{u}^{n+1}\|_{\infty} + \| \bm{e}_{\bm{u}}^{n+1}\|_{\infty} \leq K_{e} + Kh^{-1}(\tau  + h^2)  \leq K_{e} + 1, \\
&|\varepsilon^{n+1}| \leq K\|e_{c}^{n}\|_{{\rm M},\phi} + K\big(\tau^{2} + h^{2} \big) \leq K\big( \tau + \tau^{2} + h^{2} \big) \leq \eta,
\end{aligned}
$$
provided that $\tau$ and $h$ are sufficiently small and satisfy the condition \eqref{cond:be}. Therefore,
this implies that \eqref{assum:be} indeed holds for $\ell=n+1$.
\end{proof}
\begin{remark} The conclusion of Theorem \ref{thm:err:be} further implies that the auxiliary solution $c_{\eta}^{n}$ is indeed a small perturbation of the exact solution $c^{n}$, and their difference shares the same order of accuracy as the concentration error. In fact, Lemma \ref{lem:alpha:be} yields
$$
|c_{\eta}^{n} - c^{n}| = |[\alpha c]^{n}| \leq |\varepsilon^{n}| \leq K \big( \tau + h^{2} \big), \quad
|d_{\tau} \varepsilon^{n}| \leq K \big( \tau + h^{2} \big),
$$ 
which is fully consistent with the induction hypothesis $|\varepsilon^{n}| \leq \eta < 1/2$ and thereby justifies the statements in Remark \ref{rem:hatec:be}.
\end{remark}

Finally, the main stability result of the BE-BCFD-P scheme \eqref{scm:be:u}--\eqref{scm:be:min} is stated as follows.
\begin{theorem}\label{thm:stb:be} Under the assumptions of Theorem \ref{thm:err:be}, there exists a positive constant $K$, independent of $h$ and $\tau$, such that, for all $n=1,\ldots,N_t$, the following stability estimates hold
$$
\begin{aligned}
    \|P_{h}^{n}\|_{\rm M}^{2}  + \|\nabla_{h} P_{h}^{n}\|_{\rm T}^{2} 
        + \| \bm{U}_{h}^{n} \|_{\rm T}^{2} 
     \leq   K\|q^{n} \|_{\rm M}^{2},\quad
    \|C_{h}^{n}\|_{\rm M}^{2}  \leq  K\| c^{o} \|_{\rm M}^{2} 
       + K \tau \sum_{\ell=1}^{n} \|[q_{I}c_{I}]^{\ell}\|_{\rm M}^{2}.
\end{aligned}
$$
\end{theorem}
\begin{proof} The proof is divided into two parts. The first part is devoted to establishing stability estimates for the velocity and pressure, while the second part provides the stability result for the concentration.

\medskip
\textbf{Part 1: Stability for the velocity-pressure.}
Taking the inner product of the two equations in \eqref{scm:be:u} with $P_{h}^{n}$ and $\bm{U}_{h}^{n}$ in the sense of $(\cdot,\cdot)_{\rm M}$ and $(\cdot,\cdot)_{\rm T}$, respectively, and adding the resulting identities, yields
\begin{equation}\label{stb:be:u1} 
	\left(  \bm{\mal{A}}(C_{h}^{n-1}) \bm{U}_{h}^{n}, \bm{U}_{h}^{n} \right)_{\rm T} = \left( q^{n}, P_{h}^{n} \right)_{\rm M},
\end{equation}
where Lemma \ref{lem:dis} is invoked, i.e.,
$$
\left( \nabla_{h} \cdot \bm{U}_{h}^{n}, P_{h}^{n} \right)_{\rm M} +  \left(  \nabla_{h} P_{h}^{n}, \bm{U}_{h}^{n} \right)_{\rm T} = - \left( \bm{U}_{h}^{n}, \nabla_{h} P_{h}^{n} \right)_{\rm T} + \left(  \nabla_{h} P_{h}^{n}, \bm{U}_{h}^{n} \right)_{\rm T}  = 0.
$$
Furthermore, applying the Cauchy-Schwarz inequality, Lemma \ref{lem:seminorm}, and \eqref{assump:a} to \eqref{stb:be:u1} yields the estimate
\begin{equation}\label{stb:be:u3}
    a_{*}\| \bm{U}_{h}^{n}\|_{\rm T}^{2} \leq
    \left(  \bm{\mal{A}}(C_{h}^{n-1}) \bm{U}_{h}^{n}, \bm{U}_{h}^{n} \right)_{\rm T} 
    \leq \|q^{n}\|_{\rm M} \| P_{h}^{n} \|_{\rm M}
    \leq K\|q^{n}\|_{\rm M} \| \nabla_{h} P_{h}^{n} \|_{\rm T}.
\end{equation}
Moreover, the second equation in \eqref{scm:be:u} gives
\begin{equation}\label{stb:be:u2}
	\begin{aligned}
		\| \nabla_{h} P_{h}^{n} \|_{\rm T} = \| \bm{\mal{A}}(C_{h}^{n-1}) \bm{U}_{h}^{n} \|_{\rm T} \leq a^{*} \| \bm{U}_{h}^{n} \|_{\rm T}.
	\end{aligned}
\end{equation}
Therefore,  combining \eqref{stb:be:u3}, \eqref{stb:be:u2}, and Lemma \ref{lem:seminorm} yield the estimate
\begin{equation*}
	\|P_{h}^{n}\|_{\rm M} + \|\nabla_{h}P_{h}^{n}\|_{\rm T} 
    + \| \bm{U}_{h}^{n}\|_{\rm T}
	\leq K \| \bm{U}_{h}^{n}\|_{\rm T} 
    \leq K\|q^{n}\|_{\rm M}.
\end{equation*}

\medskip
\textbf{Part 2: Stability for the concentration.} As in \eqref{errc:be:1}, taking the inner product of \eqref{scm:be:c} with $\widehat{C}_{h}^{n}$ in the sense of $(\cdot,\cdot)_{\rm M}$ results in
\begin{equation}\label{stb:be:c1}
	\begin{aligned}
		  \big( \phi D_{\tau} \widehat{C}_{h}^{n}, \widehat{C}_{h}^{n} \big)_{\rm M} 
		& + \big( \nabla_{h} \cdot [  \bm{\mal{B}}(\widehat{C}_{h})\bm{U}_{h} - \bm{\mal{D}}(\bm{U}_{h}^{n}) \nabla_{h}\widehat{C}_{h} ]^{n}, \widehat{C}_{h}^{n} \big)_{\rm M} \\
        & - \big( [q_{P}\widehat{C}_{h}]^{n}, \widehat{C}_{h}^{n} \big)_{\rm M} = \big(  [q_{I}c_{I}]^{n}, \widehat{C}_{h}^{n} \big)_{\rm M} \leq K\| \widehat{C}_{h}^{n}\|_{{\rm M},\beta^{n}}^{2} + K\| [q_{I}c_{I}]^{n}\|_{\rm M}^{2}.
	\end{aligned}
\end{equation}

Note that $q_{P}^{n}$, $q_{P}^{n-1} < 0$ and Lemma \ref{lem:hat:be}, the first and third terms on the left-hand side of \eqref{stb:be:c1} can be estimated as
\begin{equation*}
    \begin{aligned}
        &  \big( \phi D_{\tau} \widehat{C}_{h}^{n}, \widehat{C}_{h}^{n} \big)_{\rm M} - \big( [q_{P}\widehat{C}_{h}]^{n}, \widehat{C}_{h}^{n} \big)_{\rm M}\\
         & \quad \ge \f{1}{2\tau} \Big( \| \widehat{C}_{h}^{n} \|_{{\rm M},\phi}^{2} 
         - \| C_{h}^{n-1} \|_{{\rm M},\phi}^{2} \Big) - \big( [q_{P}\widehat{C}_{h}]^{n}, \widehat{C}_{h}^{n} \big)_{\rm M}\\
        &  \quad = \f{1}{2\tau} \Big( \| \widehat{C}_{h}^{n} \|_{{\rm M},\beta^{n}}^{2} 
        - \| C_{h}^{n-1} \|_{{\rm M},\beta^{n-1}}^{2} \Big) 
        - \f{1}{2} \big( [q_{P}\widehat{C}_{h}]^{n}, \widehat{C}_{h}^{n} \big)_{\rm M} - \f{1}{2} \big( [q_{P}C_{h}]^{n-1}, C_{h}^{n-1} \big)_{\rm M}\\
        & \quad \geq \f{1}{4\tau} \Big( \| \widehat{C}_{h}^{n} \|_{{\rm M},\beta^{n}}^{2} + \| C_{h}^{n} \|_{{\rm M},\beta^{n}}^{2}  - \| \widehat{C}_{h}^{n-1} \|_{{\rm M},\beta^{n-1}}^{2}  -  \| C_{h}^{n-1} \|_{{\rm M},\beta^{n-1}}^{2} \Big).
    \end{aligned}
\end{equation*}
Next, analogous to \eqref{exit:be:c2}, the second left term of \eqref{stb:be:c1} admits the following estimate
\begin{equation*}
	\begin{aligned}
		   \big( \nabla_{h} \cdot [  \bm{\mal{B}}(\widehat{C}_{h})\bm{U}_{h} - D \nabla_{h}\widehat{C}_{h} ]^{n}, \widehat{C}_{h}^{n} \big)_{\rm M} 
		& = - \big(  [  \bm{\mal{B}}(\widehat{C}_{h})\bm{U}_{h} - \bm{\mal{D}}(\bm{U}_{h}^{n}) \nabla_{h}\widehat{C}_{h} ]^{n}, \nabla_{h}\widehat{C}_{h}^{n} \big)_{\rm T}\\
		  & \geq \f{1}{2} \|\nabla_{h} \widehat{C}_{h}^{n}\|_{{\rm T},\bm{\mal{D}}(\bm{U}_{h}^{n})}^{2} - K\|\widehat{C}_{h}\|_{{\rm M},\beta^{n}}^{2}.
	\end{aligned}
\end{equation*}
Therefore, substituting the above two estimates into \eqref{stb:be:c1}, we get
\begin{equation}\label{stb:be:c2}
	\begin{aligned}
		  \f{1}{4} d_{\tau} \bigl[\| \widehat{C}_{h}^{n} \|_{{\rm M},\beta^{n}}^{2} +  \| C_{h}^{n} \|_{{\rm M},\beta^{n}}^{2} \bigr]
		   + \f{1}{2} \|\nabla_{h} \widehat{C}_{h}^{n}\|_{{\rm T},\bm{\mal{D}}(\bm{U}_{h}^{n})}^{2} \leq K\| \widehat{C}_{h}^{n}\|_{{\rm M},\beta^{n}}^{2} + K\| [q_{I}c_{I}]^{n}\|_{\rm M}^{2}. 
	\end{aligned}
\end{equation}
Furthermore, multiplying both sides of \eqref{stb:be:c2} by $4\tau$, with $n$ replaced by $\ell$, and summing from $1$ to $n$, the application of Gronwall's inequality yields
\begin{equation*}
	\| \widehat{C}_{h}^{n} \|_{{\rm M},\beta^{n}}^{2} + \| C_{h}^{n} \|_{{\rm M},\beta^{n}}^{2}  + \tau \sum_{\ell=1}^{n} \|\nabla_{h} \widehat{C}_{h}^{\ell}\|_{{\rm T},\bm{\mal{D}}(\bm{U}_{h}^{\ell})}^{2} 
	\leq K\| C_{h}^{0} \|_{{\rm M}}^{2}  +  K \tau \sum_{\ell=1}^{n} \| [q_{I}c_{I}]^{n}\|_{\rm M}^{2}.
\end{equation*}
Finally, combining the stability results from both parts yields the desired conclusion.
\end{proof}

\subsection{Analysis for the second-order scheme}\label{subsec:2ndAnal}
Following the analysis developed in Subsection \ref{subsec:1stAnal}, analogous convergence and stability results can be established for the second-order scheme. To avoid repetition, we state only the final results and omit the detailed proofs unless they are required.

First, similar to Lemmas \ref{lem:lm:be}--\ref{lem:hatec:be}, we can obtain the following theoretical results for the CN-BCFD-P scheme \eqref{scm:cn:c}--\eqref{scm:cn:u}.
\begin{lemma}[Orthogonality of the weighted $L^2$-projection]\label{lem:lm:cn}
The three Lagrange multipliers $\lambda_h^{n}$, $\xi_h^{n} \in \mal{C}$ and $\gamma^{n} \in \mathbb{R}$ in \eqref{scm:cn:k1}--\eqref{scm:cn:k3} satisfy
$$
\big( \lambda_{h}^{n} - \xi_{h}^{n} - \gamma^{n} \bar{\beta}^{n}, 1 \big)_{\rm M} = 0,\quad n = 1,\ldots,N_{t}.
$$
\end{lemma}
\begin{lemma}[Stability of the weighted $L^2$-projection]\label{lem:hat:cn}
For the intermediate concentration $\widehat{C}_{h}^{n}$ and the projected concentration $C_{h}^{n}$ in the CN-BCFD-P scheme \eqref{scm:cn:c}--\eqref{scm:cn:u}, the following estimate holds
$$
  \| C_{h}^{n} \|_{{\rm M},\bar{\beta}^{n}}^{2} + \| C_{h}^{n}- \widehat{C}_{h}^{n} \|_{{\rm M},\bar{\beta}^{n}}^{2} \leq \| \widehat{C}_{h}^{n} \|_{{\rm M},\bar{\beta}^{n}}^{2},\quad n = 1,\ldots,N_{t}.
$$
\end{lemma}
\begin{lemma}\label{lem:hatec:cn}
    Let $\widehat{C}_{h}^{n}$ and $C_{h}^{n}$ be the intermediate concentration and the projected concentration, respectively, arising from the CN‑BCFD‑P scheme \eqref{scm:cn:c}--\eqref{scm:cn:u}. Suppose that $w_{h}^{n} \in \mal{C}$ satisfies
    \begin{equation*}
    0 \leq w_{h}^{n} \leq 1,\quad \big(\bar{\beta}^{n} w_{h}^{n},1 \big)_{\rm M}=\big(\bar{\beta}^{n} \widehat{C}_{h}^{n},1 \big)_{\rm M}.
    \end{equation*}
    Then, for $n=1,\ldots,N_{t}$, the following estimates hold
	$$
	\|w_{h}^{n} - C_{h}^{n}\|_{{\rm M},\bar{\beta}^{n}}^{2} + \|(w_{h}^{n} - C_{h}^{n})-(w_{h}^{n} - \widehat{C}_{h}^{n})\|_{{\rm M},\bar{\beta}^{n}}^{2} \leq \|w_{h}^{n} - \widehat{C}_{h}^{n}\|_{{\rm M},\bar{\beta}^{n}}^{2},
	$$	
    $$
    \big( \bar{\beta}^{n} \big( w_{h}^{n} - C_{h}^{n} \big) - \bar{\beta}^{n} \big( w_{h}^{n} - \widehat{C}_{h}^{n} \big), w_{h}^{n} - C_{h}^{n} \big)_{\rm M} \leq 0.    
    $$
\end{lemma}

Analogous to the first‑order scheme, to present an optimal-order error analysis for the second‑order projection scheme, we also introduce two "biased" error functions as follows:
$$
e_{\eta}^{n}:= \bar{c}_{\eta}^{n} - C_{h}^{n}, \quad \widehat{e}_{\eta}^{n}:= \bar{c}_{\eta}^{n} - \widehat{C}_{h}^{n}, 
$$
where the "biased" auxiliary solution is also introduced as
\begin{equation}\label{c:eta:cn}
    \bar{c}_{\eta}^{n}: = \big( 1 + \bar{\alpha}^{n} \big) c^{n}, 
    \quad 
    \bar{\alpha}^{n}:= \bar{\varepsilon}^{n}\chi(c^{n}),
    \quad
    \bar{\varepsilon}^{n}:= \f{ \big(\bar{\beta}^{n} C_{h}^{n},1 \big)_{\rm M} - \big(\bar{\beta}^{n} c^{n},1 \big)_{\rm M} }{\big(\bar{\beta}^{n} c^{n},\chi(c^{n}) \big)_{\rm M}}.
\end{equation}
As stated in Remark \ref{rem:hatec:be}, provided that $|\bar{\varepsilon}^{n}| \leq \eta < 1/2$, the "biased" auxiliary solution $\bar{c}_{\eta}^{n}$ meets the two conditions of Lemma \ref{lem:hatec:cn}:
$$
0 \leq \bar{c}_{\eta}^{n} \leq 1,
\quad
\big( \bar{\beta}^{n} \bar{c}_{\eta}^{n}, 1 \big)_{\rm M} = \big( \bar{\beta}^{n} \widehat{C}_{h}^{n}, 1 \big)_{\rm M},
$$
and the following projection error estimates hold:
\begin{equation}\label{aux:proj_err:cn}
       \|e_{\eta}^{n}\|_{{\rm M},\bar{\beta}^{n}}^{2} 
       + \|e_{\eta}^{n}-\widehat{e}_{\eta}^{n}\|_{{\rm M},\bar{\beta}^{n}}^{2} 
       \leq \|\widehat{e}_{\eta}^{n}\|_{{\rm M},\bar{\beta}^{n}}^{2},
       \quad 
     \big( \bar{\beta}^{n} \big(e_{\eta}^{n}- \widehat{e}_{\eta}^{n}\big), e_{\eta}^{n} \big)_{\rm M} \leq 0,
\end{equation}
In addition, following the proof of Lemma \ref{lem:alpha:be}, for $\bar{\varepsilon}^{n}$ defined in \eqref{c:eta:cn}, we can show the following key conclusion that are in great importance for the error analysis of the CN-BCFD-P scheme.
\begin{lemma}\label{lem:alpha:cn}
Under assumptions (A1)--(A2), there exists a positive constant $K$, independent of $h$ and $\tau$, such that for $n=1,\ldots,N_{t}$, the following estimates hold: 
$$
\begin{aligned}
    &|\bar{\varepsilon}^{n}| \leq K\|e_{c}^{n-1}\|_{{\rm M},\phi} + K\big( \tau^{3} + h^{2} \big),\\
    &|\bar{\varepsilon}^{n}|^{2}\leq K \tau \sum_{\ell=1}^{n-1} \|e_{c}^{\ell}\|_{{\rm M},\phi}^{2} + K \big( \tau^{4} + h^{4} \big),\\
    &|d_{\tau} \bar{\varepsilon}^{n}| \leq K\|e_{c}^{n-1}\|_{{\rm M},\phi} + K\big( \tau^{2} + h^{2} \big).
\end{aligned}
$$
\end{lemma}

Analogous to Lemma \ref{lem:be:ec}, we can also prove the following boundedness result for the second-order scheme.
\begin{lemma}\label{lem:cn:ec}
Suppose that $|\bar{\varepsilon}^{n}| \leq \eta < 1/2$. Then, there exists a positive constant $K$, independent of $h$ and $\tau$, such that for all $n=1,\ldots,N_{t}$, the following estimate holds
   \begin{equation*}
        \begin{aligned}
        \big( \phi D_{\tau} \widehat{e}_{\eta}^{n},  \widehat{e}_{\eta}^{n} +e_{\eta}^{n-1} \big)_{\rm M} 
         \geq \f{1}{2} d_{\tau} \left[ \|\widehat{e}_{\eta}^{n}\|_{{\rm M},\bar{\beta}^{n}}^{2} +  \|e_{c}^{n}\|_{{\rm M},\phi}^{2}  + \|[\bar{\alpha} c]^{n}\|_{{\rm M},\phi}^{2}   + 2 \big( \phi e_{c}^{n}, [\bar{\alpha} c]^{n} \big)_{\rm M} \right] - K \|\widehat{e}_{\eta}^{n}\|_{{\rm M},\bar{\beta}^{n}}^{2}.
    \end{aligned}
    \end{equation*}
\end{lemma}
\begin{proof} The proof of this result follows the same lines as Lemma \ref{lem:be:ec} by establishing two different lower bounds. First, note that the identities $\widehat{e}_{\eta}^{n} = \widehat{e}_{c}^{n} + [\bar{\alpha} c]^{n}$ and $e_{\eta}^{n-1}=e_{c}^{n-1} + [\bar{\alpha} c]^{n-1}$. Then, we have
\begin{equation}\label{cn:ec:1}
	\begin{aligned}
       & \big( \phi D_{\tau} \widehat{e}_{\eta}^{n},  \widehat{e}_{\eta}^{n} +e_{\eta}^{n-1} \big)_{\rm M}   \\ 
      & \quad \geq \big( \phi \big[ d_{\tau} e_{c}^{n} + (\widehat{e}_{\eta}^{n} -e_{\eta}^{n})/\tau + d_{\tau} [\bar{\alpha} c]^{n} \big],  e_{c}^{n}+e_{c}^{n-1} + \widehat{e}_{\eta}^{n}-e_{\eta}^{n} + [\bar{\alpha }c]^{n} + [\bar{\alpha c}]^{n-1}\big)_{\rm M} \\
      & \quad =:  d_{\tau} \left[\|e_{c}^{n}\|_{{\rm M},\phi}^{2} +   \|[\bar{\alpha} c]^{n}\|_{{\rm M},\phi}^{2}\right] + I_{4},
    \end{aligned}
\end{equation}
where, invoking $d_{\tau}e_{c}^{n} = d_{\tau}e_{\eta}^{n}-d_{\tau}[\bar{\alpha}c]^{n}$, the term $I_{4}$ can be transformed into
\begin{equation*}
    \begin{aligned}
        I_{4}:= &\big( \phi d_{\tau} e_{c}^{n}, \widehat{e}_{\eta}^{n} - e_{\eta}^{n} + [\bar{\alpha} c]^{n} + [\bar{\alpha} c]^{n-1} \big)_{\rm M} +  \f{1}{\tau}\big( \phi \big( \widehat{e}_{\eta}^{n} - e_{\eta}^{n} \big),  \widehat{e}_{\eta}^{n} - e_{\eta}^{n} + e_{\eta}^{n}+e_{\eta}^{n-1}\big)_{\rm M} \\
        & \quad   + \big( \phi d_{\tau} [\bar{\alpha} c]^{n},  e_{c}^{n} + e_{c}^{n-1} + \widehat{e}_{\eta}^{n} - e_{\eta}^{n}  \big)_{\rm M}  \\
        =& \big( \phi d_{\tau} e_{\eta}^{n} - \phi d_{\tau} [\bar{\alpha} c]^{n}, \widehat{e}_{\eta}^{n} - e_{\eta}^{n} \big)_{\rm M} + 2 d_{\tau} \big( \phi e_{c}^{n}, [\bar{\alpha} c]^{n} \big)_{\rm M} + \f{1}{\tau} \|\widehat{e}_{\eta}^{n} - e_{\eta}^{n}\|_{{\rm M},\phi}^{2} \\
        & \quad  +  \f{1}{\tau}\big( \phi \big( \widehat{e}_{\eta}^{n} - e_{\eta}^{n} \big),  e_{\eta}^{n}+e_{\eta}^{n-1}\big)_{\rm M}  + \big( \phi d_{\tau} [\bar{\alpha} c]^{n},   \widehat{e}_{\eta}^{n} - e_{\eta}^{n}  \big)_{\rm M}  \\
        = &  \f{2}{\tau}  \big( \phi \big( \widehat{e}_{\eta}^{n} - e_{\eta}^{n} \big),  e_{\eta}^{n} \big)_{\rm M} + 2 d_{\tau} \big( \phi e_{c}^{n}, [\bar{\alpha} c]^{n} \big)_{\rm M} + \f{1}{\tau} \|\widehat{e}_{\eta}^{n} - e_{\eta}^{n}\|_{{\rm M},\phi}^{2}.
    \end{aligned}
\end{equation*}
Furthermore, if $|\bar{\varepsilon}^{n}| \leq \eta < 1/2$, then the estimate \eqref{aux:proj_err:cn} is valid, which immediately yields
\begin{equation*}
   \begin{aligned}
      \f{2}{\tau}  \big( \phi \big(\widehat{e}_{\eta}^{n}-e_{\eta}^{n}\big), e_{\eta}^{n} \big)_{\rm M} & = \f{2}{\tau}  \big( \bar{\beta}^{n} \big(\widehat{e}_{\eta}^{n}-e_{\eta}^{n}\big), e_{\eta}^{n} \big)_{\rm M} + \big( q_{P}^{n} \big(\widehat{e}_{\eta}^{n}-e_{\eta}^{n}\big), e_{\eta}^{n} \big)_{\rm M} \\
      & \geq \big( q_{P}^{n} \big(\widehat{e}_{\eta}^{n}-e_{\eta}^{n}\big), e_{\eta}^{n} \big)_{\rm M}
      \geq - K \|\widehat{e}_{\eta}^{n}\|_{{\rm M},\bar{\beta}^{n}}^{2}.
   \end{aligned}
\end{equation*}
Therefore, similar to \eqref{lowbound:e1}, substituting the above estimates into \eqref{cn:ec:1} yields the first lower bound
\begin{equation}\label{lowbound:e3}
    \begin{aligned}
        \big( \phi D_{\tau} \widehat{e}_{\eta}^{n},  \widehat{e}_{\eta}^{n} +e_{\eta}^{n-1} \big)_{\rm M} \geq d_{\tau} \left[  \|e_{c}^{n}\|_{{\rm M},\phi}^{2}  + \|[\bar{\alpha} c]^{n}\|_{{\rm M},\phi}^{2}   + 2 \big( \phi e_{c}^{n}, [\bar{\alpha} c]^{n} \big)_{\rm M} \right] - K \|\widehat{e}_{\eta}^{n}\|_{{\rm M},\bar{\beta}^{n}}^{2}.
    \end{aligned}
\end{equation}

Next, analogous to \eqref{lowbound:e2}, 
we can derive another lower bound:
\begin{equation}\label{lowbound:e4}
\begin{aligned}
       & \left( \phi D_{\tau} \widehat{e}_{\eta}^{n}, \widehat{e}_{\eta}^{n} + e_{\eta}^{n-1} \right)_{\rm M} \\ 
       &\quad = \f{1}{\tau} \Big( \|\widehat{e}_{\eta}^{n}\|_{{\rm M},\bar{\beta}^{n}}^{2} 
       - \|e_{\eta}^{n-1}\|_{{\rm M},\bar{\beta}^{n-1}}^{2} \Big) 
       + \f{1}{2} \left( [q_{P}\widehat{e}_{\eta}]^{n}, \widehat{e}_{\eta}^{n} \right)_{\rm M} 
       - \f{1}{2} \left( [q_{P}e_{\eta}]^{n-1}, e_{\eta}^{n-1} \right)_{\rm M} \\
       & \quad
       \geq \f{1}{\tau}\Big( \|\widehat{e}_{\eta}^{n}\|_{{\rm M},\bar{\beta}^{n}}^{2}  - \|\widehat{e}_{\eta}^{n-1}\|_{{\rm M},\bar{\beta}^{n-1}}^{2} \Big) - K\|\widehat{e}_{\eta}^{n}\|_{{\rm M},\bar{\beta}^{n}}^{2} .
    \end{aligned}
\end{equation}
Finally, adding the two lower bounds \eqref{lowbound:e3} and \eqref{lowbound:e4} together yields the desired result.
\end{proof}

Following the proof of Theorem \ref{thm:err:be} and invoking Lemmas \ref{lem:alpha:cn} and \ref{lem:cn:ec}, we can obtain the following error estimate for the CN-BCFD-P scheme \eqref{scm:cn:c}--\eqref{scm:cn:u}.
\begin{theorem}\label{thm:err:cn}
Under assumptions (A1)--(A2), the CN-BCFD-P scheme \eqref{scm:cn:c}--\eqref{scm:cn:u} admits a unique solution triple $\{P_{h}^{n},\bm{U}_{h}^{n},C_{h}^{n}\}$, provided that the time-step
condition $\tau=o(h^{1/2})$ is satisfied. Moreover, there exists a positive constant $K$, independent of $h$ and $\tau$, such that for all $n=1,\ldots,N_{t}$,  the following estimate holds:
$$
\begin{aligned}
    \|e_{c}^{n}\|_{\rm M}^{2}  
    + \|e_{p}^{n}\|_{\rm M}^{2} + \|\nabla_{h} e_{p}^{n}\|_{\rm T}^{2} + \| \bm{e}_{\bm{u}}^{n} \|_{\rm T}^{2} 
    \leq K\big( \tau^{4}  +  h^{4}\big).
\end{aligned}
$$
In particular, at the initial time $t=t^{0}$, the system \eqref{scm:cn:u} admits a unique solution pair $\{ P_{h}^{0}, \bm{U}_{h}^{0} \}$, and there exists a positive constant $K$, independent of $\tau$ and $h$, such that the following estimate holds:
$$
 \|e_{p}^{0}\|_{\rm M}^{2} + \|\nabla_{h} e_{p}^{0}\|_{\rm T}^{2} + \| \bm{e}_{\bm{u}}^{0} \|_{\rm T}^{2} \leq K h^{2}.
$$
\end{theorem}

Similarly, an application of Lemma \ref{lem:hat:cn}, together with an argument analogous to that of Theorem \ref{thm:stb:be}, yields the following stability result for the second-order scheme \eqref{scm:cn:c}--\eqref{scm:cn:u}.
\begin{theorem}\label{thm:stb:cn}
Under the assumptions of Theorem \ref{thm:err:cn}, there exists a positive constant $K$, independent of $h$ and $\tau$, such that for all $n=1,\ldots,N_{t}$, the following stability estimates hold:
$$
\|P_{h}^{n}\|_{\rm M}^{2} + \|\nabla_{h} P_{h}^{n}\|_{\rm T}^{2} + \| \bm{U}_{h}^{n} \|_{\rm T}^{2} \leq K\|q^{n} \|_{\rm M}^{2},
\quad
\|C_{h}^{n}\|_{\rm M}^{2} \leq K\| c^{o} \|_{\rm M}^{2} + K \tau \sum_{\ell=1}^{n} \|[q_{I}c_{I}]^{\ell-1/2}\|_{\rm M}^{2}.
$$
\end{theorem}

\section{Numerical experiments}\label{sec:num}
In this section, we present several numerical examples to verify the convergence and demonstrate the performance of the proposed BE-BCFD-P and CN-BCFD-P schemes, given by \eqref{scm:be:u}--\eqref{scm:be:min} and \eqref{scm:cn:c}--\eqref{scm:cn:u}, respectively, as well as their counterparts BE-BCFD-P2 and CN-BCFD-P2 with the weighted $L^{2}$-projection replaced by the weighted $H^{1}$-projection. 

To determine the mass-conserving Lagrange multiplier $\gamma^{n}$ via the nonlinear equation \eqref{F:be}, we apply the secant method as follows:
$$
\gamma_{(\ell+1)}^{n} = \gamma_{(\ell)}^{n} - \frac{F(\gamma_{(\ell)}^{n})(\gamma_{(\ell)}^{n} - \gamma_{(\ell-1)}^{n})}{F(\gamma_{(\ell)}^{n}) - F(\gamma_{(\ell-1)}^{n})}, \quad \ell = 1, 2, \dots,
$$
with the initial guesses $\gamma_{(0)}^{n}=0$ and $\gamma_{(1)}^{n}$,  and the iteration is stopped once $|F(\gamma_{(\ell+1)}^{n})| \leq TOL$, where $\gamma_{(1)}^{n}=TOL=10^{-16}$ in Example \ref{exm:e1} and $\gamma_{(1)}^{n}=TOL=10^{-10}$ in Examples \ref{exm:e2}--\ref{exm:e3}. In addition, the semismooth Newton method, detailed in \ref{App:A}, is employed for the nonlinear system \eqref{F:cn}, where the convergence criterion is set as $\|V_{h,\ell+1}^n-V_{h,\ell}^n\|_{\infty} \leq 10^{-10}$. 

In practical simulations, the diffusion coefficient $\bm{D}(\bm{x},\bm{u})$ is nonlinear and may take various forms. A well-known form is the Bear-Scheidegger diffusion-dispersion tensor \cite{Peaceman77}: 
\begin{equation*}
	\bm{D}(\bm{x},\bm{u}) 
    = \phi(\bm{x}) \left( d_{\rm mol} \bm{I} + d_{\rm long}|\bm{u}|\bm{E}(\bm{u})  
           + d_{\rm tran}|\bm{u}| \left(  \bm{I} - \bm{E}(\bm{u}) \right)  \right),\quad 
    \bm{E}(\bm{u}):=  \f{\bm{u}\bm{u}^{\top}}{|\bm{u}|^{2}},
\end{equation*}
where $|\bm{u}|^{2}=\bm{u} \cdot \bm{u}$ denotes the Euclidean norm, $\bm{I}$  is the identity matrix, and $d_{\rm mol}$, $d_{\rm long}$ and $d_{\rm tran}$ are the molecular diffusion, longitudinal dispersion, and transverse dispersion coefficients, respectively. Such a velocity-dependent nonlinear diffusion-dispersion tensor is symmetric and positive definite. For simplicity, we write
$$
\bm{D}(\bm{v})=\bm{D}(\bm{x},\bm{v})= \left(
\begin{aligned}
    & D_{11}(\bm{v}) & D_{12}(\bm{v}) \\
    & D_{21}(\bm{v}) & D_{22}(\bm{v}) 
\end{aligned} \right).
$$
Accordingly, the corresponding second‑order modified upwind finite difference approximation of $\bm{v} w - \bm{D}(\bm{v}) \nabla w$ is given by
$$
  \bm{\mal{B}}(w)\bm{v} - \bm{\mal{D}}(\bm{v}) \nabla_{h} w = \left(
  \begin{aligned}
      &v^{x}\,\mal{H}_{x} w\\
      &v^{y}\,\mal{H}_{y} w
  \end{aligned} \right)
  - \left(
  \begin{aligned}
     & \underline{D}_{1,1}\bigl(v^{x},\Tvu v^{y}\bigr) 
     && D_{1,2}\bigl(v^{x},\Tvu v^{y} \bigr) \Tvu \\
     & D_{2,1} \bigl(\Tuv v^{x},v^{y}\bigr)\Tuv 
     && \underline{D}_{2,2}\bigl(\Tuv v^{x},v^{y}\bigr) 
  \end{aligned} \right) \nabla_{h} w,
$$
where 
$\Tuv:~\mal{U} \rightarrow \mal{V}$ and $\Tvu:~\mal{V} \rightarrow \mal{U}$ are local two-dimensional average operators, defined componentwise as
$$
  [\Tuv v^{x} ]_{i,j+1/2} = \f{v^{x}_{i-1/2,j} + v^{x}_{i+1/2,j} + v^{x}_{i-1/2,j+1} + v^{x}_{i+1/2,j+1} }{4},
$$
$$
 [\Tvu v^{y} ]_{i+1/2,j} = \f{v^{y}_{i,j-1/2} + v^{y}_{i+1,j-1/2} + v^{y}_{i,j+1/2} + v^{y}_{i+1,j+1/2} }{4}.
$$

\begin{example}\label{exm:e1}  For the first example, the spatial domain is selected as $\Omega=(0,1)^2$. We consider model \eqref{mod} with the following manufactured exact solutions:
 	\[
 	c(\bm{x},t) = 8t^{2}e^{\f{-(x-0.5)^2+(y-0.5)^2}{0.02}}( x^2(x-1)^2 + y^2(y-1)^2),\quad
 	p(\bm{x},t) = t^3\sin(\pi x)\sin(\pi y),
 	\]
 	\[
 	\bm{u}(\bm{x},t) = t^3(x(x-1)(2x-1),~y(y-1)(2y-1))^{\top},
 	\]
where the model physical parameters are specified as
 	\[
 	\phi(\bm{x}) = (x+y+1)^2/10,\quad
 	k(\bm{x}) = (x+y+1)^3,\quad
 	\mu(c) = 1 + c^2,\quad
 	\]
 	\[
 	q_P(\bm{x},t) = \cos(2\pi(x+y+t)) -2,\quad
 	\bm{D}(\bm{x},\bm{u}) = \phi(\bm{x}) ( 0.1\bm{I} + \bm{u} \bm{u}^{\top} ).
 	\]
\end{example}

This example is used to evaluate  both the convergence behavior and structure-preserving properties of the proposed BE-BCFD-P scheme \eqref{scm:be:u}--\eqref{scm:be:min} and the CN-BCFD-P scheme \eqref{scm:cn:c}--\eqref{scm:cn:u}, as well as their BE-BCFD-P2 and CN-BCFD-P2 counterparts, wherein the weighted $L^2$-projection replaced by the weighted $H^1$-projection. The errors and convergence orders for the velocity $\bm{u}$, pressure $p$, and concentration $c$ are reported in Tables \ref{tab:e1:bel2}–\ref{tab:e1:cnh1} for all four schemes. Furthermore, Fig. \ref{fig:e1-maxmin} presents the evolution of the extrema of the projection solution $C_h$ and the prediction solution $\widehat{C}_h$, the number of iterations  required to solve for the Lagrange multiplier $\gamma^n$, and the mass errors. These results are obtained using the BE-BCFD-P and BE-BCFD-P2 schemes with $N_{t} = 1600$ and $N_{x}=N_{y}=80$, and the CN-BCFD-P and CN-BCFD-P2 schemes with $N_{t}=N_{x}=N_{y}=80$. Basically, we have the following observations:
\begin{table}[!tbp] \small
	\setlength{\abovecaptionskip}{0.05cm} 
	\centering
	\caption{(Example \ref{exm:e1}) Errors and convergence orders of the BE-BCFD-P scheme.}
	\label{tab:e1:bel2}
	{
        \begin{tabular*}{0.9\columnwidth}{@{\extracolsep\fill}ccccccccc@{\extracolsep\fill}}
		\toprule
		{$N_{t}$} &{$N_{x}=N_{y}$} & $\|e_{c}\|_{\rm M}$ & Order & $\|e_{p}\|_{\rm M}$ & Order & $\|\bm{e}_{\bm{u}}\|_{\rm T}$ & Order & CPU time\\
		\midrule
        25    &10  &1.49e-02  &---	 &4.20e-02  &---   &2.66e-02  &---	 &0.38 \\
        100	  &20  &3.35e-03  &2.15	 &1.01e-02  &2.06  &8.17e-03  &1.70  &1.37\\
        400	  &40  &8.56e-04  &1.97	 &2.31e-03  &2.13  &2.01e-03  &2.02	 &10.12\\
        1600  &80  &2.15e-04  &1.99	 &5.14e-04  &2.17  &5.08e-04  &1.98  &130.69\\	
		\bottomrule
	\end{tabular*}}
\end{table}
\begin{table}[!tbp] \small
	\setlength{\abovecaptionskip}{0.05cm} 
	\centering
	\caption{(Example \ref{exm:e1}) Errors and convergence orders of the BE-BCFD-P2 scheme.}
	\label{tab:e1:beh1}
	{
        \begin{tabular*}{0.9\columnwidth}{@{\extracolsep\fill}ccccccccc@{\extracolsep\fill}}
		\toprule
		{$N_{t}$} &{$N_{x}=N_{y}$} & $\|e_{c}\|_{\rm M}$ & Order & $\|e_{p}\|_{\rm M}$ & Order & $\|\bm{e}_{\bm{u}}\|_{\rm T}$ & Order & CPU time\\
		\midrule
        25    &10  &1.49e-02  &---	 &4.20e-02  &---   &2.66e-02  &---	  &0.55\\
        100	  &20  &3.35e-03  &2.15	 &1.01e-02  &2.06  &8.17e-03  &1.70   &2.15\\
        400	  &40  &8.55e-04  &1.97	 &2.31e-03  &2.13  &2.01e-03  &2.02	  &17.19\\
        1600  &80  &2.15e-04  &1.99	 &5.14e-04  &2.17  &5.08e-04  &1.98   &155.81\\	
		\bottomrule
	\end{tabular*}}
\end{table}
\begin{table}[!tbp] \small
	\setlength{\abovecaptionskip}{0.05cm} 
	\centering
	\caption{(Example \ref{exm:e1}) Errors and convergence orders of the CN-BCFD-P scheme.}
	\label{tab:e1:cnl2}
	{ 
        \begin{tabular*}{0.9\columnwidth}{@{\extracolsep\fill}ccccccccc@{\extracolsep\fill}}
		\toprule
		{$N_{t}$} &{$N_{x}=N_{y}$} & $\|e_{c}\|_{\rm M}$ & Order & $\|e_{p}\|_{\rm M}$ & Order & $\|\bm{e}_{\bm{u}}\|_{\rm T}$ & Order & CPU time\\
		\midrule
		10	 &10	 &1.38e-02  &---	 &4.22e-02  &---	 &2.83e-02  &---	 &0.21\\
        20	 &20	 &3.08e-03  &2.16	 &1.02e-02  &2.05	 &8.31e-03  &1.77	 &0.41\\
        40	 &40	 &7.88e-04  &1.97	 &2.32e-03  &2.14	 &2.03e-03  &2.03	 &1.23\\
        80	 &80	 &1.98e-04  &1.99	 &5.17e-04  &2.17	 &5.13e-04  &1.98	 &7.09\\
		\bottomrule
	\end{tabular*}}
\end{table}
\begin{table}[!tbp] \small
	\setlength{\abovecaptionskip}{0.05cm} 
	\centering
	\caption{(Example \ref{exm:e1}) Errors and convergence orders of the CN-BCFD-P2 scheme.}
	\label{tab:e1:cnh1}
	{ 
        \begin{tabular*}{0.9\columnwidth}{@{\extracolsep\fill}ccccccccc@{\extracolsep\fill}}
		\toprule
		{$N_{t}$} &{$N_{x}=N_{y}$} & $\|e_{c}\|_{\rm M}$ & Order & $\|e_{p}\|_{\rm M}$ & Order & $\|\bm{e}_{\bm{u}}\|_{\rm T}$ & Order & CPU time\\
		\midrule
		10	 &10	 &1.37e-02  &---	 &4.22e-02  &---	 &2.82e-02  &---	 &0.35\\
        20	 &20	 &3.08e-03  &2.15	 &1.02e-02  &2.05	 &8.31e-03  &1.76	 &0.71\\
        40	 &40	 &7.87e-04  &1.97	 &2.32e-03  &2.14	 &2.03e-03  &2.03	 &3.61\\
        80	 &80	 &1.98e-04  &1.99	 &5.17e-04  &2.17	 &5.13e-04  &1.98	 &24.14\\
		\bottomrule
	\end{tabular*}}
\end{table}
\begin{itemize}
    \item[(i)] Both the BE-BCFD-P and BE-BCFD-P2 schemes exhibit first-order accuracy in time and second-order accuracy in space for all the physical variables of interest, whereas the CN-BCFD-P and CN-BCFD-P2 schemes achieve second-order accuracy in both time and space. These numerical convergence rates are fully consistent with the theoretical error estimates established in Theorems \ref{thm:err:be} and \ref{thm:err:cn}. 

    \item[(ii)] The weighted $H^{1}$-projection and $L^{2}$-projection schemes yield essentially identical errors, but in our simulations the latter ones, i.e., the BE-BCFD-P scheme \eqref{scm:be:u}--\eqref{scm:be:min} and the CN-BCFD-P scheme \eqref{scm:cn:c}--\eqref{scm:cn:u}, are computationally more efficient. First, as observed in Fig. \ref{fig:e1-maxmin}, the $L^{2}$-projection schemes require fewer iterations than their $H^{1}$-projection counterparts. Second, they involve solving only a single-variable nonlinear algebraic equation, whereas the $H^{1}$-projection schemes necessitate solving a system of nonlinear algebraic equations, which undoubtedly increases the computational cost.
    
    \item[(iii)] Compared with the first-order-in-time schemes, the second-order-in-time schemes allow for larger time steps, consume less CPU time, and yield numerical solutions of comparable accuracy. For example, the BE-BCFD-P and BE-BCFD-P2 schemes with $N_{t} = 1600$ and $N_{x}=N_{y}=80$ produce error results for all physical variables that are almost identical to those obtained by the CN-BCFD-P and CN-BCFD-P2 schemes with $N_{t}=N_{x}=N_{y}=80$. However, while the CN-BCFD-P scheme costs only about $7$ seconds, the BE-BCFD-P scheme consumes over $2$ minutes of CPU time. Thus, the second-order-in-time schemes prove to be more efficient in practical simulations.

    \item[(iv)] It is observed that, prior to projection, the first-order-in-time schemes begin to violate the lower bound around $t \approx 0.93$, and the second-order-in-time schemes do so around $t \approx 0.36$. No violations of the upper bound are observed in the simulations for any scheme. After projection, however, all four schemes yield solutions that naturally satisfy both the physical lower and upper bounds on concentration. Moreover, as seen the predicted and projected maximum concentration curves almost overlap for all schemes; to better visualize their discrepancy, we also plot the enlarged difference in Fig. \ref{fig:e1-maxmin}. As long as the prediction remains within the desired physical bounds, the two maxima fully coincide; once a bound violation occurs, the projection step is triggered, and the difference between the predicted and projected maxima emerges. In addition, the mass errors for all schemes remain on the order of $10^{-14}$. Thus, all four projection schemes are both bound-preserving and mass-conserving.
\end{itemize}
\begin{figure}[!t]
	\centering
    \includegraphics[width=0.24\textwidth]{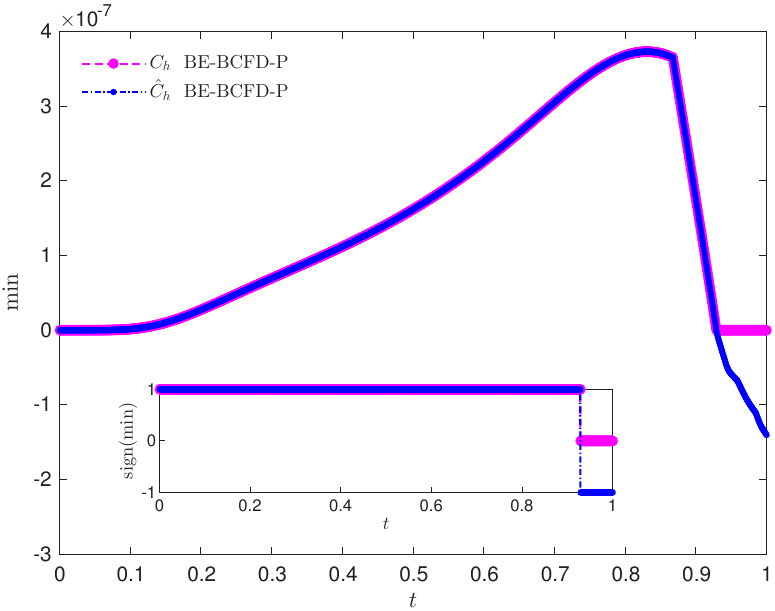}
    \includegraphics[width=0.24\textwidth]{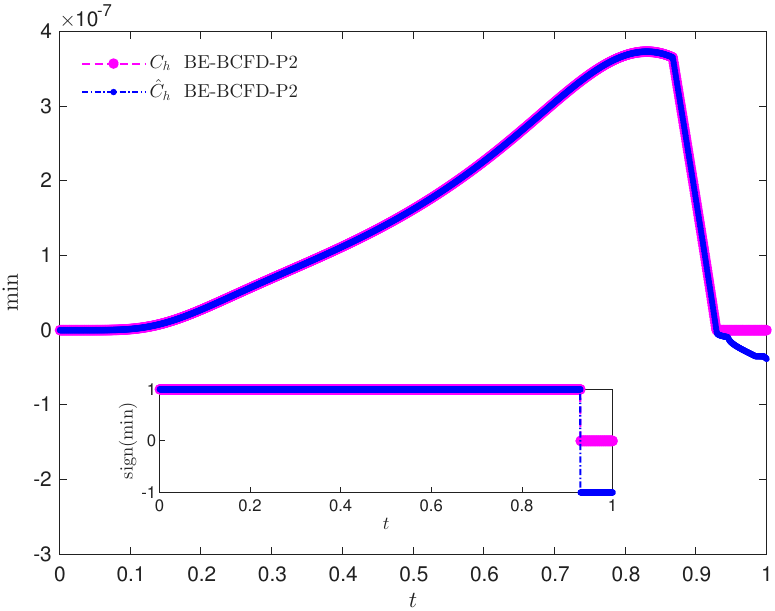}
    \includegraphics[width=0.24\textwidth]{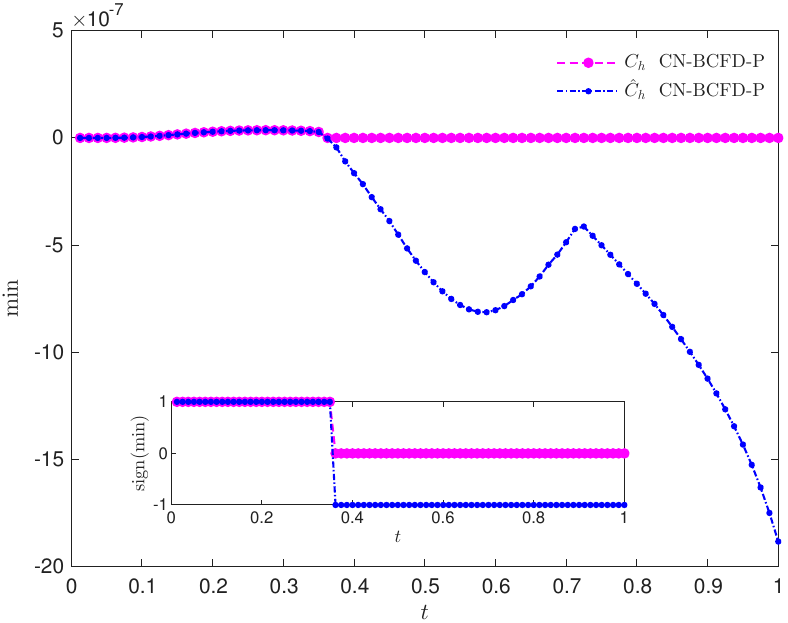}
    \includegraphics[width=0.24\textwidth]{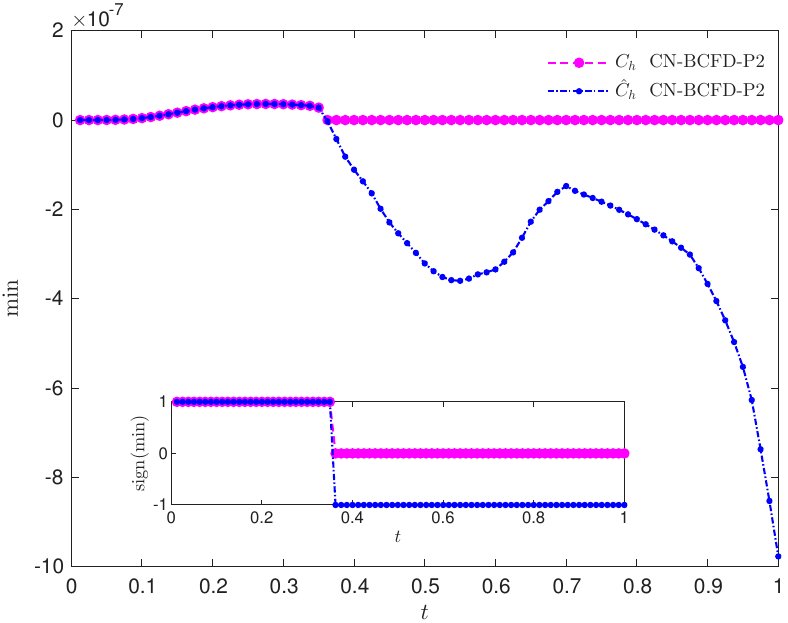}
    \includegraphics[width=0.24\textwidth]{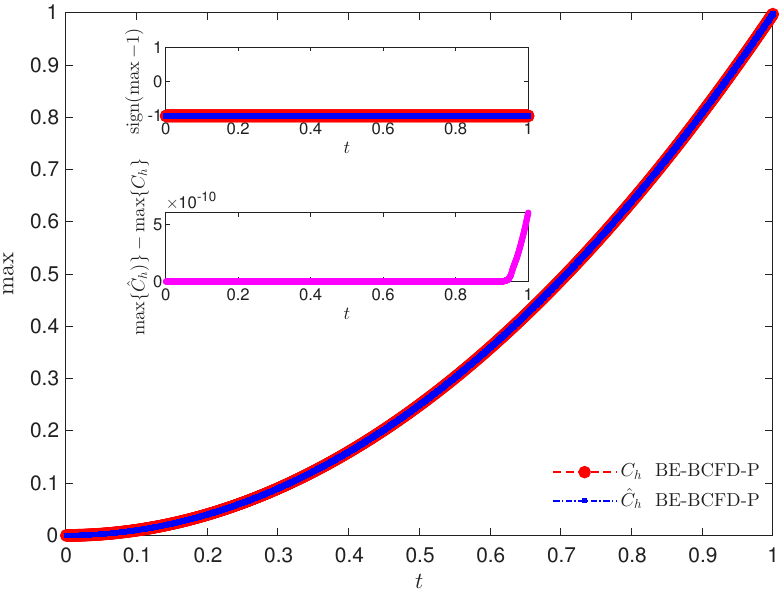}
    \includegraphics[width=0.24\textwidth]{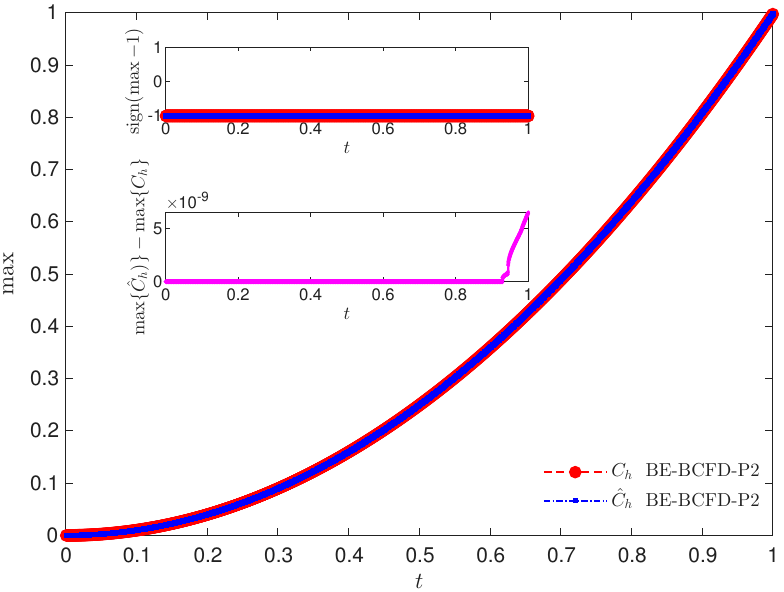}
    \includegraphics[width=0.24\textwidth]{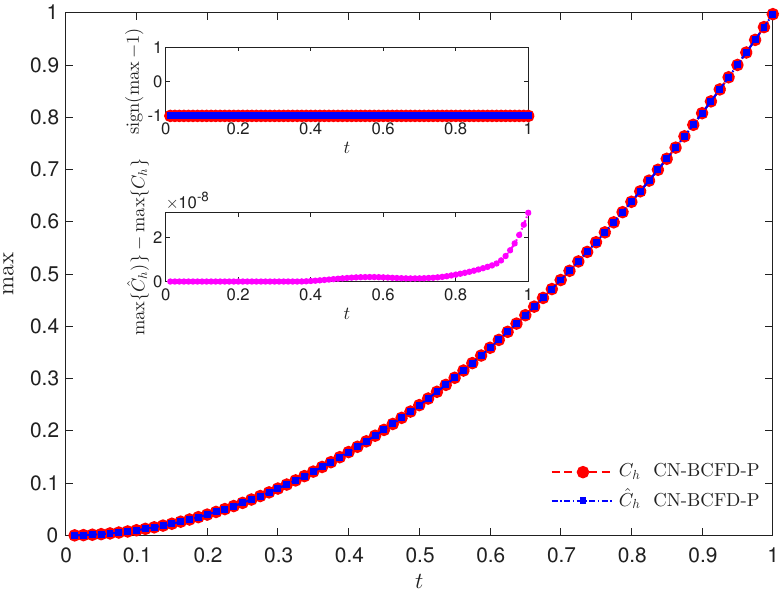}
    \includegraphics[width=0.24\textwidth]{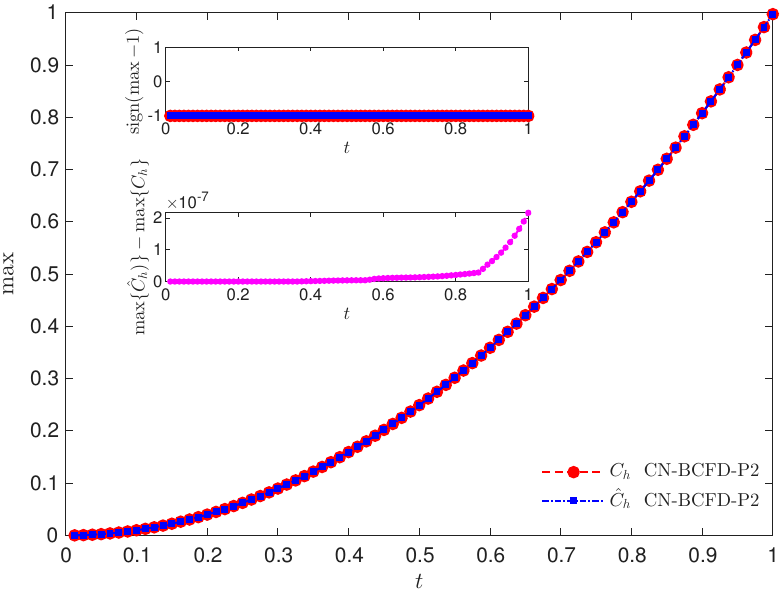}
    \includegraphics[width=0.24\textwidth]{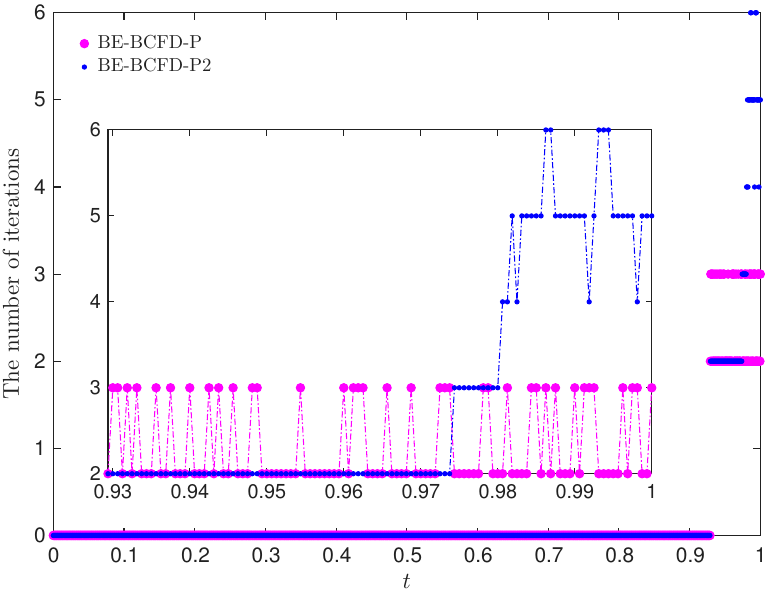}
    \includegraphics[width=0.24\textwidth]{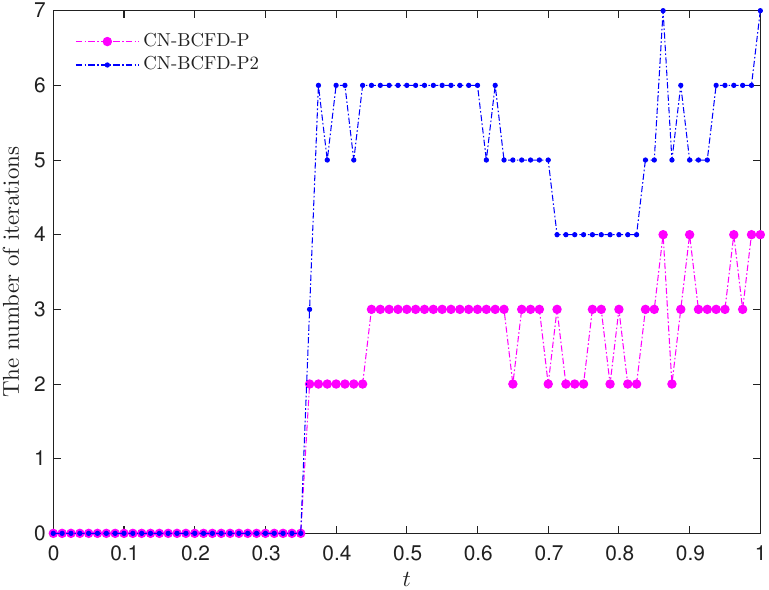}
    \includegraphics[width=0.24\textwidth]{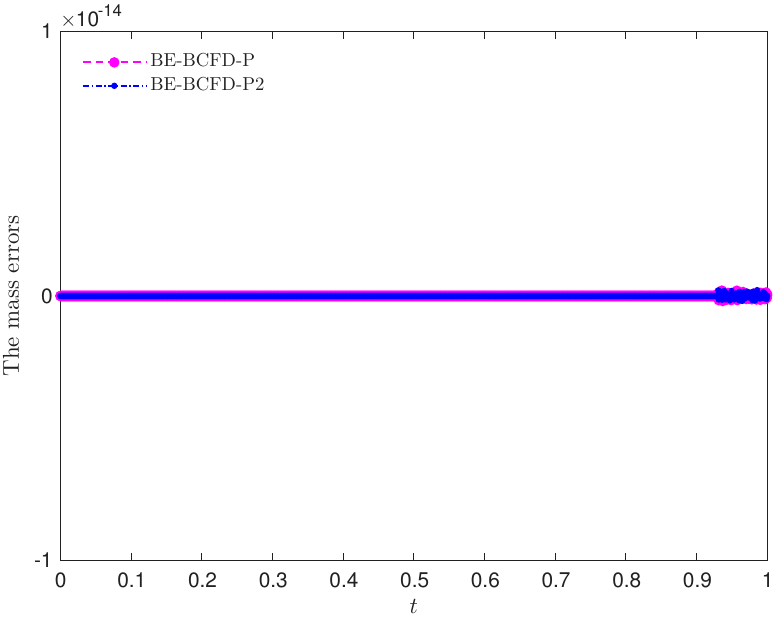}
    \includegraphics[width=0.24\textwidth]{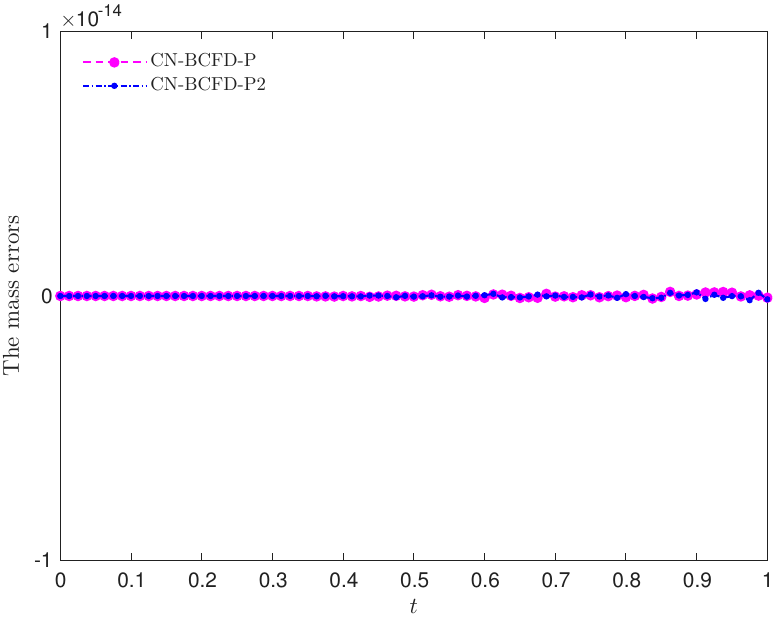}
    \caption{(Example \ref{exm:e1}) Evolution of the extrema of $C_{h}^{n}$ and $\widehat{C}_{h}^{n}$, the number of iterations and the mass errors for the BE-BCFD-P, BE-BCFD-P2, CN-BCFD-P and CN-BCFD-P2 schemes.}
    \label{fig:e1-maxmin}
\end{figure}

In summary, the second-order scheme with $L^{2}$-projection, namely the CN-BCFD-P scheme \eqref{scm:cn:c}--\eqref{scm:cn:u}, outperforms the other three schemes in both accuracy and efficiency. We therefore adopt this scheme in all the following simulations.

\begin{example}[\cite{WH00,CD07,VPV21,Eymard19,WRF26}]\label{exm:e2}
In this example, we simulate various incompressible displacement problems in a water-oil system within porous media to systematically assess the numerical performance of the proposed CN-BCFD-P scheme \eqref{scm:cn:c}--\eqref{scm:cn:u}. All numerical simulations are performed in the first quadrant of a standard five-spot well pattern for the horizontal reservoir of unit thickness, where the injection and production wells are located at opposite corners \cite{WH00,CD07,VPV21,Eymard19,WRF26}. The spatial domain is $\Omega = (0, 1000)^2$ $\text{ft}^2$, the time period is $[0, T] = [0, 3600]$ days, and the viscosity of the oil is $\mu_r = 1.0$ cp. The injection well is located at the upper-right corner $(1000,1000)$ of the domain with an injection rate $q_{I} = 30$ $\text{ft}^2/\text{day}$, and the production well is located at the lower-left corner $(0,0)$ with a production rate $q_{P} = - 30$ $\text{ft}^2/\text{day}$. The initial concentration and injection concentration of invading fluid are $c^{o} = 0$ and $c_{I}=1$, respectively.  In all simulations, we use a fairly coarse spatial grids with $h^{x}=h^{y}=20$ ft, and take $\tau = 120$ days.

As discussed in the introduction, a large mobility ratio $\mal{M}$ can trigger viscous fingering. In addition, dispersion and heterogeneity also play important roles in the presence of this phenomenon. Accordingly, we present four distinct tests in this example, varying in mobility ratio, dispersion, and permeability, to examine viscous fingering and displacement performance of the proposed scheme. The parameter configurations are specified below.
\begin{itemize}
    \item[\textbf{Case I.}] Homogeneous and isotropic porous media with the following parameters:
    \[
    \phi(\bm{x}) = 0.1,\quad k(\bm{x}) = 80~ {\rm md}, \quad d_{\rm mol} = 10~ {\rm ft}^{2}/{\rm day}, \quad d_{\rm long} = d_{\rm tran} = 0~ {\rm ft}/{\rm day}, \quad \mal M=1.
    \]    
        
    \item[\textbf{Case II.}] 
    Anisotropic dispersion in tensor form and an adverse mobility ratio:
    \[
    \phi(\bm{x}) = 0.1,~~ k(\bm{x}) = 80~ {\rm md}, ~~ d_{\rm mol} = 0~ {\rm ft}^{2}/{\rm day}, ~~ d_{\rm long} = 50~ {\rm ft}/{\rm day},~~  d_{\rm tran} = 5~ {\rm ft}/{\rm day}, ~~ \mal{M} = 41.
    \]
    
    \item[\textbf{Case III.}] Heterogeneous porous media with a discontinuous permeability field:
    \[
    k(\bm{x}) = \left\{ 
        \begin{aligned}
              & 80~ {\rm md}, && \Omega_{L}:=(0,1000)\times (0,500),\\
              & 20~ {\rm md}, && \Omega_{U}:=(0,1000)\times (500,1000),
        \end{aligned}
    \right.
    \]
    with all other parameters identical to those in Case II.
    
    \item[\textbf{Case IV.}] Piecewise porous media with with the following permeability and porosity distributions:
        \[
     \phi(\bm{x}) = \left\{ \begin{aligned}
         & 0.09, &\Omega_{1}:=(150,550)\times (150,550),\\
         & 0.1,  & \Omega_{2}:=\Omega \setminus \Omega_{1},
     \end{aligned} \right. \qquad 
     k(\bm{x}) = \left\{ \begin{aligned}
         & 25~{\rm md}, &\Omega_{1},\\
         & 80~{\rm md},  & \Omega_{2},
     \end{aligned} \right. 
       \]
   with all other parameters identical to those in Case II.
\end{itemize}   
\end{example}
\begin{figure}[!ht]
	\centering
    \includegraphics[width=0.24\textwidth]{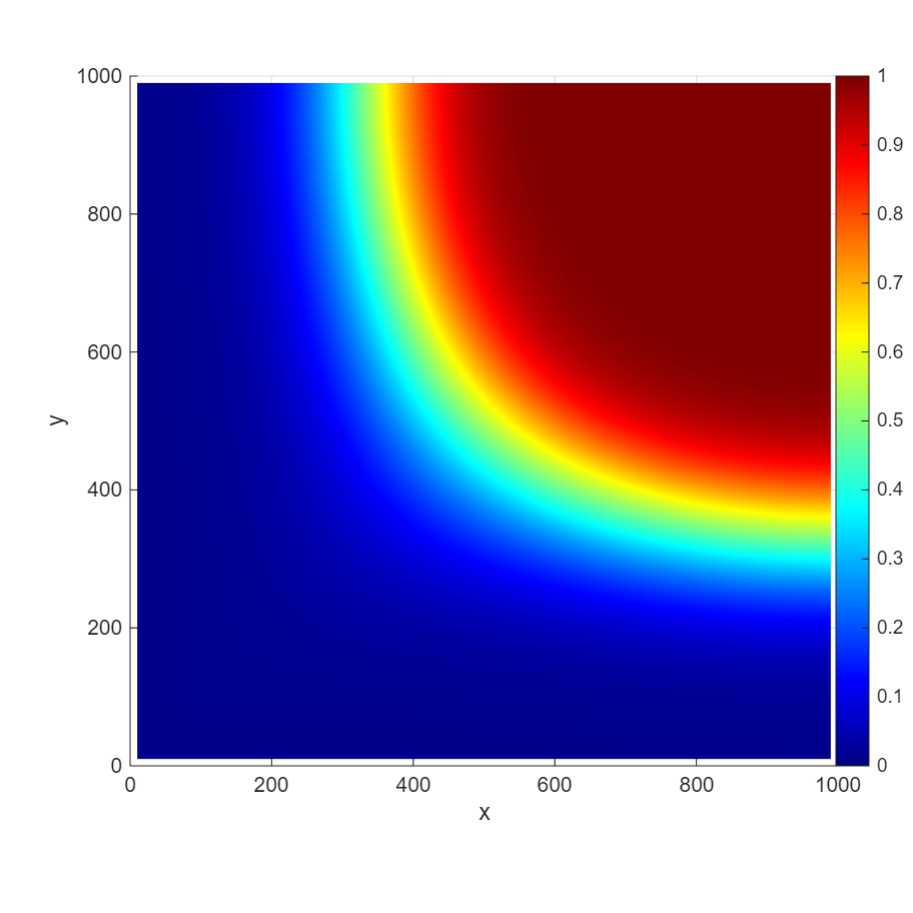}
    \includegraphics[width=0.24\textwidth]{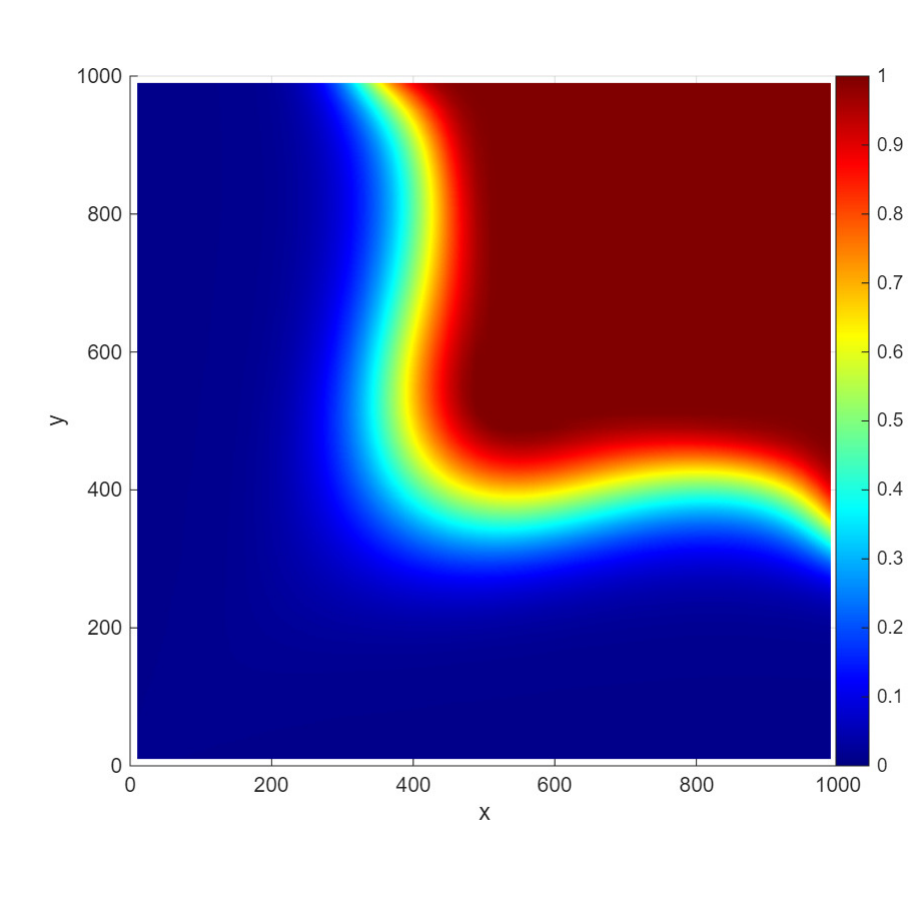}
    \includegraphics[width=0.24\textwidth]{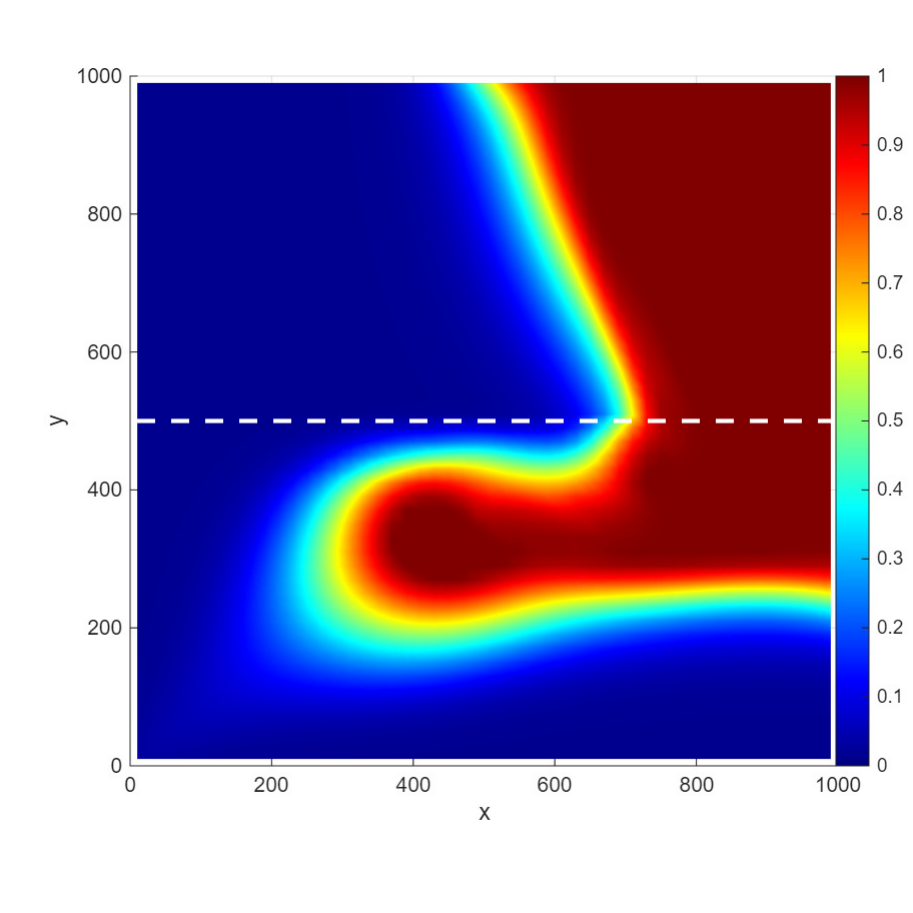}
    \includegraphics[width=0.24\textwidth]{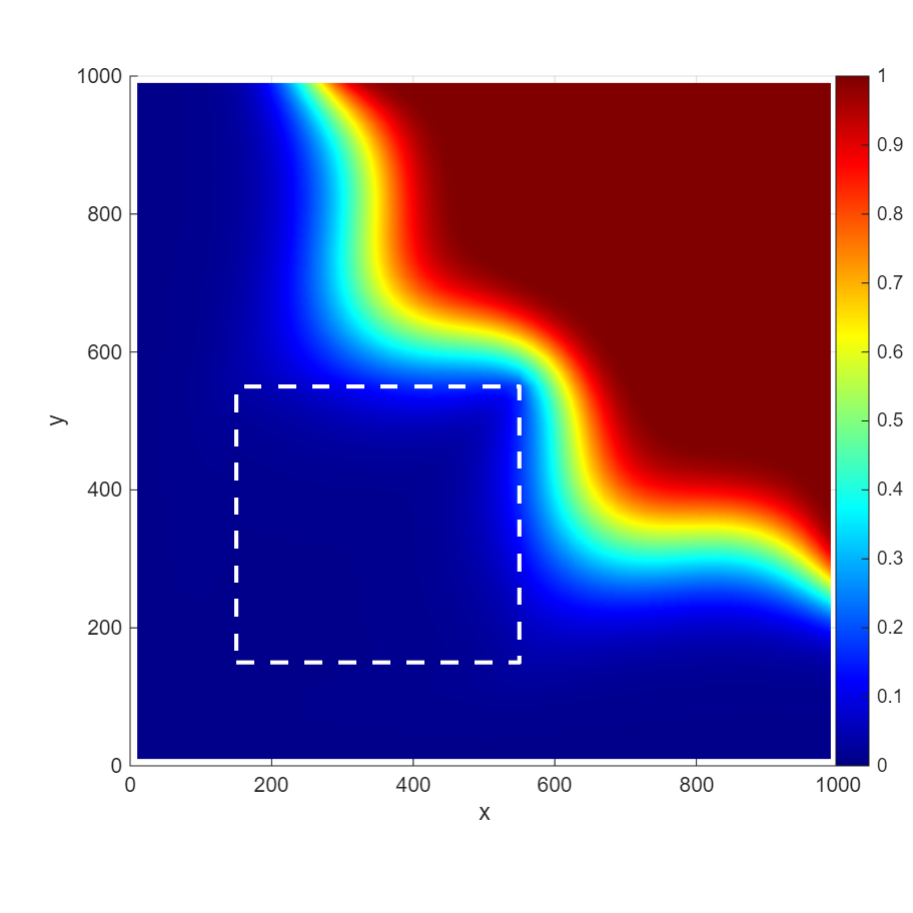}    
    \includegraphics[width=0.24\textwidth]{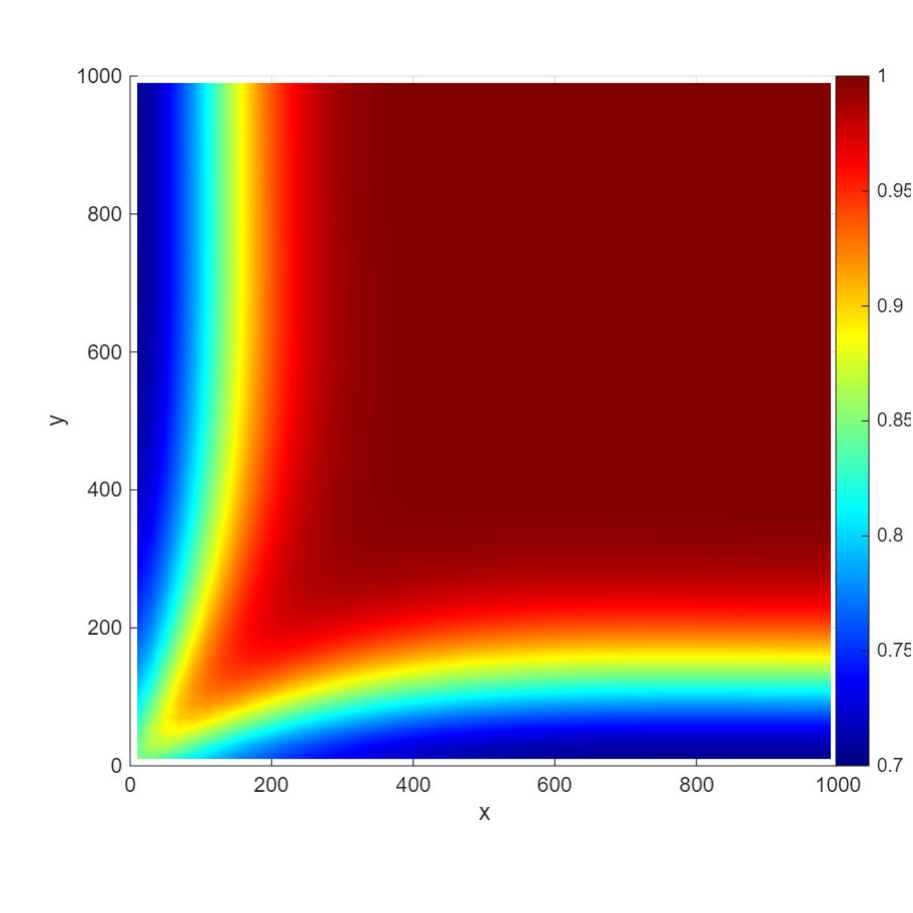}
    \includegraphics[width=0.24\textwidth]{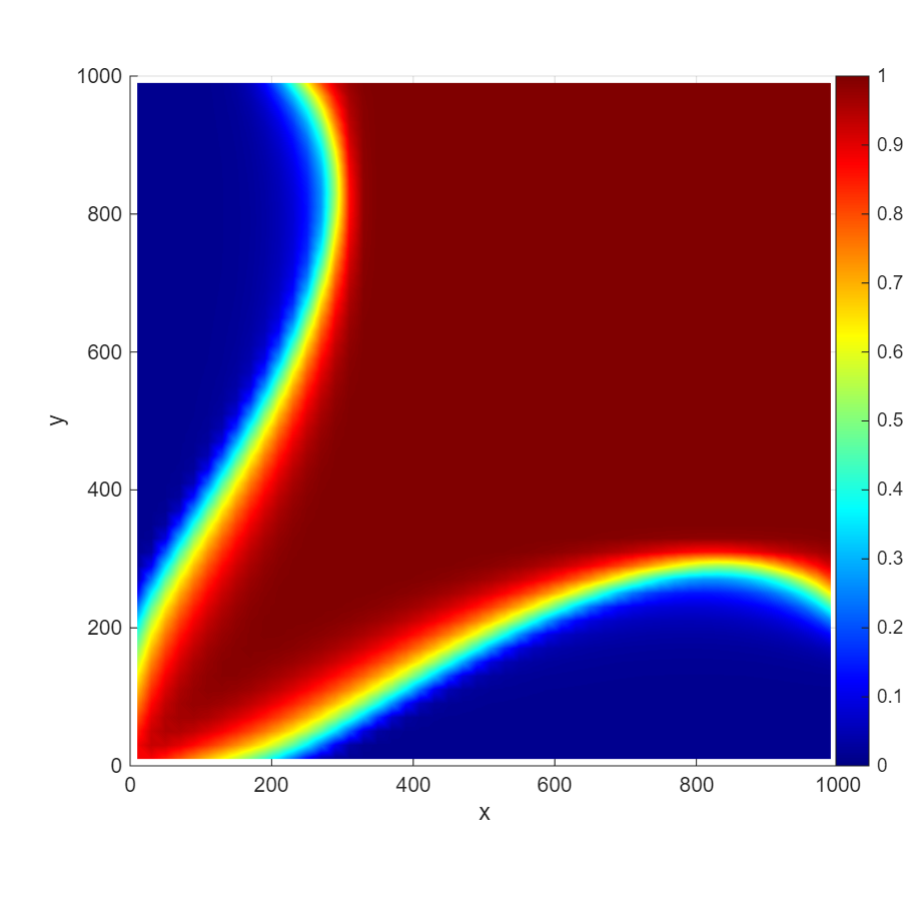}
    \includegraphics[width=0.24\textwidth]{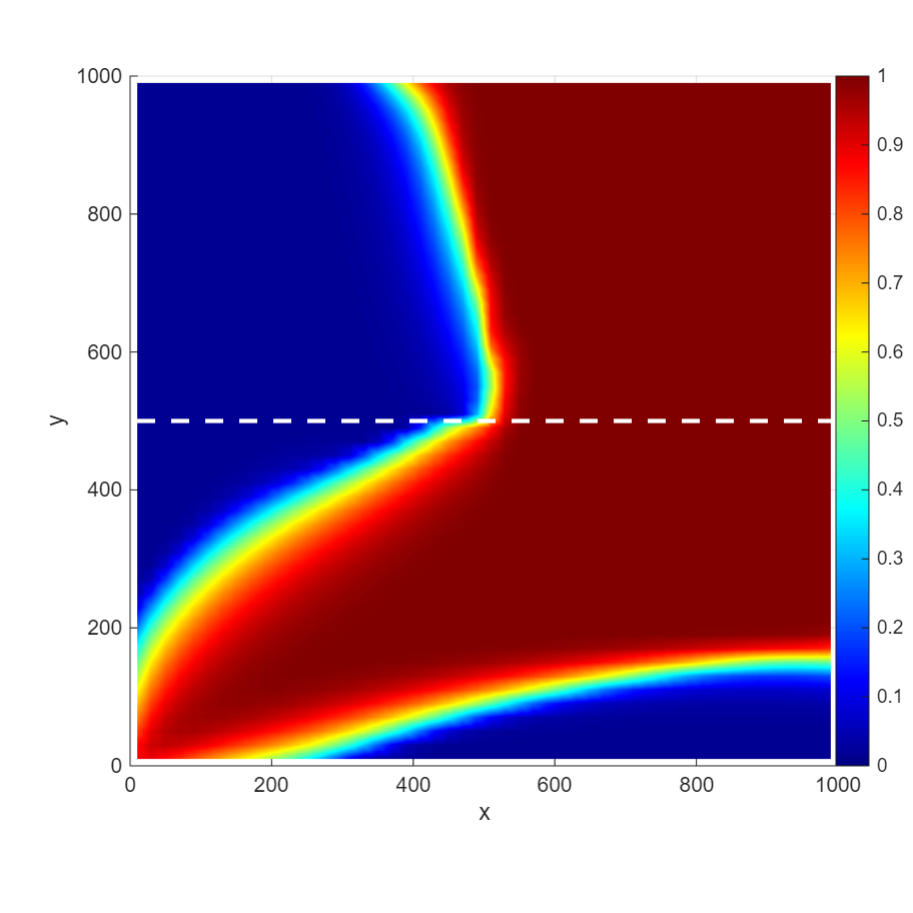}
    \includegraphics[width=0.24\textwidth]{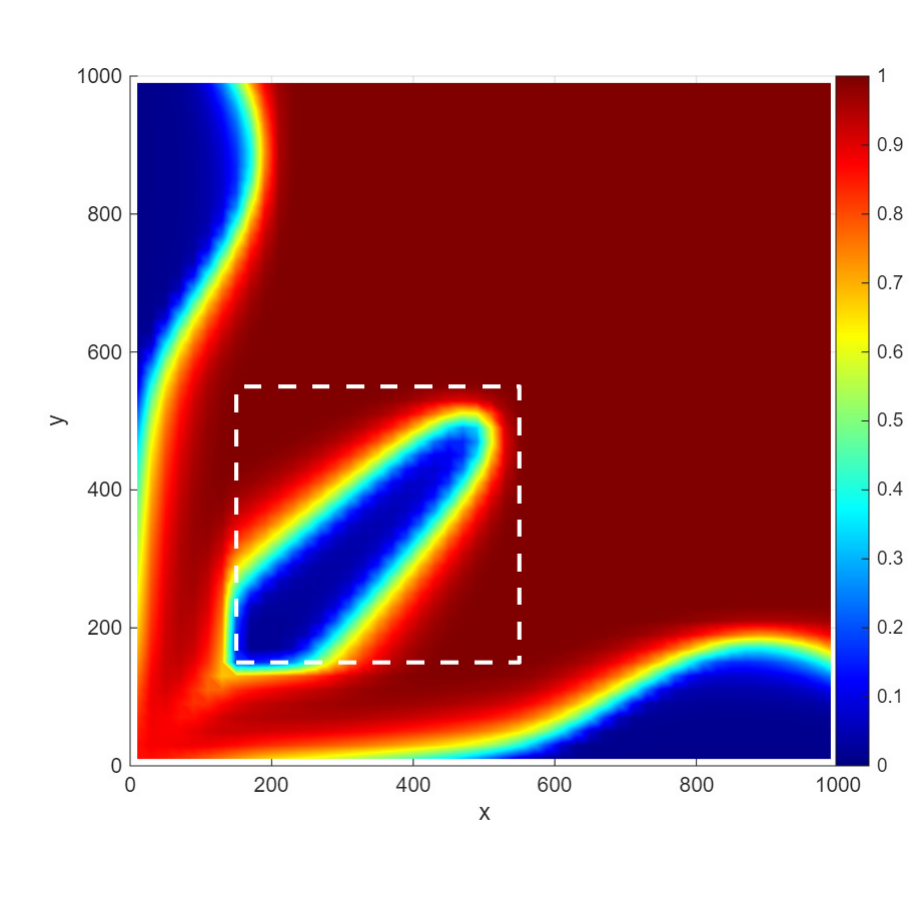}
    \setlength{\abovecaptionskip}{0.0cm} 
	  \setlength{\belowcaptionskip}{0.0cm}
	\caption{(Cases I-IV in Example \ref{exm:e2} (from left to right)) The concentration of the invading fluid at $t=3$ and $10$ years obtained by the CN-BCFD-P scheme  (from top to bottom), respectively.}
    \label{fig:e2-c}
\end{figure}

We first examine the displacement process simulated by the proposed CN-BCFD-P scheme for the four cases. Figure \ref{fig:e2-c} displays the concentration distributions of the invading fluid at $t=3$ and $t=10$ years for all cases. Several observations can be made:
\begin{itemize}
    \item In Case I, the concentration fields at $t=3$ years exhibit a concentric circular pattern, as expected given the homogeneous permeability, constant mobility, and purely molecular diffusion. In addition, owing to the influence of the no-flow boundary conditions and the locations of the injection and production wells, as time marches, the invading fluid propagates faster along the diagonal of the reservoir, as clearly seen from the first column of Fig. \ref{fig:e2-c} at $t=10$ years.

    \item In Case II, the large adverse mobility ratio causes a sharp variation in $\mu(c)$ (cf. \eqref{mod:mu}) across the fluid interface, leading to a rapid change in the Darcy velocity. As a result, viscous fingering develops, as shown in the second column of Fig. \ref{fig:e2-c}: the concentration front advances much faster along the diagonal than in Case I, causing early water breakthrough, a reduced swept area, and lower recovery.

    \item In Case III, the front initially advances faster vertically than horizontally, owing to the higher permeability in the lower domain $\Omega_{L}$, which yields a greater Darcy velocity. Subsequently, upon reaching $\Omega_{L}$, the invading fluid propagates horizontally at a significantly higher speed than  in the upper domain $\Omega_{U}$. 

    \item In Case IV, the permeability contrast creates a region that cannot be completely swept by the invading fluid. As shown in the fourth column of Fig. \ref{fig:e2-c}, as time evolves, the invading fluid advances from the injection well towards the production well; however, the low-permeability region $\Omega_{1}$ acts as a barrier. Owing to the lower flow resistance in the high-permeability region $\Omega_{2}$, the fluid preferentially migrates along these high-speed channels, leading to fingering. Consequently, two fingers develop that bypass the barrier and eventually reach the production well. 

    \item Compared to Case II, Case IV exhibits significantly higher sweep efficiency, leading to improved displacement and enhanced oil recovery. Therefore, it is recommended that production wells should be placed in low-permeability regions to maximize sweep efficiency. Furthermore, this observation also underpins the use of polymers in enhanced oil recovery: their high viscosity enables selecting blocking or permeability reduction in specific pores or flow zones, thereby redirecting fluid flow to optimize hydrocarbon recovery.
\end{itemize}

We next examine the preservation of physical structures, focusing on the bound-preserving property and mass conservation. Fig. \ref{fig:e2} presents the evolution of the extrema of projection concentration $C_{h}^{n}$ and predicted concentration $\widehat{C}_{h}^{n}$, the mass-conserving Lagrange multiplier $\gamma^{n}$, the number of iterations, and the relative mass errors, obtained by the CN-BCFD-P scheme. The results further demonstrate that the CN-BCFD-P scheme preserves both concentration bounds and mass conservation, whereas its prediction step \eqref{scm:cn:c} fails to maintain the bounds. Moreover, the secant method employed for computing $\gamma^{n}$ takes only $2 \sim 5$ iterations, confirming the efficiency of the proposed weighted $L^{2}$-projection.
\begin{figure}[!ht]
	\centering
    \includegraphics[width=0.24\textwidth]{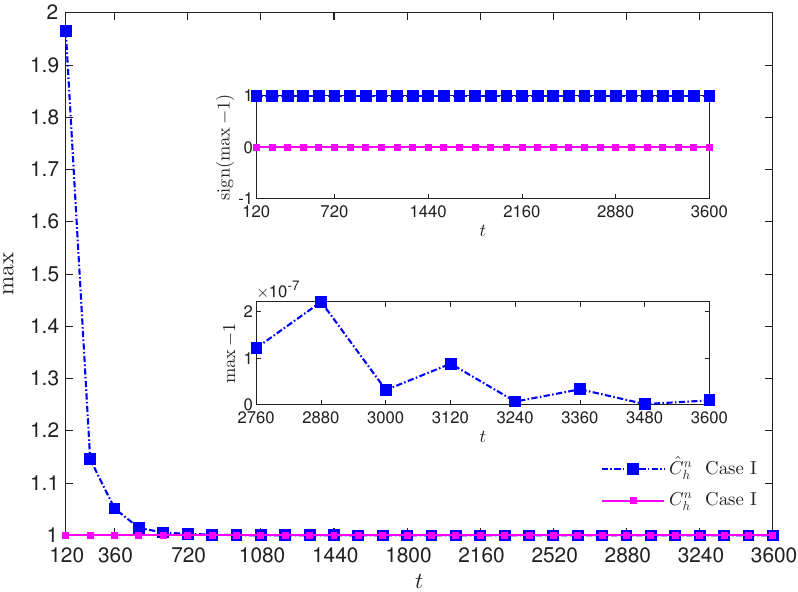}
    \includegraphics[width=0.24\textwidth]{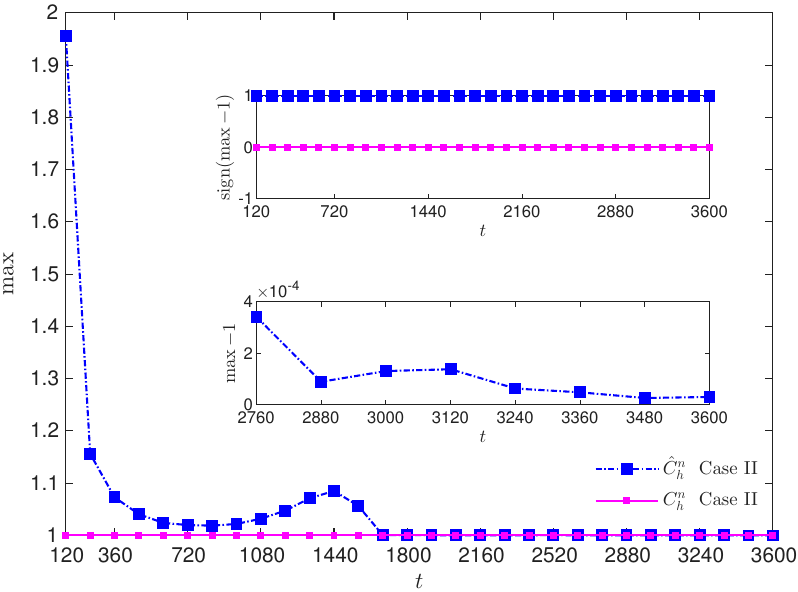}
    \includegraphics[width=0.24\textwidth]{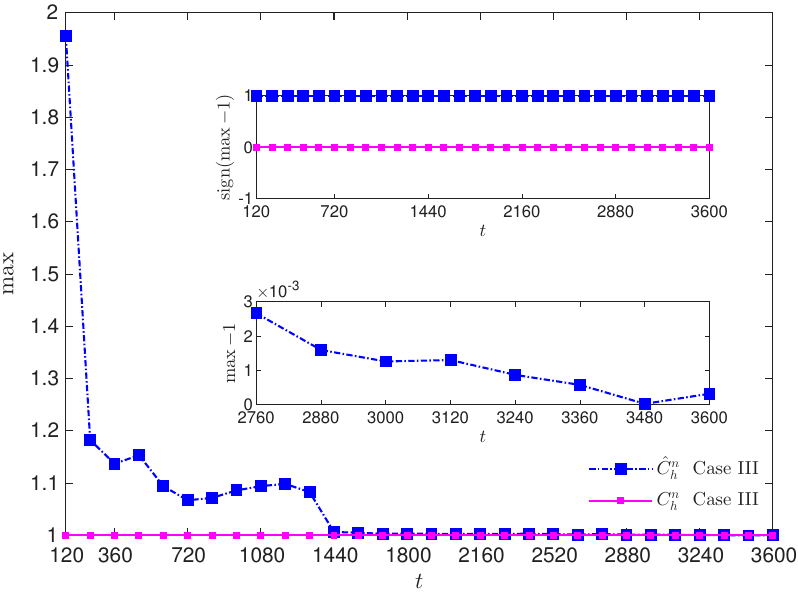}
    \includegraphics[width=0.24\textwidth]{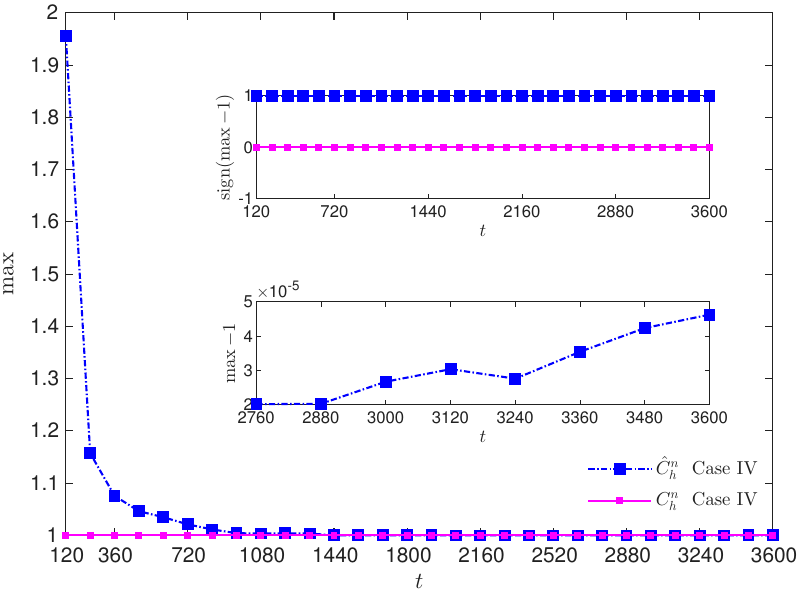}
    \includegraphics[width=0.24\textwidth]{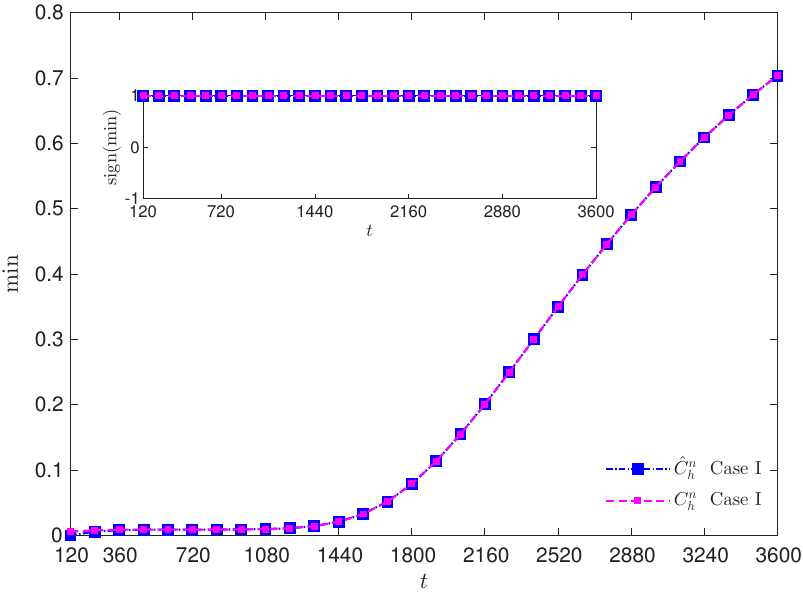}
    \includegraphics[width=0.24\textwidth]{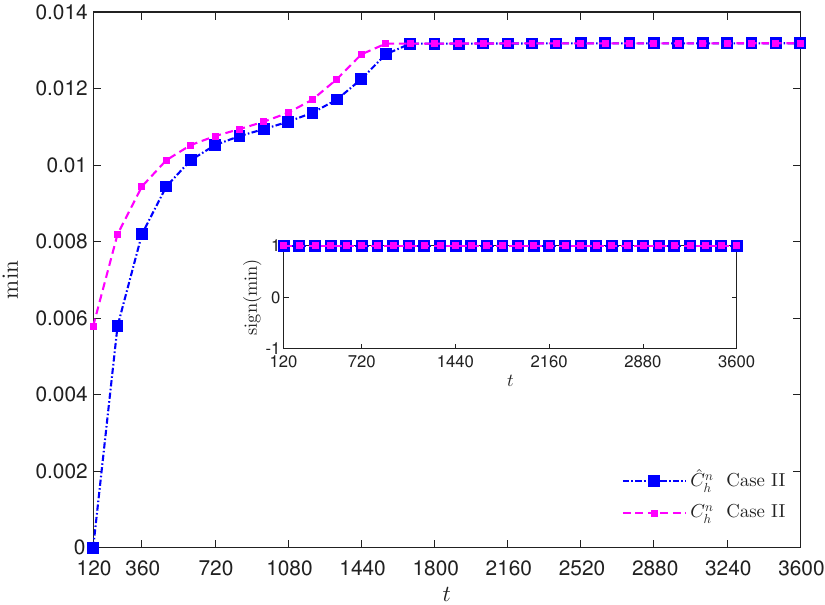}
    \includegraphics[width=0.24\textwidth]{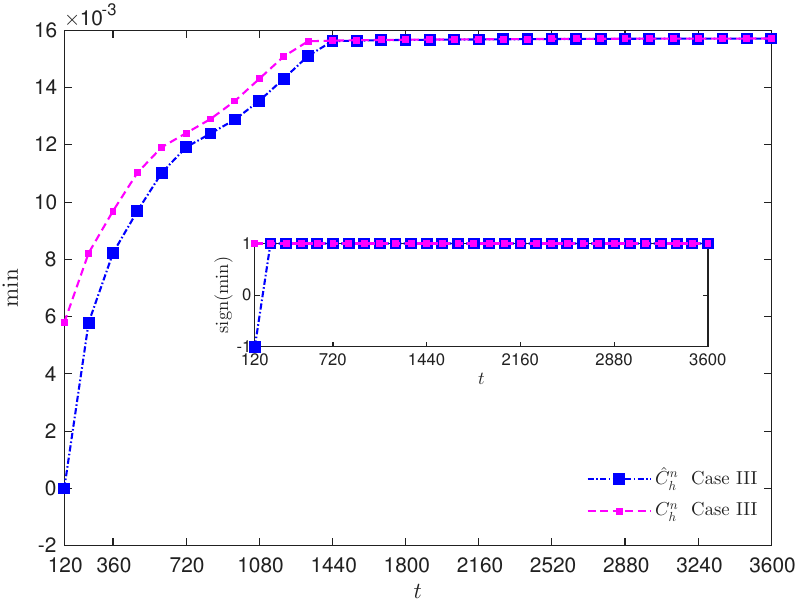}
    \includegraphics[width=0.24\textwidth]{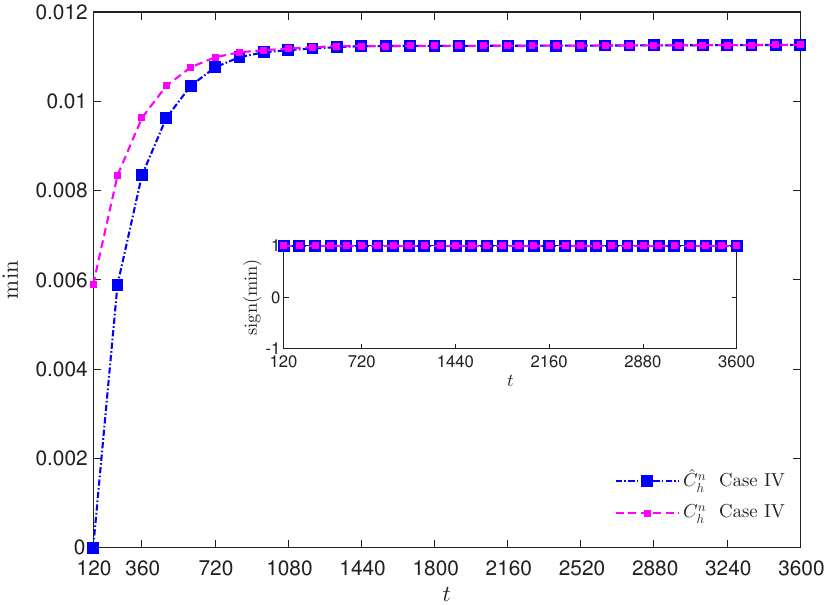}
    \includegraphics[width=0.24\textwidth]{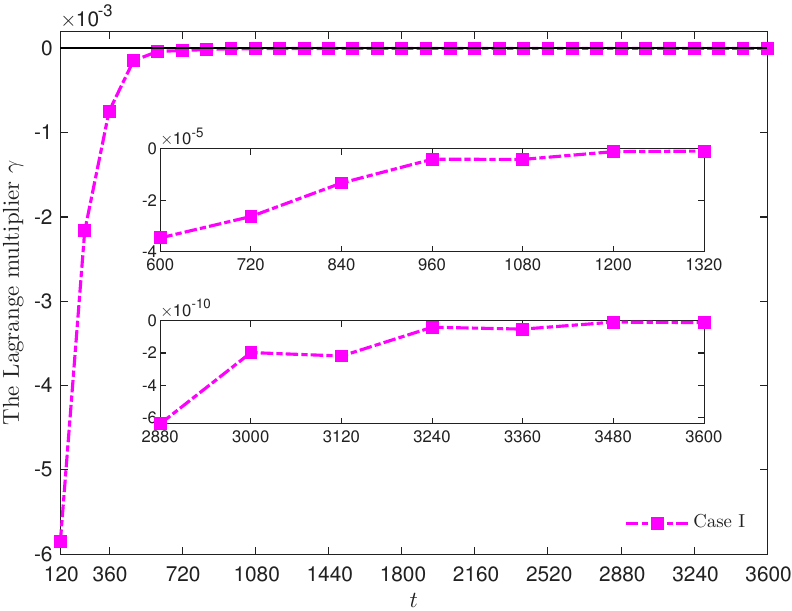}
    \includegraphics[width=0.24\textwidth]{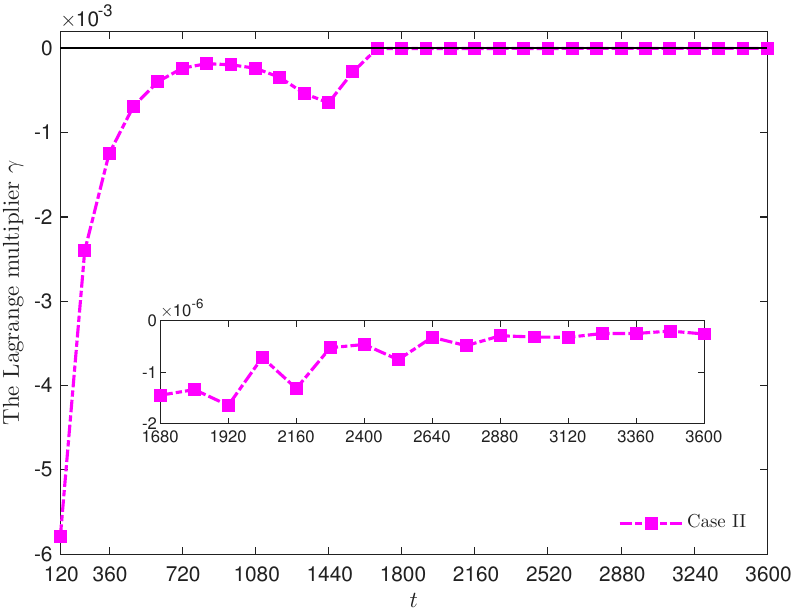}
    \includegraphics[width=0.24\textwidth]{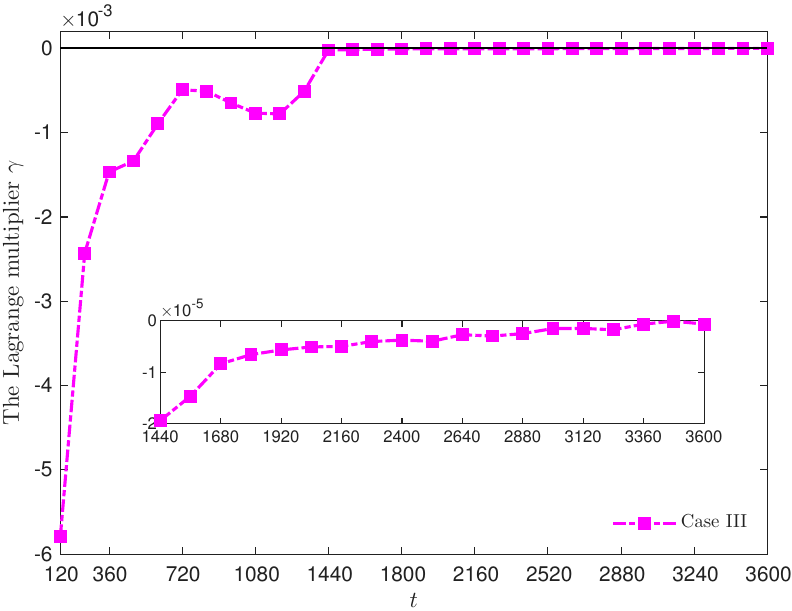}
    \includegraphics[width=0.24\textwidth]{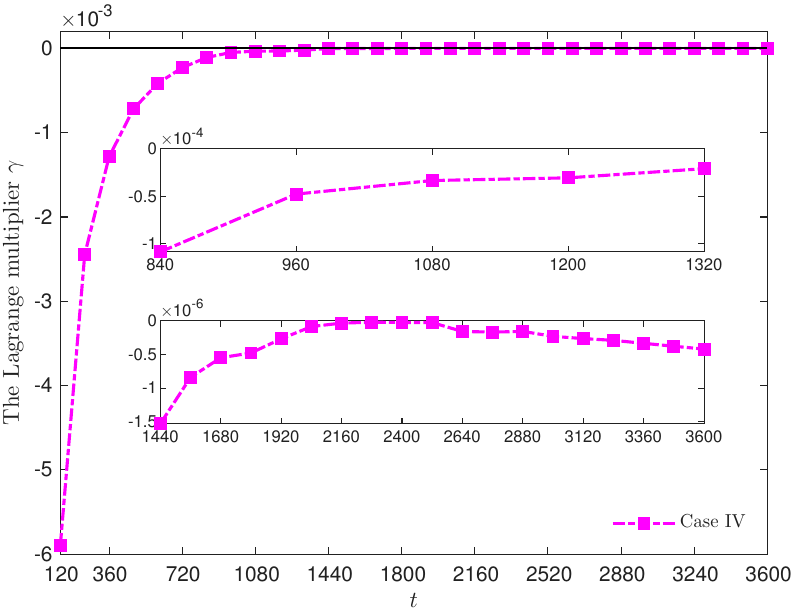}
    \includegraphics[width=0.24\textwidth]{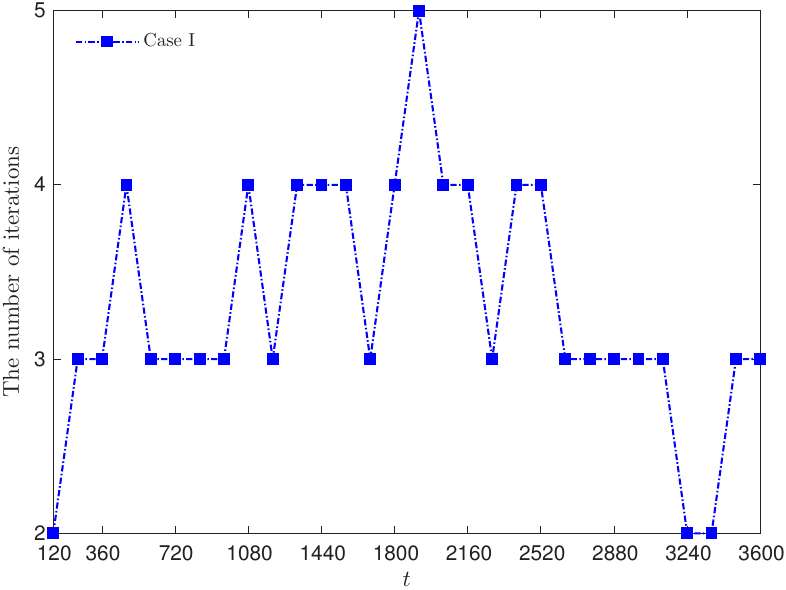}
    \includegraphics[width=0.24\textwidth]{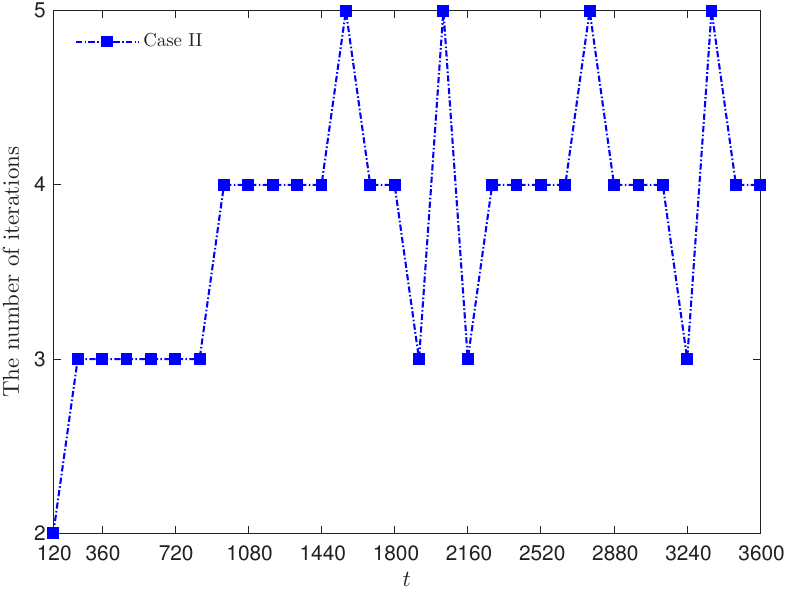}
    \includegraphics[width=0.24\textwidth]{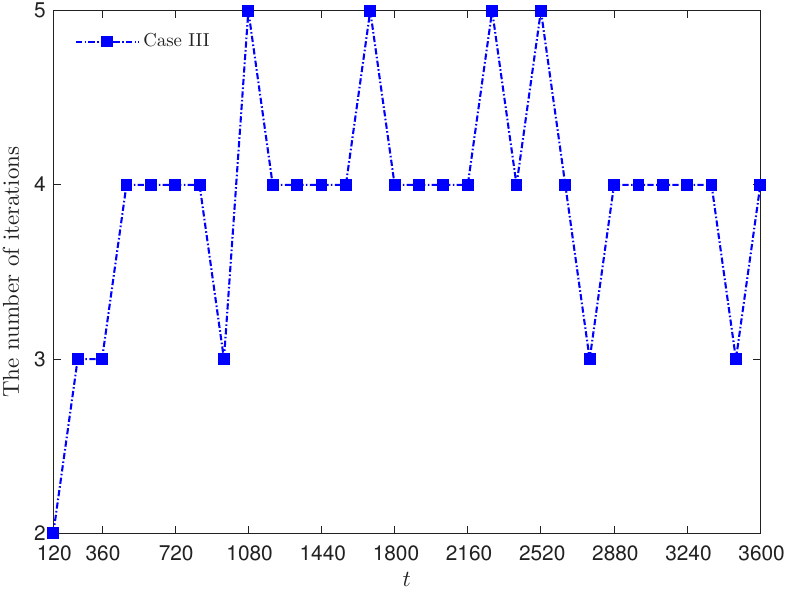}
    \includegraphics[width=0.24\textwidth]{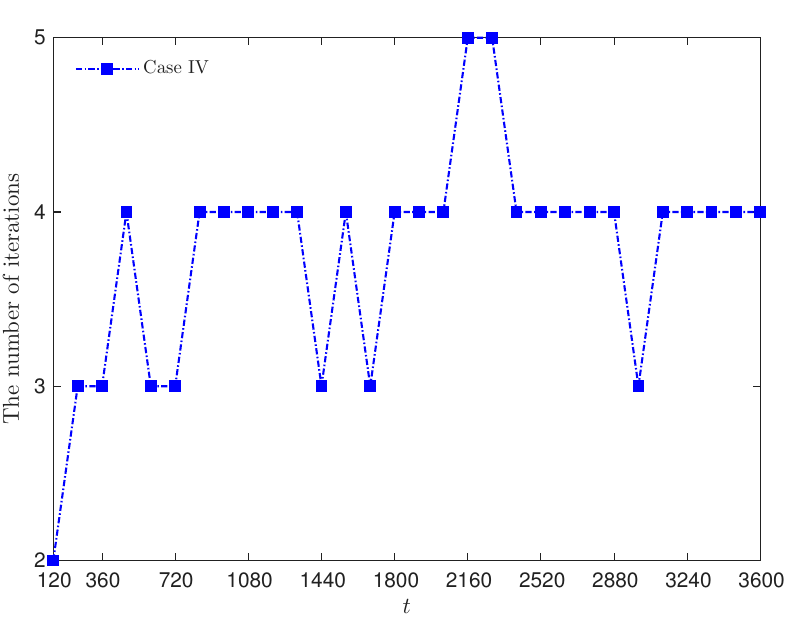}
    \includegraphics[width=0.24\textwidth]{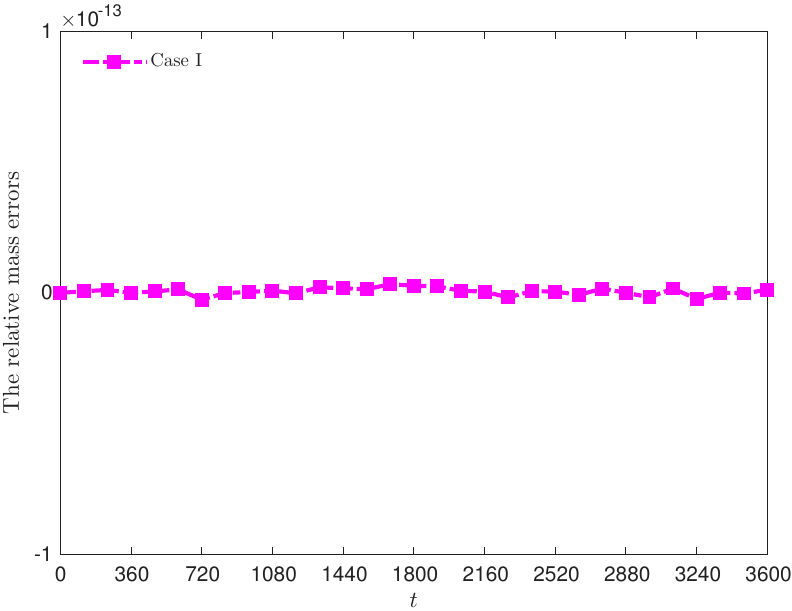}
    \includegraphics[width=0.24\textwidth]{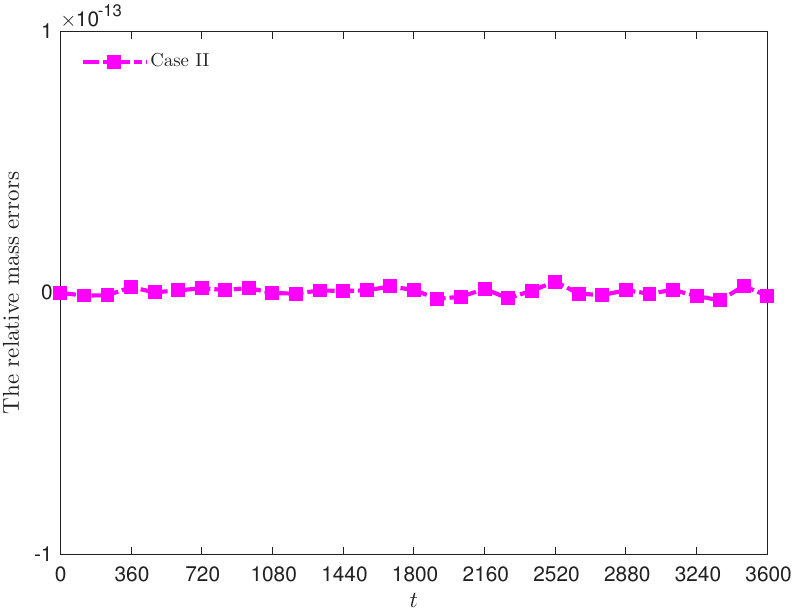}
    \includegraphics[width=0.24\textwidth]{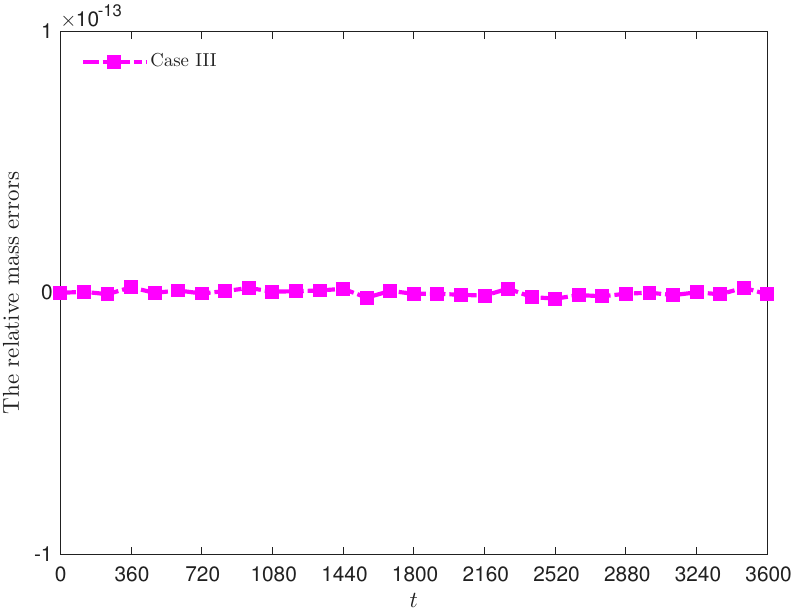}
    \includegraphics[width=0.24\textwidth]{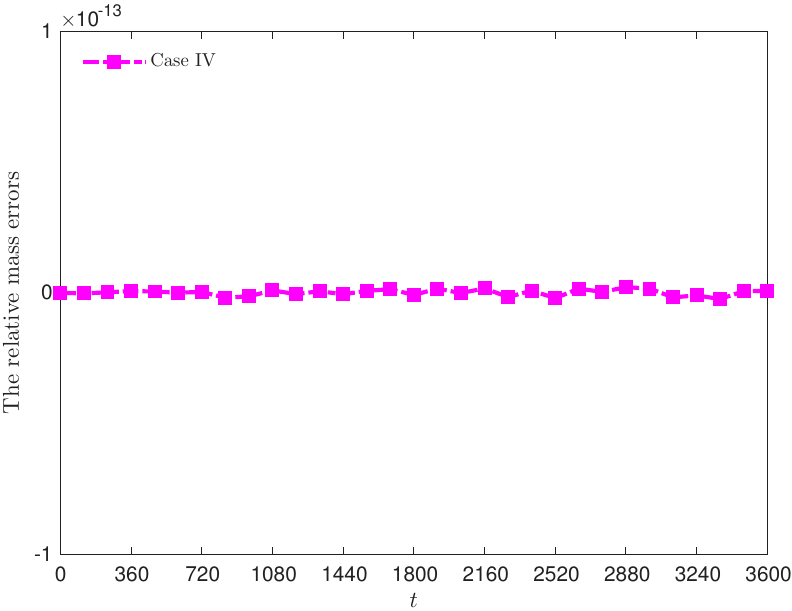}
    \caption{(Cases I--IV in Example \ref{exm:e2} (from left to right)) Evolution of the extrema of $C_{h}^{n}$ and $\widehat{C}_{h}^{n}$, the Lagrange multiplier $\gamma^{n}$, the number of iterations and the relative mass  errors (from top to bottom).}
    \label{fig:e2}
\end{figure}

In summary, the numerical experiments in Example \ref{exm:e2} demonstrate that the proposed CN-BCFD-P scheme can effectively simulate fluid flow in highly complex porous media. Even with relatively coarse spatial and temporal grid discretizations, the method yields accurate and physically reasonable numerical solutions, while also offering practical guidance for predicting and enhancing oil recovery. Moreover, the computed results and observed flow behaviors are consistent with those reported in \cite{WH00,CD07,VPV21,Eymard19,WRF26}, further confirming the reliability of the proposed discretization approach.
    
\begin{example}\label{exm:e3}   In this example, we aim to illustrate the implementation issue discussed in Remark \ref{rem:extrap}, which concerns reversing the sequence of the second-order CN-BCFD-P scheme, i.e., solving for velocity and pressure before the concentration. We comment that this reversed order may cause severe computational difficulties and even lead to simulation failure. Specifically, the alternative second-order scheme is as follows: 

\medskip
\textbf{Step 1:} (Velocity \& pressure update). Solve the velocity $\bm{U}_{h}^{n} \in \mal{U} \times \mal{V}$ and pressure $P_{h}^{n} \in \mal{C}$ through the linear algebraic system:
\begin{equation}\label{scm:cn2:u}
	\nabla_{h} \cdot \bm{U}_{h}^{n} = q^n,
\quad 
	\bm{\mal{A}}(\widetilde{C}_{h}^{n}) \bm{U}_{h}^{n} +  \nabla_{h} P_{h}^{n} = 0,
\end{equation}
where
$$
\widetilde{C}_{h}^{n}:= \left\{
\begin{aligned}
    &C_{h}^{1,*}, & n=1,\\
    &2C_{h}^{n-1} - C_{h}^{n-2}, & n \geq 2,
\end{aligned}
\right.
$$
and $C_{h}^{1,*}$ is the solution of BE-BCFD-P scheme at time $t=t^{1}$. In addition, to ensure the uniqueness of the solution $P_{h}^{n}$, we still enforce $P_{1,1}^{n}=0$.

\textbf{Step 2:} (Prediction of concentration). Solve the intermediate concentration $\widehat{C}_{h}^{n}$ via the CN-BCFD method
\begin{equation}\label{scm:cn2:c}
	\begin{aligned}
		\phi D_{\tau}\widehat{C}_{h}^{n} & + \f{1}{2}\big[ \nabla_{h} \cdot \big(  \bm{\mal{B}}(\widehat{C}_{h})\bm{U}_{h} - \bm{\mal{D}}(\bm{U}_{h}) \nabla_{h}\widehat{C}_{h}  \big) - q_{P}\widehat{C}_{h} \big]^{n}  \\
		& + \f{1}{2} \big[ \nabla_{h} \cdot \big(  \bm{\mal{B}}(C_{h})\bm{U}_{h} - \bm{\mal{D}}(\bm{U}_{h}) \nabla_{h}C_{h}  \big) - q_{P}C_{h} \big]^{n-1}  = \big[q_{I}c_{I}\big]^{n-1/2},
	\end{aligned}
\end{equation}
enclosed with homogeneous boundary conditions 
$$ 
\Big( \big[ \bm{\mal{B}}(\widehat{C}_{h})\bm{U}_{h} - \bm{\mal{D}}(\bm{U}_{h}) \nabla_{h}\widehat{C}_{h}   \big]^{n} + \big[   \bm{\mal{B}}(C_{h})\bm{U}_{h} - \bm{\mal{D}}(\bm{U}_{h}) \nabla_{h}C_{h}  \big]^{n-1} \Big) \cdot \bm{n}=0.
$$ 

\textbf{Step 3:} (Correction of concentration). Project the nodal solution $\widehat{C}_{h}^{n}$ to the constrained manifold with mass conservation and bound preservation via the weighted $L^{2}$-projection such that
\begin{equation}\label{scm:cn2:min}
	\begin{aligned}
		&\quad \min_{C_{h}^{n} \in \mal{C}}~ \f{1}{2} \|C_{h}^{n} - \widehat{C}_{h}^{n}\|_{{\rm M},\bar{\beta}^{n}}^{2}\\
		& \text{s.t.} \quad 0 \leq C_{h}^{n} \leq 1, \quad \bar{G}(C_{h}^{n})=0.
	\end{aligned}
\end{equation}
\end{example}
 
In the simulations, we choose $\mal{M}=1/16$ and adopt all other parameters from the four cases of Example \ref{exm:e2}. Figure \ref{fig:e3} displays the distributions of the extrapolated concentration $\widetilde{C}_h^{2}$, defined by \eqref{scm:cn2:u}, of the invading fluid at $t=t^{2}=240$ days for Cases I--IV, computed using the alternative second-order scheme \eqref{scm:cn2:u}--\eqref{scm:cn2:min}. We observe that for all four cases, the numerical extrapolated concentrations $\widetilde{C}_h^{2}$ obtained by \eqref{scm:cn2:u}--\eqref{scm:cn2:min} not only exceed the upper bound $1$, but also reach the value $2$ at certain spatial points. As a result, the term $\mal{M}^{1/4}\widetilde{C}_h^{2} + 1 - \widetilde{C}_h^{2}$ in the viscosity equation \eqref{mod:mu} vanishes, rendering the problem \eqref{scm:cn2:u} completely ill-posed. For clearer illustration, the corresponding viscosity distributions $\mu(\widetilde{C}_{h}^{2})$ for the four cases at $t=240$ days are also displayed in Fig. \ref{fig:e3}. Notably, in the upper-right region where the extrapolated concentration reaches $\widetilde{C}_h^{2}=2$, the viscosity $\mu(\widetilde{C}_h^{2})$ becomes unbounded. As a consequence, infinite values appear in the coefficient matrix of the velocity-pressure system \eqref{scm:cn2:u}, giving rise to an arithmetic singularity that causes the linear solver to fail and the program to terminate prematurely. This observation clearly demonstrates the importance of the computational sequence and corroborates the assertion made in Remark \ref{rem:extrap}.
\begin{figure}[!ht]
	\centering
    \includegraphics[width=0.24\textwidth]{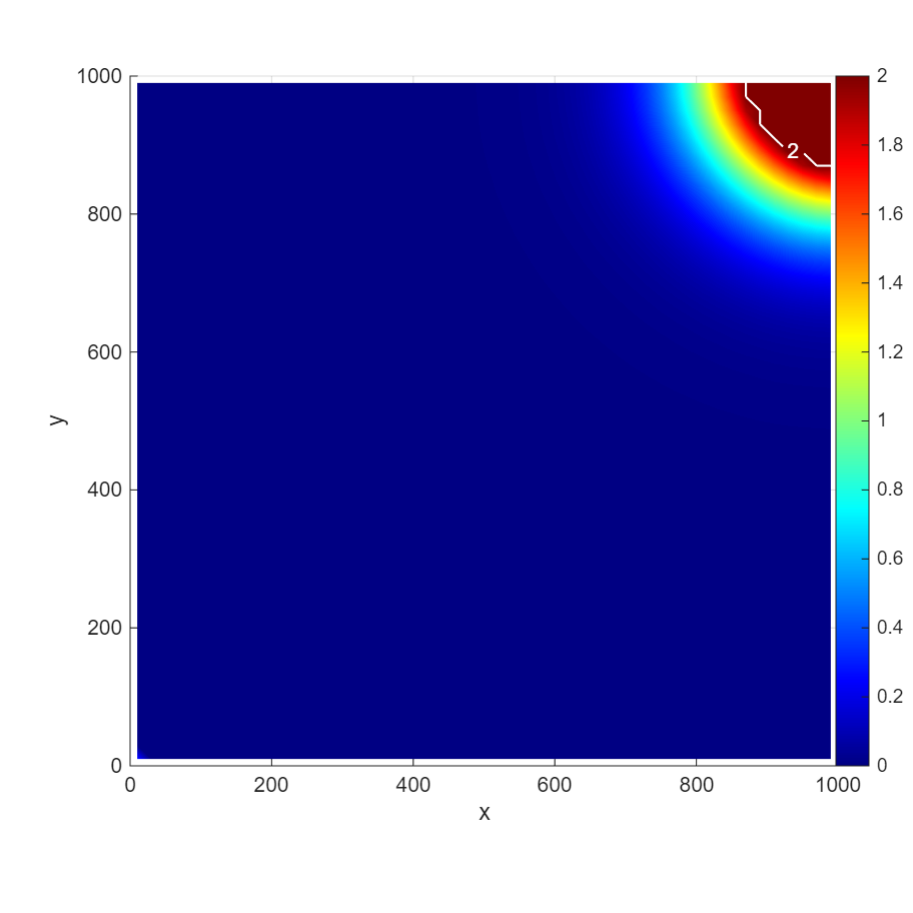}
    \includegraphics[width=0.24\textwidth]{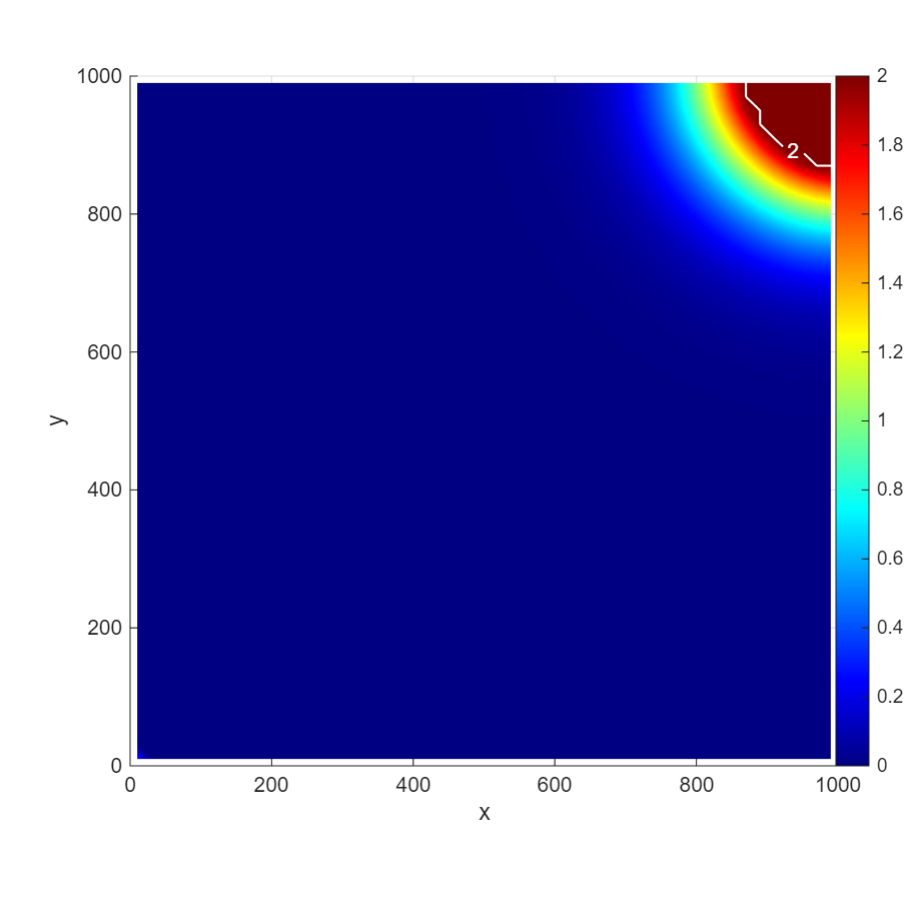}
    \includegraphics[width=0.24\textwidth]{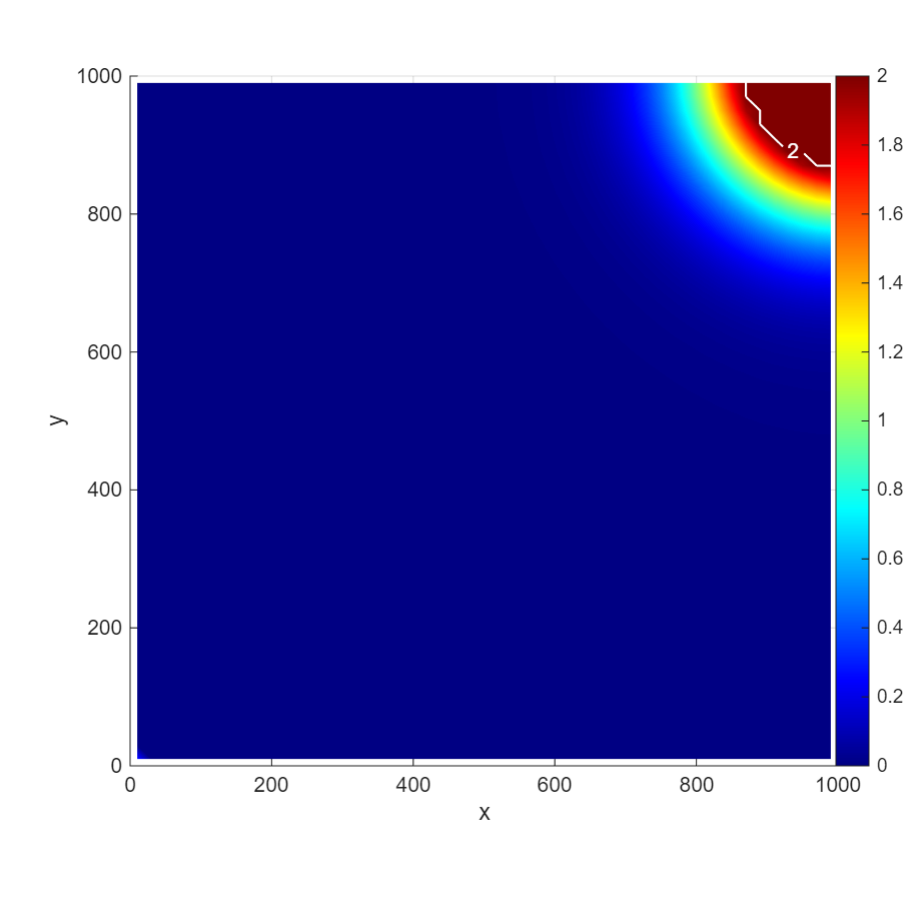}
    \includegraphics[width=0.24\textwidth]{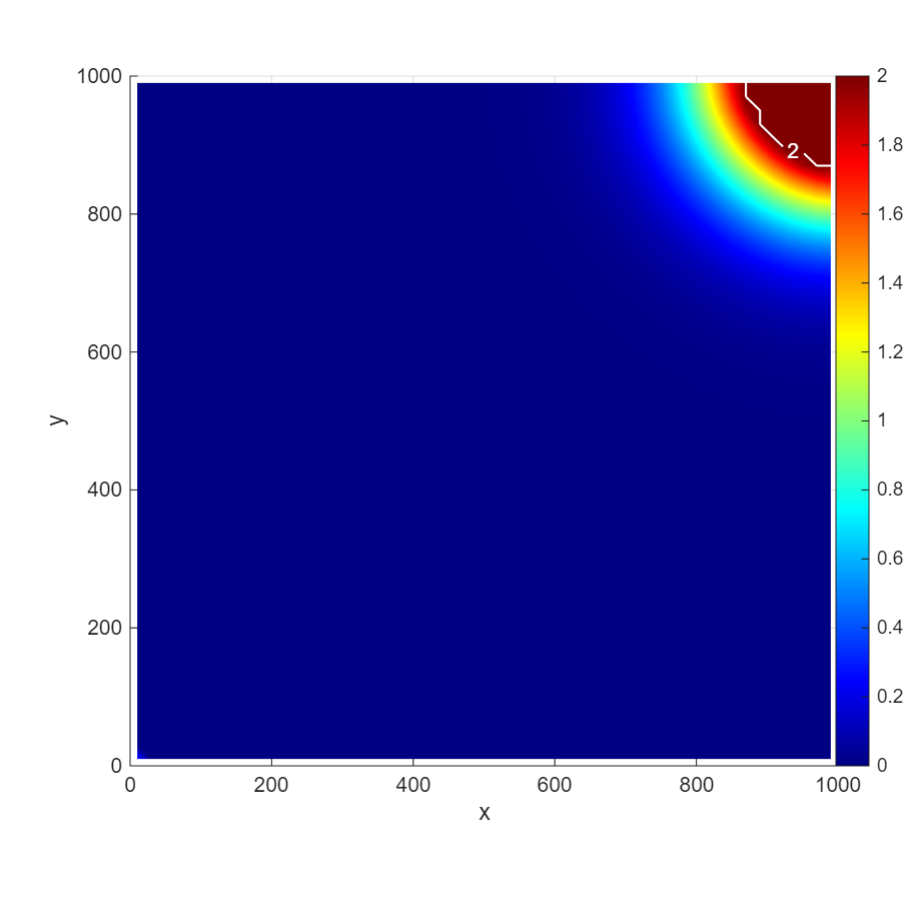}
    \includegraphics[width=0.24\textwidth]{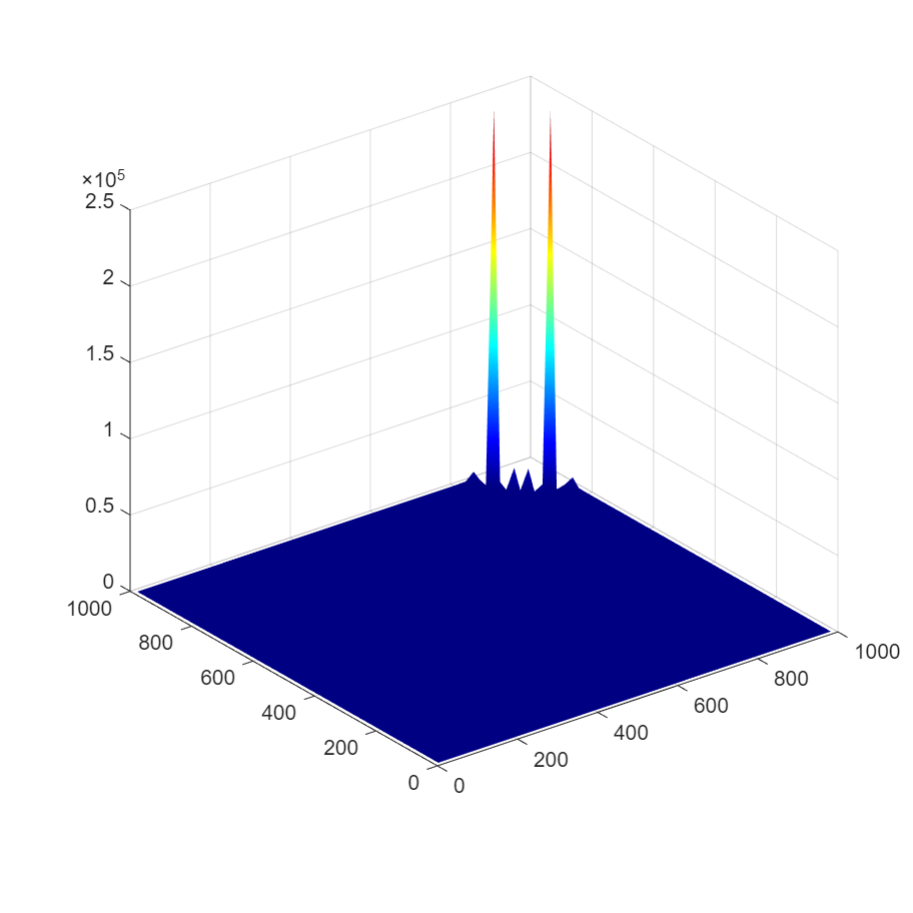}
    \includegraphics[width=0.24\textwidth]{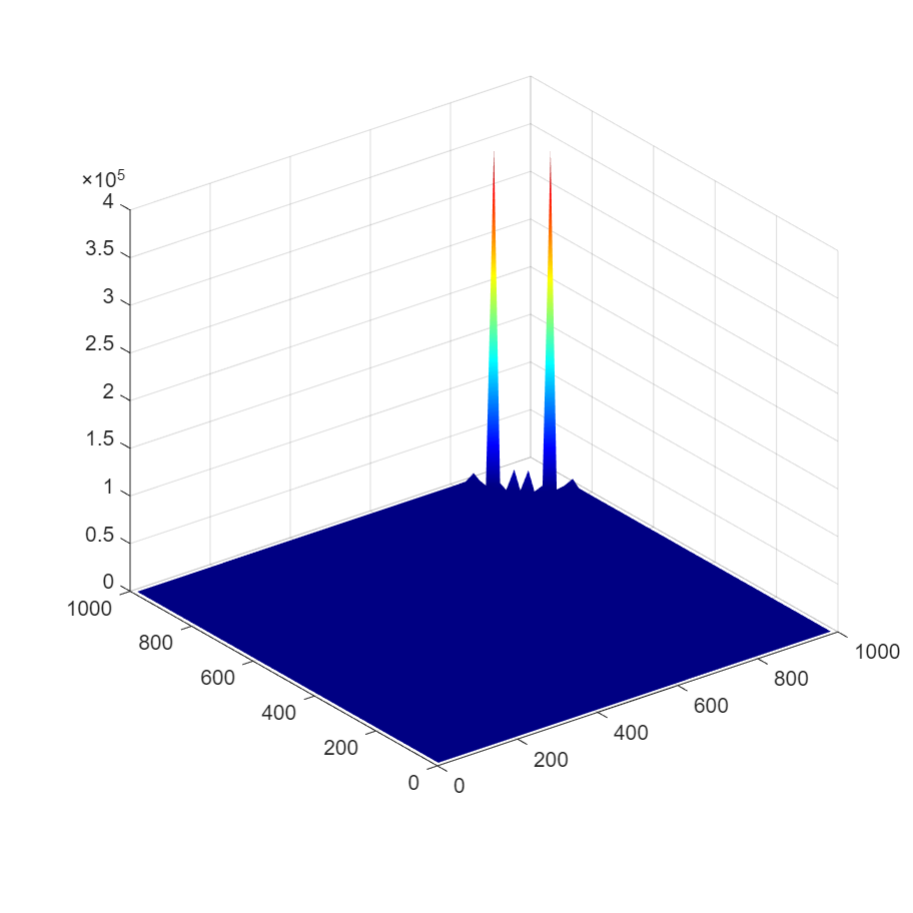}
    \includegraphics[width=0.24\textwidth]{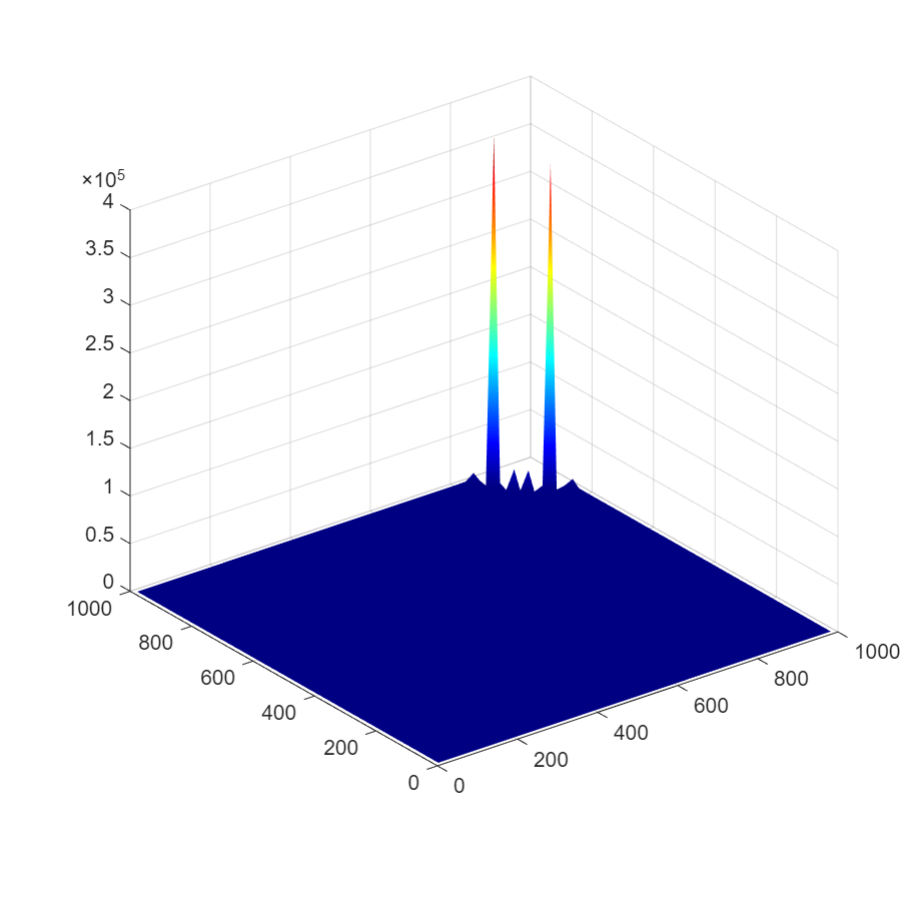}
    \includegraphics[width=0.24\textwidth]{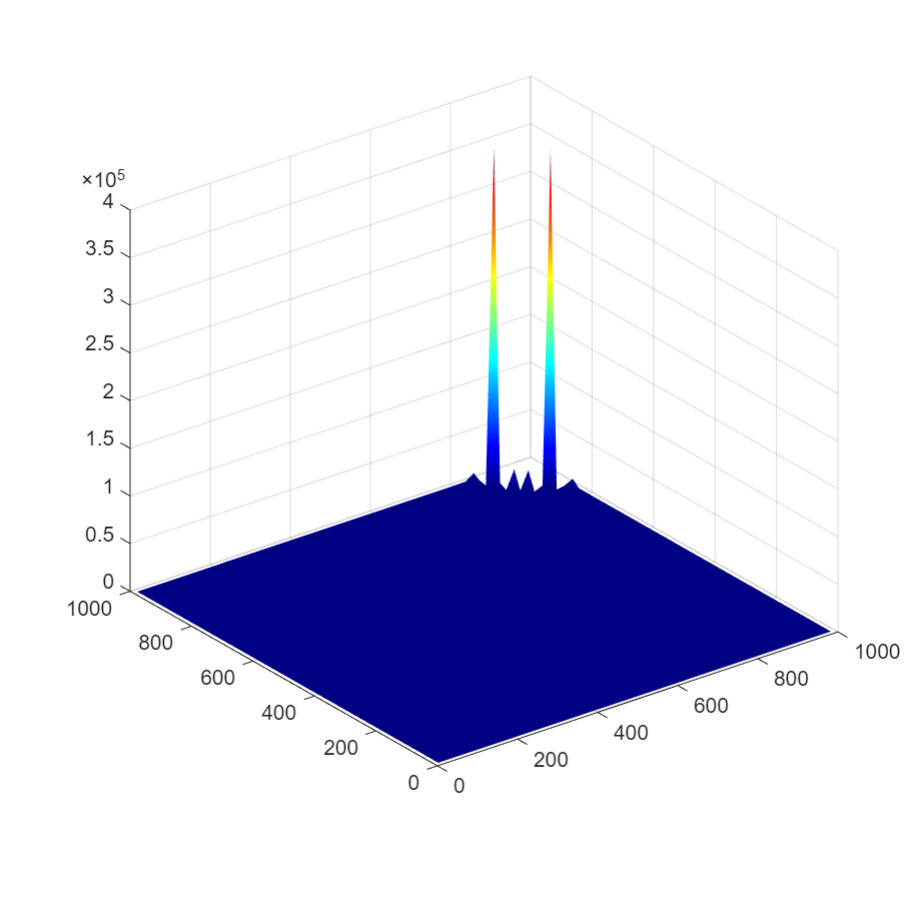}
    \setlength{\abovecaptionskip}{0.0cm} 
	\setlength{\belowcaptionskip}{0.0cm}
    \caption{ (Cases I--IV in Example \ref{exm:e3} (from left to right)) The extrapolated concentration $\widetilde{C}_h$ of the invading fluid and the viscosity $\mu(\widetilde{C}_h)$ at $t=240$ days obtained by the second-order scheme \eqref{scm:cn2:u}--\eqref{scm:cn2:min} (from top to bottom).}
    \label{fig:e3}
\end{figure}

\section{Conclusions}\label{sec:conclusion}
In this work, we have developed two linear, fully decoupled, structure‑preserving prediction-then-projection BCFD schemes for incompressible miscible displacement in porous media. The projection approach enforces both bound preservation and mass conservation at the discrete level, while different update orders for the velocity-pressure and concentration systems are adopted in the temporally first- and second-order schemes to guarantee well‑posedness and numerical stability. 

A major challenge in the error analysis arises from the time-dependent total mass of the present model (cf. \eqref{mod:mass_conservation}), which distinguishes it from the constant-mass settings considered in \cite{CYY24,CYY26,ZZF26} and precludes a direct extension of their analysis tools. To fill the gap, we have developed a tailored error-analysis framework  comprising three essential components: (i) the non-expansive property of the weighted $L^{2}$-projections (Lemmas \ref{lem:hatec:be} and \ref{lem:hatec:cn}); (ii) a novel "biased" auxiliary solution whose deviation from the exact solution is of the same order as the numerical error (Lemmas \ref{lem:alpha:be} and \ref{lem:alpha:cn}); and (iii) the crucial relation between the "biased" errors and the original errors (Lemmas \ref{lem:be:ec} and \ref{lem:cn:ec}). Building on these results, we have established rigorous optimal-order error estimates for the proposed BCFD projection algorithms applied to the incompressible miscible displacement problem in porous media (Theorems \ref{thm:err:be} and \ref{thm:err:cn}). The proofs rely on mathematical induction and discrete energy analysis, with the aforementioned lemmas playing an indispensable role.

Extensive numerical experiments have confirmed the accuracy, bound preservation, and mass conservation of the proposed schemes, and have further demonstrated their capability to capture the viscous fingering phenomenon in a standard five-spot well pattern. The extension of the proposed methodology and error-analysis framework to compressible miscible displacement \cite{GH17,GH22,GH23} and thermodynamically consistent immiscible two-phase flow \cite{SSY22,SSY24} is left for future work.

\section*{CRediT authorship contribution statement}
\textbf{Xiaoying Wang}: Methodology, Formal analysis, Software, Writing- Original draft.
\textbf{Shusen Xie}: Methodology, Writing- Reviewing and Editing, Funding acquisition.
\textbf{Hongfei Fu}: Conceptualization, Supervision, Methodology, Writing- Reviewing and Editing, Funding acquisition.

\section*{Declaration of competing interest}
The authors declare that they have no competing interests.

\section*{Data availability } Data will be made available on request.

 \section*{Acknowledgments}
This work was partially supported by the National Natural Science Foundation of China under grant 12131014, by the Shandong Provincial Natural Science Foundation under grant ZR2024MA023, and by the Shandong Provincial Key Laboratory of Stochastic System Control and Scientific Computing.

\appendix

\section{The weighted $H^{1}$-projection and its implementation}\label{App:A}
The weighted $L^{2}$-projections employed in Step 3 of the BE-BCFD-P scheme and Step 2 of the  CN-BCFD-P scheme can be replaced, respectively, by the following weighted $H^{1}$-projections:
\begin{equation}\label{scm:beh1:min}
	\begin{aligned}
		&\quad \min_{C_{h}^{n} \in \mal{C}}~ \f{1}{2} \|C_{h}^{n} - \widehat{C}_{h}^{n}\|_{{\rm M},\beta^{n}}^{2} + \f{1}{2} \| \nabla_{h} C_{h}^{n} - \nabla_{h} \widehat{C}_{h}^{n}\|_{{\rm M},\bm{\mal{D}}(\bm{U}_{h}^{n})}^{2}\\
		& \text{s.t.} \quad 0 \leq C_{h}^{n} \leq 1, \quad G(C_{h}^{n})=0,
	\end{aligned}
\end{equation}
and 
\begin{equation}\label{scm:cnh1:min}
	\begin{aligned}
		&\quad \min_{C_{h}^{n} \in \mal{C}}~ \f{1}{2} \|C_{h}^{n} - \widehat{C}_{h}^{n}\|_{{\rm M},\bar{\beta}^{n}}^{2} + \f{1}{2} \| \nabla_{h} C_{h}^{n} - \nabla_{h} \widehat{C}_{h}^{n}\|_{{\rm M},\bm{\mal{D}}(\widetilde{U}_{h}^{n-1/2})}^{2}\\
		& \text{s.t.} \quad 0 \leq C_{h}^{n} \leq 1, \quad \bar{G}(C_{h}^{n})=0.
	\end{aligned}
\end{equation}

The implementation of \eqref{scm:beh1:min} and \eqref{scm:cnh1:min} is analogous; hence, we take \eqref{scm:cnh1:min} as an example to illustrate the detailed implementation. The convex optimization problem \eqref{scm:cnh1:min} is equivalent to the following KKT conditions: 
\begin{equation}\label{scm:cnh1:k1}
	\bm{H}_{h} C_{h}^{n} = \bm{H}_{h} \widehat{C}_{h}^{n} + \lambda_h^{n} - \xi_h^{n} - \gamma^{n} \bar{\beta}^{n},
\end{equation}
\begin{equation}\label{scm:cnh1:k2}
	\bar{G}(C_{h}^{n})=0,
\end{equation}
\begin{equation}\label{scm:cnh1:k3}
	 0 \leq C_{h}^{n} \leq 1, \quad \lambda_h^{n}C_{h}^{n} =0,\quad \xi_h^{n}\,(1-C_{h}^{n})=0, \quad \lambda_h^{n} \geq 0, \quad \xi_h^{n} \geq 0,
\end{equation}
where $\bm{H}_{h}:=\bar{\beta}^{n} \bm{I}_{h} - \nabla_{h} \cdot \bm{\mal{D}}(\widetilde{\bm{U}}_{h}^{n-1/2}) \nabla_{h}$, $\bm{I}_{h}$ is the identity operator, and $\lambda_h^{n}$, $\xi_h^{n} \in \mal{C}$ and $ \gamma^{n} \in \mathbb{R}$ are the corresponding Lagrange multipliers as the first-order $L^{2}$-projection scheme.
Let
$
V_h := C_{h}^{n} - \lambda_{h}^{n} + \xi_{h}^{n}.
$
Then, from the KKT condition \eqref{scm:cnh1:k3}, we obtain
$$
C_{h}^{n} = \mbb{P} V_{h} := \max\{0, \min\{1,V_{h}\}\},\quad 
\lambda_{h}^{n} = \max\{0,-V_{h}\}, \quad 
\xi_{h}^{n} = \max\{0,V_{h}-1\}.
$$
Consequently, the coupled system \eqref{scm:cnh1:k1}–\eqref{scm:cnh1:k2} can be equivalently reformulated as the following nonlinear system:
\begin{equation}\label{F:cn}
    \begin{aligned}
        F_{1}( V_{h}, \gamma^{n} ) 
        & =  \big( \bm{H}_{h}  - \bm{I}_{h} \big)  C_{h}^{n}  + V_{h} + \gamma^{n} \bar{\beta}^{n} - \bm{H}_{h} \widehat{C}_{h}^{n}\\
        & = \big( \bm{H}_{h}  - \bm{I}_{h} \big)  \mbb{P} V_{h}  + V_{h} + \gamma^{n} \bar{\beta}^{n} - \bm{H}_{h} \widehat{C}_{h}^{n} = 0,\\
        F_{2}( V_{h}, \gamma^{n} ) 
        & = \big( \bar{\beta}^{n} C_{h}^{n} , 1 \big)_{\rm M} - \big( \bar{\beta}^{n} \widehat{C}_{h}^{n}, 1 \big)_{\rm M} \\
        & = \big( \bar{\beta}^{n} \mbb{P} V_{h} , 1 \big)_{\rm M} - \big( \bar{\beta}^{n} \widehat{C}_{h}^{n}, 1 \big)_{\rm M} = 0.
    \end{aligned}
\end{equation}
 Analogous to Ref. \cite{CYY24}, the semismooth Newton method can be employed to solve the nonlinear system \eqref{F:cn}:
Given iterate $V_{h}\in \mal{C}$ and $\gamma^{n} \in \mathbb{R}$, the generalized Jacobian $\bm{J}: \mal{C} \times \mathbb{R} \rightarrow \mal{C} \times \mathbb{R}$, acting on an increment $(\delta V, \delta \gamma)^{\top}$, admits the representation
$$
\bm{J}\left(
\begin{aligned}
    \delta V\\
    \delta \gamma
\end{aligned}\right) = \left(
\begin{aligned}
    & \big( \bm{H}_{h}  - \bm{I}_{h} \big)  \bm{1}_{V_{h} \in [0,1]} \delta V + \delta V + \bar{\beta}^{n} \delta \gamma\\
    & h^{x}h^{y} \sum_{i=1}^{N_{x}}\sum_{j=1}^{N_{y}} \bar{\beta}_{i,j}^{n} \bm{1}_{V_{i,j} \in [0,1]} [\delta V]_{i,j} 
\end{aligned}\right),
$$
where $\bm{1}_{V_{h} \in [0,1]}$ denotes the indicator function such that $\bm{1}_{V_{h} \in [0,1]} =1$ if $V_{h} \in [0,1]$ and $0$ otherwise. Set the initial iterates as $V_{h,0}=\widehat{C}_{h}^{n}$ and $\gamma_{0}=0$. For \(\ell = 0, 1, \dots\), the semismooth Newton iterative proceeds as follows:

\medskip
\textbf{Step 1:} (Residual evaluation). Compute the nonlinear residuals
$$
\begin{aligned}
    &F_{1,\ell}=\big( \bm{H}_{h}  - \bm{I}_{h} \big)  \mbb{P} V_{h,\ell}  + V_{h,\ell} + \gamma_{\ell} \bar{\beta}^{n} - \bm{H}_{h} \widehat{C}_{h}^{n},\\
    &F_{2,\ell}=\big( \bar{\beta}^{n} \mbb{P} V_{h,\ell} , 1 \big)_{\rm M} - \big( \bar{\beta}^{n} \widehat{C}_{h}^{n}, 1 \big)_{\rm M}.
\end{aligned}
$$

\textbf{Step 2:} (Increment solution).  Solve the two linear systems
$$
\big( I_{h} + \big( \bm{H}_{h}  - \bm{I}_{h} \big)  \bm{1}_{V_{h,\ell} \in [0,1]} \big) \delta V_{1} = - F_{1,\ell},\quad
\big( I_{h} + \big( \bm{H}_{h}  - \bm{I}_{h} \big)  \bm{1}_{V_{h,\ell} \in [0,1]} \big) \delta V_{2} = \bar{\beta}^{n}.
$$
Then compute the increments
$$
\delta \gamma = \frac{F_{2,\ell} + (\bar{\beta}^{n} \bm{1}_{V_{h,\ell} \in [0,1]} \delta V_1,1)}{(\bar{\beta}^{n} \bm{1}_{V_{h,\ell} \in [0,1]} \delta V_2,1)},\quad
\delta V = \delta V_1 -  \delta \gamma \delta V_2.
$$

\textbf{Step 3:} (Update). Update the multiplier and the primary variable by
$$
\gamma_{\ell+1} = \gamma_{\ell} + \delta \gamma,\quad
V_{h,\ell+1} =  V_{h,\ell} + \delta V.
$$

\textbf{Step 4:} (Termination check). If the residual norm $\|V_{h,\ell+1} - V_{h,\ell}\|_\infty$ falls below a prescribed tolerance, then terminate the iteration, accept $(V_{h,\ell+1}, \gamma_{\ell+1})$ as the approximate solution, and set the concentration and Lagrange multipliers as follows:
$$
C_h^n = \max\{0, \min\{ V_{h,\ell+1},1\}\}, \quad
\lambda_h^n = \max\{0,-V_{h,\ell+1}\}, \quad
\xi_h^n = \max\{0,V_{h,\ell+1}-1\},\quad
\gamma^n = \gamma_{\ell+1}.
$$
Otherwise, set $\ell \leftarrow \ell+1$ and return to Step 1.
\bibliographystyle{elsarticle-num}
\bibliography{ref_bp_mc.bib}

@book{KKT08,
	author = {Ito, Kazufumi and Kunisch, Karl},
	title = {Lagrange Multiplier Approach to Variational Problems and Applications},
	publisher = {SIAM},
	year = {2008},
	address = {Philadelphia}
}

@book{Ewing83,
	author = {Ewing, Richard E.},editor = {Richard E. Ewing},
	title = {The Mathematics of Reservoir Simulation},
	publisher = {SIAM},
	year = {1983},
	address = {Philadelphia},
	edition   = {}
}

@book{Bear72,
    author={Bear, Jacob},
   title={Dynamics of Fluids in Porous Media},
    publisher = {Elsevier},
    year={1972},
    address = {New York},
    edition   = {}
	
}

@book{Peaceman77,
  author = {Peaceman, D. W.},
  title = {Fundamentals of Numerical Reservoir Simulation},
  publisher = {Elsevier},
  address = {Amsterdam},
  year = {1977}
}

@article{Lazarov96-SIAMNA,
author = {Lazarov, R. D. and Mishev, Ilya D. and Vassilevski, P. S.},
title = {Finite Volume Methods for Convection-Diffusion Problems},
journal = {SIAM J. Numer. Anal.},
volume = {33},
pages = {31-55},
year = {1996}
}

@article{LD17-JCAM,
title = {The mass-preserving and modified-upwind splitting {DDM} scheme for time-dependent convection-diffusion equations},
journal = {J. Comput. Appl. Math.},
volume = {317},
pages = {247-273},
year = {2017},
author = {Zhongguo Zhou and Dong Liang}
}

@article{LD20-JCP,
title = {The conservative splitting domain decomposition method for multicomponent contamination flows in porous media},
journal = {J. Comput. Phys.},
volume = {400},
pages = {108974},
year = {2020},
author = {Dong Liang and Zhongguo Zhou}
}

@article{viscosity63,
    author = {Koval, E.J.},
    title = {A Method for Predicting the Performance of Unstable Miscible Displacement in Heterogeneous Media},
    journal = {SPE J.},
    volume = {3},
    pages = {145-154},
    year = {1963}
}

@article{CCGM26,
title = {A cell centered {G}alerkin method for miscible displacement in heterogeneous porous media},
journal = {Comput. Methods Appl. Mech. Engrg.},
volume = {450},
pages = {118608},
year = {2026},
author = {Maurice S. Fabien}
}

@article{Douglas83cmd,
  title={Numerical Methods for a Model for Compressible Miscible Displacement in Porous Media},
  author={Douglas, J. and Roberts, J. E.},
  journal={Math. Comp.},
  volume={41},
  pages={441-459},
  year={1983},
}

@article{CYY26,
	title = {Positivity preserving and mass conservative projection methods for the {P}atlak-{K}eller-{S}egel equation},
	journal = {Appl. Numer. Math. },
	volume = {221},
	pages = {63-79},
	year = {2026},
	author = {Yongyong Cai and Fenghua Tong}
}

@article{CYY24,
	author = {Tong, Fenghua and Cai, Yongyong},
	title = {Positivity preserving and mass conservative projection method for the {P}oisson-{N}ernst-{P}lanck equation},
	journal = {SIAM J. Numer. Anal.},
	volume = {62},
	pages = {2004-2024},
	year = {2024}
}

@article{SQW10,
title = {On positivity-preserving high order discontinuous {G}alerkin schemes for compressible {E}uler equations on rectangular meshes},
journal = {J. Comput. Phys.},
volume = {229},
pages = {8918-8934},
year = {2010},
author = {Xiangxiong Zhang and Chi-Wang Shu}
}

@article{SJ92,
	author = {Shen, Jie},
	title = {On error estimates of projection methods for {N}avier-{S}tokes equations: {F}irst-order schemes},
	journal = {SIAM J. Numer. Anal.},
	volume = {29},
	pages = {57-77},
	year = {1992}
}

@article{CYY18,
	author = {Cai, Yongyong and Shen, Jie},
	year = {2018},
	pages = {2057-2090},
	title = {Error estimates for a fully discretized scheme to a {C}ahn-{H}illiard phase-field model for two-phase incompressible flows},
	volume = {87},
	journal = {Math. Comput.}
}

@article{YJ22,
	author = {Yang, Jiang and Yuan, Zhaoming and Zhou, Zhi},
	year = {2022},
	pages = {76},
	title = {Arbitrarily high-order maximum bound preserving schemes with cut-off postprocessing for {A}llen-{C}ahn equations},
	volume = {90},
	journal = {J. Sci. Comput.}	
}

@article{LBY20,
	author = {Li, Buyang and Yang, Jiang and Zhou, Zhi},
	title = {Arbitrarily high-order exponential cut-off methods for preserving maximum principle of parabolic equations},
	journal = {SIAM J. Sci. Comput.},
	volume = {42},
	pages = {A3957-A3978},
	year = {2020}
}

@article{CWB19,
title = {Positivity-preserving, energy stable numerical schemes for the {C}ahn-{H}illiard equation with logarithmic potential},
journal = {J. Comput. Phys. X},
volume = {3},
pages = {100031},
year = {2019},
author = {Wenbin Chen and Cheng Wang and Xiaoming Wang and Steven M. Wise}
}

@article{LJG18,
  title = {Positivity-preserving and asymptotic preserving method for 2{D} {K}eller-{S}egal equations},
  author = {Liu, Jianguo and Wang, Li and Zhou, Zhennan},
  journal = {Math. Comput.},
  volume = {87},
  pages = {1165--1189},
  year = {2018}
}

@article{HJW20,
  title = {A fully discrete positivity-preserving and energy-dissipative finite difference scheme for {P}oisson-{N}ernst-{P}lanck equations},
  author = {Hu, Jingwei and Huang, Xiaodong},
  journal = {Numer. Math.},
  volume = {145},
  pages = {77--115},
  year = {2020},
}

@article{HZ23,
  title = {Positivity-preserving and energy-dissipative finite difference schemes for the {F}okker-{P}lanck and {K}eller-{S}egel equations},
  author = {Jingwei Hu and Xiangxiong Zhang},
  journal = {IMA J. Numer. Anal.},
  volume = {43},
  pages = {1450-1484},
  year = {2023},
}

@article{SJ21,
author = {Huang, Fukeng and Shen, Jie},
title = {Bound/Positivity Preserving and Energy Stable Scalar auxiliary Variable Schemes for Dissipative Systems: Applications to {K}eller-{S}egel and {P}oisson-{N}ernst-{P}lanck Equations},
journal = {SIAM J. Sci. Comput.},
volume = {43},
pages = {A1832-A1857},
year = {2021}
}

@article{ZBY23,
	title = {A high-order two-grid difference method for nonlinear time-fractional biharmonic problems and its unconditional $\alpha$-robust error estimates},
	journal = {J. Sci. Comput.},
	volume = {96},
	pages = {54},
	year = {2023},
	author = {Fu, Hongfei and Zhang, Bingyin and Zheng, Xiangcheng}
}

@article{ZBY24,
author = {Zhang, Bingyin and Fu, Hongfei},
year = {2024},
pages = {421-455},
title = {An efficient two-grid high-order compact difference scheme with variable-step {BDF}2 method for the semilinear parabolic equation},
volume = {58},
journal = {ESAIM: M2AN},
}

@article{Russell85,
	author = {Russell, Thomas F.},
	title = {Time Stepping Along Characteristics with Incomplete Iteration for a {G}alerkin Approximation of Miscible Displacement in Porous Media},
	journal = {SIAM J. Numer. Anal.},
	volume = {22},
	pages = {970-1013},
	year = {1985}
}

@article{Rui13,
    author = {Hao Pan and Hongxing Rui},
	title = {A mixed element method for {D}arcy-{F}orchheimer incompressible miscible displacement problem},
	journal = {Comput. Methods Appl. Mech. Engrg.},
	volume = {264},
	pages = {1-11},
	year = {2013}
}

@article{WH00,
	author = {Wang, Hong and Liang, Dong and Ewing, Richard E. and Lyons, Stephen L. and Qin, Guan},
	title = {An Approximation to Miscible Fluid Flows in Porous Media with Point Sources and Sinks by an {E}ulerian-{L}agrangian Localized Adjoint Method and Mixed Finite Element Methods},
	journal = {SIAM J. Sci. Comput.},
	volume = {22},
	pages = {561-581},
	year = {2000}
}

@article{Arbogast97,
	author = {Arbogast, Todd and Wheeler, Mary F. and Yotov, Ivan},
	title = {Mixed Finite Elements for Elliptic Problems with Tensor Coefficients as Cell-Centered Finite Differences},
	journal = {SIAM J. Numer. Anal.},
	volume = {34},
	pages = {828-852},
	year = {1997}
}

@article{Weiser88,
	author = {Weiser, Alan and Wheeler, Mary Fanett},
	title = {On Convergence of Block-Centered Finite Differences for Elliptic Problems},
	journal = {SIAM J. Numer. Anal.},
	volume = {25},
	pages = {351-375},
	year = {1988}
}

@article{XJ22,
	author = {J. Xu and S. Xie and H. Fu},
	title = {A two-grid block-centered finite difference method for the nonlinear regularized long wave equation},
	journal = {Appl. Numer. Math.},
	volume = {171},
	pages = {128-148\color{black}},
	year = {2022}
}

@article{XJ25,
	author = {Jie Xu and Hongfei Fu},
	title = {A decoupled linear, mass-conservative block-centered finite difference method for the {K}eller-{S}egel Chemotaxis system},
	journal = {J. Comput. Phys.},
	volume = {526},
	pages = {113775},
	year = {2025}
}

@article{SYL21,
	author = {Shi, Yilei and Xie, Shusen and Liang, Dong and Fu, Kai},
	title = {High Order Compact Block-Centered Finite Difference Schemes for Elliptic and Parabolic Problems},
	journal = {J. Sci. Comput.},
	volume = {87},
	pages = {86},
	year = {2021}
}

@article{BH,
	author = { Berikelashvili, Givi and  Gupta, Murli M. and  Mirianashvili, Manana},
	title = {Convergence of Fourth Order Compact Difference Schemes for Three-Dimensional Convection-Diffusion Equations},
	journal = {SIAM J. Numer. Anal.},
	volume = {45},
	pages = {443-455},
	year = {2007}
}

@article{CD07,
	author = {Chainais-Hillairet, Claire and Droniou, J\'{e}r\^{o}me},
	title = {Convergence Analysis of a Mixed Finite Volume Scheme for an Elliptic-Parabolic System Modeling Miscible Fluid Flows in Porous Media},
	journal = {SIAM J. Numer. Anal.},
	volume = {45},
	pages = {2228-2258},
	year = {2007}
}

@article{VPV21,
    author = {L. {Beir$\tilde{a}$o da Veiga} and A. Pichler and G. Vacca},
	title = {A virtual element method for the miscible displacement of incompressible fluids in porous media},
	journal = {Comput. Methods Appl. Mech. Engrg.},
	volume = {375},
	pages = {113649},
	year = {2021}
}

@article{Eymard19,
	author = {Droniou, J$\'{e}$r$\hat{o}$me and  Eymard, Robert and Talbot, Kyle S.},
	title = {Unified Convergence Analysis of Numerical Schemes for a Miscible Displacement Problem},
	journal = {Found. Comput. Math.},
	volume = {19},
	pages = {333-374},
	year = {2019}
}

@book{TNA09,
    author={Weimin Han and Kendall E. Atkinson},
	title={Theoretical Numerical Analysis: A Functional Analysis Framework},
    publisher = {Springer},
    year={2009},
    address = {New York},
    edition   = {}	
}

@article{ZXX18,
author = {Li, Hao and Xie, Shusen and Zhang, Xiangxiong},
title = {A High Order Accurate Bound-Preserving Compact Finite Difference Scheme for Scalar Convection Diffusion Equations},
journal = {SIAM J. Numer. Anal.},
volume = {56},
pages = {3308-3345},
year = {2018}
}

@article{XY19,
author = {Van der Vegt, J. J. W. and Xia, Yinhua and Xu, Yan},
title = {Positivity Preserving Limiters for Time-Implicit Higher Order Accurate Discontinuous {G}alerkin Discretizations},
journal = {SIAM J. Sci. Comput.},
volume = {41},
pages = {A2037-A2063},
year = {2019}
}

@article{SJ22-CMAME,
author = {Qing Cheng and Jie Shen},
title = {A new {L}agrange multiplier approach for constructing structure preserving schemes, {I}. {P}ositivity preserving},
journal = {Comput. Methods Appl. Mech. Engrg.},
volume = {391},
pages = {114585},
year = {2022}
}

@article{SJ22-SIAM,
author = {Cheng, Qing and Shen, Jie},
title = {A New {L}agrange Multiplier Approach for Constructing Structure Preserving Schemes, {II}. {B}ound Preserving},
journal = {SIAM J. Numer. Anal.},
volume = {60},
pages = {970-998},
year = {2022}
}

@article{GH17,
author = {Guo, Hui and Yang, Yang},
title = {Bound-Preserving Discontinuous {G}alerkin Method for Compressible Miscible Displacement in Porous Media},
journal = {SIAM J. Sci. Comput.},
volume = {39},
pages = {A1969-A1990},
year = {2017}
}

@article{GH22,
title = {Bound-preserving discontinuous {G}alerkin methods with second-order implicit pressure explicit concentration time marching for compressible miscible displacements in porous media},
journal = {J. Comput. Phys.},
volume = {463},
pages = {111240},
year = {2022},
author = {Wenjing Feng and Hui Guo and Yue Kang and Yang Yang}
}

@article{GH23,
title = {Sign-preserving second-order {IMPEC} time discretization and its application in compressible miscible displacement with {D}arcy-{F}orchheimer models},
journal = {J. Comput. Phys.},
volume = {474},
pages = {111775},
year = {2023},
author = {Wenjing Feng and Hui Guo and Lulu Tian and Yang Yang}
}

@article{ZZF26,
title = {A linear, decoupled and positivity-preserving time-staggered block-centered finite difference method for the multi-species {Keller-Segel} chemotaxis system},
journal = {Arxiv},
volume = { },
pages = {arXiv:2607.00713},
year = {2026},
author = {Ao Zhang and Bingyin Zhang and Hongfei Fu}
}

@article{WRF26,
title = {A linear, mass-conserving, multi-time-step compact block-centered finite difference method for incompressible miscible displacement problem in porous media},
journal = {J. Comput. Appl. Math.},
volume = {485},
pages = {117603},
year = {2026},
author = {Xiaoying Wang and Hongxing Rui and Hongfei Fu}
}

@article{WLF24,
  title   = {Optimal error estimate of unconditionally positivity-preserving, mass-conserving and energy stable method for the {K}eller-{S}egel chemotaxis model},
  journal = {Math. Comp.},
  volume  = {94},
  pages   = {2761--2793},
  year    = {2025},
  author  = {K. Wang and E. Liu and X. Feng},
}

@article{LXL25,
title = {A class of high-order physics-preserving schemes for thermodynamically consistent model of incompressible and immiscible two-phase flow in porous media},
journal = {J. Comput. Phys.},
volume = {529},
pages = {113864},
year = {2025},
author = {Xiaoli Li and Yujing Yan and Huangxin Chen}
}

@article{SSY24,
title = {An energy stable and positivity-preserving computational method for compressible and immiscible two-phase flow in porous media},
journal = {J. Computat. Phys.},
volume = {519},
pages = {113391},
year = {2024},
author = {Jisheng Kou and Huangxin Chen and Amgad Salama and Shuyu Sun}
}

@article{SSY22,
title = {An energy stable linear numerical method for thermodynamically consistent modeling of two-phase incompressible flow in porous media},
journal = {J. Computat. Phys.},
volume = {451},
pages = {110854},
year = {2022},
author = {Jisheng Kou and Xiuhua Wang and Shigui Du and Shuyu Sun}
}

@article{SX,
title = {Unconditionally positivity preserving and energy dissipative schemes for {P}oisson-{N}ernst-{P}lanck equations},
journal = {Numer. Math.},
volume = {148},
pages = {671-697},
year = {2021},
author = {Jie Shen and Jie Xu}
}

\end{document}